\documentclass[preprint,12pt]{elsarticle}
\usepackage[utf8]{inputenc} % for writing other that basic characters
\usepackage{graphicx}
\usepackage{caption}
\usepackage{subcaption}
\usepackage{etoolbox}
\usepackage{hhline}
\usepackage{siunitx}
\usepackage{verbatim}  % Needed for the "comment" environment to make LaTeX comments
\usepackage{amsmath}
\usepackage{algpseudocode}
\usepackage{mathtools}
\usepackage{amssymb}
\usepackage{accents}
\usepackage{etoolbox}
\allowdisplaybreaks[1]
\usepackage{stmaryrd}
\usepackage{bm}

\usepackage{tikz,pgf}
\usetikzlibrary{shapes,arrows,graphs,calc,scopes,decorations.markings,decorations.pathmorphing,positioning,shapes.arrows}
\newcommand{\ubar}[1]{\underaccent{\bar}{#1}}

\usepackage{macros}

\begin{document}
\begin{frontmatter}
\title{An posteriori error estimator \\ for discontinuous Galerkin discretisations \\ of  convection-diffusion problems with application \\ to Earth's mantle convection simulations}
\author[leic]{Tiffany Barry}
\author[sissa]{Andrea Cangiani}
\author[hdruk]{Samuel P. Cox}
\author[ntua]{\mbox{Emmanuil H.~Georgoulis}}

\address[leic]{School of Geography Geology and the Environment, University of Leicester, University Road, Leicester LE1 7RH, United Kingdom. E:  tlb2@leicester.ac.uk }		
\address[sissa]{Mathematics Area, SISSA, International School for Advanced Studies, via Bonomea 265, I-34136 Trieste, Italy. E: andrea.cangiani@sissa.it.}		
\address[hdruk]{Health Data Research UK, London, United Kingdom. E: sam.cox@hdruk.ac.uk}
\address[ntua]{Department of Mathematics and The Maxwell Institute for Mathematical Sciences, Heriot-Watt University, Edinburgh EH14 4AS, United Kingdom {\sc AND} Department of Mathematics, School of Applied Mathematical and Physical Sciences, National Technical University of Athens, Zografou 15780, Greece, {\sc AND} IACM-FORTH, Greece. \mbox{E: {e.georgoulis@hw.ac.uk}}}

%\addtotoc{Abstract}  % Add the "Abstract" page entry to the Contents
\begin{abstract}
%    \addtocontents{toc}{\vspace{1em}}  % Add a gap in the Contents, for aesthetics
We present new \aposteriori{} error estimates for the interior penalty discontinuous Galerkin method applied to non-stationary convection-diffusion equations. The focus is on strongly convection-dominated problems without zeroth-order reaction terms, which leads to the absence of positive $L^2$-like components. An important specific example is the energy/temperature equation of the Boussinesq system arising from the modelling of mantle convection of the Earth. The key mathematical challenge of mitigating the effects of exponential factors with respect to the final time, arising from the use of Gr\"onwall-type arguments, is addressed by an exponential fitting technique. The latter results to a new class of \aposteriori{} error estimates for the stationary problem, which are valid in cases of convection and reaction coefficient combinations not covered by the existing literature. This new class of estimators is combined with an elliptic reconstruction technique to derive new respective estimates for the non-stationary problem, exhibiting reduced dependence on Gr\"onwall-type exponents and, thus, offer more accurate estimation for longer time intervals. We showcase the superior performance of the new class of \aposteriori{} error estimators in driving mesh adaptivity in Earth's mantle convection simulations, in a setting where the energy/temperature equation is discretised by the discontinuous Galerkin method, coupled with the Taylor-Hood finite element for the momentum and mass conservation equations.
We exploit the community code \aspect\, to present numerical examples showing the effectivity of the proposed approach.
\end{abstract}

\begin{keyword}
%% keywords here, in the form: keyword \sep keyword

%% MSC codes here, in the form: \MSC code \sep code
%% or \MSC[2008] code \sep code (2000 is the default)

Discontinuous Galerkin \sep non-stationary convection-diffusion \sep\aposteriori{} error estimation \sep adaptive finite element methods \sep Boussinesq system.

\end{keyword}
\end{frontmatter}

\section{Introduction}\label{sec:intro}

%The main benefits of interest exploited in the current work are: ease in mesh adaptivity, and the stability of the method when modelling convection-dominated problems.

It is well known that the standard, conforming finite element method (FEM) may suffer from  
spurious oscillations when solving convection-diffusion problems in the 
convection-dominated regime. This is typically treated with the addition of 
artificial diffusion \cite{vonneumann1950method}, or in a more refined fashion
with the addition of diffusion only in the direction of the streamlines \cite{hughes1979multidimensional,kelly1980note}.
Following this, the method known as streamline upwind Petrov-Galerkin (SUPG) 
\cite{hughes1982theoretical} enhanced the capability to solve 
convection-dominated problems with finite elements. Since then, numerous techniques have been proposed 
to stabilise FEM, such as artificial viscosity, entropy viscosity \cite{GuermondPasquettiPopov2011}, 
etc.
On the other hand, it is possible to define discontinuous Galerkin (dG) methods, with carefully chosen
``upwinded'' numerical fluxes, to localise or even alleviate possible oscillatory behaviour in the vicinity of sharp/boundary layers or shocks. 
Consequently, no additional stabilisation term is required on top of the natural stabilising effect 
embedded in the numerical fluxes. This makes discontinuous Galerkin methods particularly well-suited for solving strongly convection-dominated problems, such as those arising from the modelling of the temperature of the Earth's mantle where diffusion is negligible compared to convection effects. Moreover, the weak imposition of interelement continuity characterising discontinuous Galerkin methods seamlessly allows for the treatment of hanging nodes in the context of adaptive mesh refinement. It also enables the extension to meshes containing general polygons/polyhedra \cite{cangiani2014hp, CaDoGeHo16, CaDoGeHo17, CaDoGe21}, which is of 
particular benefit when considering problems on intricate or heterogeneous
domains. For further details of the use of discontinuous methods on general polygonal/polyhedral elements 
in an $hp$ setting, we refer to \cite{CaDoGeHo17} and the references therein.

Realistic Earth mantle convection simulations require vast computational resources to resolve the various scales appearing in the respective flows.  A nonexhaustive literature review of known approaches for mantle convection simulation is postponed to Section \ref{sec:Boussinesq}. The extremely hot Earth's core heats the mantle, creating circulation effects which, upon reaching the crust, contribute to the movement of tectonic plates.  These circulation effects are driven by sharp variations in temperature. The numerical treatment of the Earth's mantle flow problem is further complicated by the greatly
varying parameter values of the models, the existence of boundary and interior layers, the nonlinear dependencies, and the vastly differing scales upon which the constituent
processes are set. Therefore, dynamic mesh adaptivity is very attractive as a tool to reduce the overall computational cost without adversely damaging the local mesh resolution required to resolve the sharp variations in temperature, and thus helps to bring
larger problems within the reach of current computing abilities. 

Mesh adaptive strategies in finite element analysis are typically driven by \aposteriori{} error indicators/estimators. To ensure reliable error control, mathematically rigorous \aposteriori{} error \emph{bounds}, whereby the error is  bounded by computable quantities, have been developed in the numerical analysis literature for various classes of problems involving partial differential equations (PDEs). The mathematically rigorous \aposteriori{} error analysis of FEM and of dG methods is fairly mature: we refer to \cite{ainsworth2000} for an overview of standard results for FEM, and to
 \cite{Karakashian2003,HoustonSchotzauWihler2007} for the first results for dG methods discretising pure diffusion problems. 
The \aposteriori{} error analysis of stationary linear convection-diffusion equations discretised by stabilised FEM or dG methods for various settings can be found in 
\cite{verfurth1998posteriori,Kunert2003,Verfurth2005,Sangalli2008,Schotzau2009,ern2010guaranteed,zhu2011robust}.
\emph{A posteriori} error estimators of various kinds for conforming finite element methodologies discretising non-stationary convection-diffusion problems can be found in \cite{houston2001adaptive,berrone2004multilevel,araya2005adaptive,araya2005hierarchical,ern2005posteriori,verfurth2005robust,sun2006posteriori,georgoulis2007discontinuous} and other works. Respective results for discontinuous Galerkin methods are
less abundant \cite{Cangiani2013a}. 

However, to the best of our knowledge, current literature for both FEM and dG discretisations does \emph{not} cover the case of \aposteriori{} error bounds for general convection fields: available results require that, in the absence of (zeroth-order) reaction term, the convective field must admit non-positive divergence of the convection field to avoid the presence of  Gr\"onwall-type exponential components of the final time in the
resulting \aposteriori{} error bounds for standard norms. Unfortunately, such assumptions are hard or even impossible to be satisfied whenever the convection field is also simultaneously computed, e.g., from a non-exactly divergence-free approximation of incompressible flows,  or in cases whereby we do not \apriori{} know the behaviour of the flow. This is exactly the case for the Boussinesq system of equations, which is the mostly widely used basic mathematical model of the convective flow of the Earth's mantle.  Under some simplifying assumptions, the mantle dynamics is modelled by a system of coupled equations: a convection-dominated diffusion
equation for the  temperature combined with the Stokes system modelling the mantle velocity and pressure.  The complexity and nonlinearity, due to coupling, of these systems mean that \apriori~knowledge  of the flow characteristics is often extremely limited.

Aiming to harness the attractive properties of dG methodologies within an adaptive setting, we derive new \aposteriori{} error bounds for convection-dominated non-stationary convection-diffusion problems discretised by the interior-penalty discontinuous Galerkin method. The key technical developments include the use of, so-called, exponential fitting techniques, whereby the analysis is performed on exponentially weighted norms with carefully constructed weights for the respective stationary problem. The \aposteriori{} error analysis for the (parabolic) non-stationary problem then follows by employing the elliptic reconstruction framework \cite{Makridakis2003, LM06, M07, Georgoulis2011, Bansch2012, Cangiani2013a,GM23}. Crucially, the new \aposteriori{} error analysis remains valid for general convection fields in the absence of (zeroth order) reaction terms and, thus, it is directly applicable to the Boussineq system modelling mantle convection. The flexibility of the proposed approach allows for a mathematically rigorous \aposteriori{} error estimation that drives mesh adaptivity in the study of geodynamic flows, 
in particular mantle convection.

We test the new  \aposteriori{} error bounds for the interior penalty discontinuous Galerkin method for the temperature equation,
coupled to Taylor-Hood finite elements for the Stokes system in realistic mantle convection simulation scenarios. Specifically, 
we present an implementation of the dG
method in the community code \aspect~\cite{KronbichlerHeisterBangerth2012,Heiste2017,Bangerth2018}, along with an adaptivity indicator based on the proven error limits \aposteriori{}. We report a number of numerical examples
exploring the applicability of the approach in different circumstances, with the ultimate goal of
reducing the computational cost of large mantle convection simulations. 

The remainder of this work is organised as follows. In Section \ref{sec:method} we detail the model convection-diffusion problem, as well as  its discretisation by the interior penalty discontinuous Galerkin method. In Section \ref{sec:apost}, we discuss the new  \aposteriori{} error analysis for dG methods for convection-dominated stationary convection-diffusion problems, once we provide details on the generality of in terms of convection fields permitted. In Sections \ref{sec:apost_time} and \ref{sec:apost_fully}, we employ the elliptic reconstruction framework to prove \aposteriori{} error bounds for the semi-discrete and the fully-discrete schemes, respectively, admitting general convection fields, as well as a comparison with existing results from the literature in Section \ref{sect:relation_to_existing_results}, along with implementation details of the estimators. Section \ref{sec:numerics} contains an extensive series of numerical experiments testing the new estimators for a range of qualitatively different convection fields.  In Section \ref{sec:Boussinesq}, we present the detailed Boussinesq system modelling mantle convection, along with a (non-exhaustive) literature review of numerical approaches in mantle convection simulation. In Section \ref{sec:numerics_mantle} adaptive simulations for the full Boussinesq system modelling mantle convection.  Finally, in Section \ref{sec:conclusions}, we draw some conclusions.

\section{The discontinuous Galerkin method for a model convection-diffusion problem}
\label{sec:method}

We introduce a non-stationary convection-diffusion model problem and its discretisation by the interior penalty discontinuous Galerkin method.

To simplify notation, we abbreviate the $L^2(\omega)$-inner product and $L^2(\omega)$-norm for a Lebesgue-measurable subset $\omega\subset \mathbb{R}^d$ as $(\cdot,\cdot)_{\omega}$ and $\norm{\cdot}{\omega}$, respectively.
Moreover, when $\omega=\domain$, with  $\domain \subset \reals^\spacedim$, $\spacedim\in\{2,3\}$, denoting the computational domain of the problem below, we will further compress the notation to $(\cdot,\cdot)\equiv (\cdot,\cdot)_{\Omega}$.
The standard notation $W^{k,p}(\omega)$ for Sobolev spaces, $k\in\mathbb{R}$, $p\in[1,\infty]$ will be used; when $p=2$, we set $H^k(\omega):=W^{k,2}(\omega)$.
In addition, given an interval $J\subset \mathbb{R}$ and a Banach space $V$, we use the standard notation for Bochner spaces $W^{k,p}(J; V)$, $p\in[1,\infty]$, with corresponding norms. 

Throughout this work the symbol ``$X\lesssim Y$'' means ``$X\le C Y$'' for a constant $C>0$ which is independent of other quantities appearing in the inequality.   
%ADD TEXT %Although the application we have in mind...

\subsection{Model problem}
\label{sec:model}

Let $\domain \subset \reals^\spacedim$, $\spacedim\in\{2,3\}$, be an open, bounded 
domain that either has smooth boundaries, or is convex and polytopic, i.e., polygonal for $d=2$ or polyhedral for $d=3$. 
We denote its closure by $\closure{\Omega}$, its boundary by $\domainbdy$, and by $\normal(\bx)$ the outward normal from the 
boundary at a.e. point $\bx\in\domainbdy$. 
The boundary is split into two disjoint subsets 
$\dirbdy$ and $\neubdy$, whence $\domainbdy=\dirbdy \union\neubdy$ and  $\dirbdy \intersect \neubdy = \emptyset$. 
Further, we let $\timeinterval = \left[ \starttime, \finaltime \right] \subset \reals$, 
$\finaltime>0$, be a time interval. 

Given a convection field
$\conv(\bx,\timevar) \equiv \conv
= 
(\convelem[1],\ldots,\convelem[\spacedim])^\transpose \in 
\left[\extendedCspace{\starttime}{\finaltime}{\sobolevspace{1}{\infty}}\right]^\spacedim$
and, hence, 
%whose elemen\u\test $\convelem[i], i=1,\ldots,\spacedim$ are in $\sobolevspace{1}{\infty}$, 
$\div\conv\in \extendedlinftyspace{\starttime}{\finaltime}{\linftyspace}$,  such that 
$\conv(\bx,\timevar) \cdot \normal(\bx) = 0$ for $(\bx,\timevar)$ in 
$\neubdy\times\timeinterval$,
we consider the convection-diffusion initial-boundary value problem:
\begin{align}
\timederiv{\u} - \diffusivity\laplacian\u + \conv(\bx,\timevar)\cdot\grad\u 
&= \heating(\bx,\timevar) &&\text{ on } \domain \times \timeinterval, \label{eqn:convdiff}\\
\u 
&= \dirdata(\bx,\timevar) &&\text{ on } \dirbdy \times \timeinterval, \label{eqn:convdiffdirbdy}\\
\diffusivity\deriv{\u}{\normal} 
&= \neudata(\bx,\timevar) 
&&\text{ on } \neubdy \times \timeinterval, \label{eqn:convdiffneubdy}\\
\u(\bx,\starttime)
&=\initu(\bx) &&\text{ on } \domain.\label{eqn:convdiffinit}
\end{align}
Here, $\diffusivity$ is a, typically small, positive constant, ($0 < \diffusivity \ll 1$,)
%$\conv(\bx) = (\convelem[1],\convelem[2])^\transpose$ (or $= (\convelem[1],\convelem[2],\convelem[3])^\transpose$ in 3 dimensions) 
$\heating \in \extendedltwospace{\starttime}{\finaltime}{\ltwospace}$, 
and $\initu \in \ltwospace$, and % \highlight{/ the same assumptions on $\diffusivity$, 
%$\heating(\bx,\timevar)$, and $\initu(\bx)$ as in Section~\ref{sect:boussinesq}.} 
 $\dirdata\in\extendedHonespace{\starttime}{\finaltime}{\Hpspace{\domainbdy}{\half}}$.

Upon introducing the bilinear form $a:\Hone\times\Hone\rightarrow\reals$ by
\begin{align*}%\label{eqn:defnbilineara}
%\bilineara{\u}{\test} = \int_\domain \diffusivity\nabla\u \cdot \nabla\test + \conv(\bx,\timevar)\cdot\grad\u \test \dx
\bilineara{w}{\test}: = 
\left(\diffusivity\nabla\w , \nabla\test \right) 
+ \left( \conv\cdot\grad\w , \test\right)\qquad\forall \w,\test\in\Hone,
\end{align*}
where, for brevity, we omit the dependence on time through $\conv$ and, similarly, the linear functional $l:\Hone\rightarrow\reals$ by
\begin{align*}%\label{eqn:defnlinearl}
\linearl{\test} 
= 
\int_\domain \heating \test \dx + \int_\neubdy \neudata\test \ds \qquad\forall \test\in\Hone,
\end{align*}
the \emph{weak formulation} of the problem 
\eqref{eqn:convdiff}-\eqref{eqn:convdiffinit} reads: fix  $\u(0) = \initu$ and 
for each $\timevar \in \timeinterval$, find $\u(\timevar) \in \Hone$ such that $\u\evalat{\dirbdy}= \dirdata$ and 
\begin{equation}
%\left\{
%\begin{array}{rl}
\left(\dtimederiv{\u}(t), \test\right) + \bilineara{\u(t)}{\test} = \linearl{\test},
%\int_\domain \timederiv{\u} \test + \diffusivity\nabla\u \cdot \nabla\test + \conv(\bx,\timevar)\cdot\grad\u \test + \reac(\bx,\timevar)\u \test \dx
%= \int_\domain \heating(\bx) \test \dx + \int_\neubdy \neudata\test \ds \\
%\u\evalat{\dirbdy}&= \dirdata, \\
%\u(\bx,0) &= \initu(\bx),
%\end{array}
%\right.
\label{eqn:convdiffweak}
\end{equation}
for all $\test\in\Honenonhomzero:=\{\test\in \Hone: \test|_\dirbdy=0 \}$.

The existence and uniqueness of a solution to the problem \eqref{eqn:convdiffweak} and, equivalently the existence and uniqueness of a weak solution to
\eqref{eqn:convdiff}--\eqref{eqn:convdiffinit}, is given by standard energy arguments for sufficiently smooth $\conv$. A particular result, which is of interest in the context of mantle convection application below 
on a annular domain of interest
is shown in \cite[Lemma 2]{Tabata2002}.
\begin{lemma}[Well-posedness; {\cite[Lemma 2]{Tabata2002}}]\label{lemma:convdiffwellposed}
Let $\domain = \{\bx\in\domain \suchthat R_1 < \abs{\bx} < R_2\}$ and 
suppose that $\heating\in\extendedltwospace{\starttime}{\finaltime}{\Hpspace{\domain}{-1}}$, 
$\dirdata\in\extendedHonespace{\starttime}{\finaltime}{\Hpspace{\domainbdy}{\half}}$, 
$\conv\in\extendedltwospace{\starttime}{\finaltime}{{\left[\lpspace{3}\right]}^3}$,
$\div\conv\in \extendedltwospace{\starttime}{\finaltime}{\lpspace{3}}$, 
and $\initu\in\ltwospace$.
Then there exists a unique solution 
$\u\in\extendedltwospace{\starttime}{\finaltime}{\Hone}\intersect\extendedlinftyspace{\starttime}{\finaltime}{\ltwospace}$ 
to  \eqref{eqn:convdiffweak}.% \hl{Note here we have no Neumann boundary!}
\end{lemma}

\subsection{Discontinuous Galerkin semi-discretisation in space}
\label{sec:dg}

We begin by introducing some notation, so that we can define the discontinuous Galerkin discretisation in space of problem~\eqref{eqn:convdiffweak}.

Consider a shape-regular family of simplicial or box-type (quadtrilateral/hexahedral) meshes $\{\tria\}_h$. Each mesh $\tria$ is a collection of open and disjoint simplicial or box-type cells 
$\cell$ that subdivide the domain $\domain$, hence  $\bigcup_{\cell\in\tria} \closure{\cell} = \closure{\domain}$, and 
$\cell_i \intersect \cell_j = \emptyset$ for all pairs of cells 
$\cell_i,\cell_j\in\tria$, $i\neq j$. 

 For each $\cell\in \tria$, we denote the boundary of the cell by 
$\cellbdy \coloneqq \closure{\cell}\backslash\cell$.
For each pair of cells $\cell, \neighborcell\in\tria$, we say the cells are vertex-neighbours 
if $\closure{\cell}\intersect\closure{\neighborcell}\neq\emptyset$, and define their 
interface to be a face.
We denote by $\edges$ the collection of all $(\spacedim-1)$-dimensional 
faces $\edge$ defined by the interfaces between cells. 
We also define the set of interior faces $\internaledges$ and set of faces on the 
boundary $\boundaryedges$. 
Thus, we have $\edges = \internaledges\union\boundaryedges$.
We define the boundary of the domain as
$\domainbdy = \bigcup_{\edge\in\boundaryedges}\edge$.
We also subdivide $\boundaryedges$ into faces on the Dirichlet boundary 
$\dirichletedges$ and faces on the Neumann boundary $\neumannedges$, with
$\dirichletedges\union\neumannedges=\boundaryedges$ and 
$\dirichletedges\intersect\neumannedges=\emptyset$.
We denote by $\edgediam$ the $(\spacedim-1)$-dimensional measure of the face $\edge$,
and by $\celldiam$ the $\spacedim$-dimensional measure of the cell $\cell$. Due to assumed shape-regularity, there exists a constant $c_{\rm sh}\ge1$ such that $\celldiam \le c_{\rm sh}\edgediam$ for all $K\in\tria$ and $\edge\in\edges$.

We assume that each cell $\cell\in \tria$ is constructed via an affine mapping
$\mapping{\cell} : \refcell \rightarrow \cell$  with non-singular Jacobian where 
$\refcell$ is the reference simplex  or the reference hypercube.
And thus define the discontinuous Galerkin finite element space of piecewise-polynomial functions $\dgspace$, in the following way: 
\begin{equation}
V_h \equiv \dgspace[k](\tria)\coloneqq \left\{ \testh\in\ltwospace \suchthat \testh\evalat{\cell}\circ\mapping{\cell}\in\pspace{k}(\refcell) ~\forall~ \cell \in \tria \right\},\label{eqn:simplicial_polys}
\end{equation}
depending on polynomial degree $k\in\naturals$ and with
$\pspace{k}(\refcell)$ is the space of polynomials of total degree $k$ if $\refcell$ is a simplex or  the space of polynomials 
of degree $k$ in each variable if $\refcell$ is hypercube. Throughout this work, we will denote by $\ltwoproj[k]:L^2(\Omega)\rightarrow \dgspace[k](\tria)$ the orthogonal $L^2$-projection, defined by
\[
(v-\ltwoproj[k]v,w_h)=0 \qquad \forall v\in L^2(\Omega)\ \&\  \forall w_h\in \dgspace[k](\tria).
\] 

\begin{remark}
The above is the standard choice of discontinuous spaces. We note here in passing that it is equally possible to apply the space $\pspace{k}$ 
to the case of quadrilateral and hexahedral meshes. This has the added benefit of 
reducing the number of degrees of freedom per cell, and has been shown \cite{cangiani2014hp,CaDoGeHo16, CaDoGeHo17}
to exhibit the same order of convergence as $\qspace{k}$. Furthermore, variable polynomial degrees can also be easily accommodated. 
\end{remark}

Further, we introduce the notation $\u[+]_{\cell}$ for the internal 
trace of $\u$, for a given cell $\cell$, and $\u[-]_{\cell}$
the external trace. Each internal face $\edge\in\internaledges$ 
(the set of internal faces) has two neighboring cells, $\cell$ and 
$\neighborcell$, with outward normals $\normal[\cell], \normal[\neighborcell]$ 
on the face $\edge$. Then the jumps over $\edge$ for a scalar-valued 
function $\w$ and vector-valued function $\bfw$ are defined as 
\begin{align*}
\jump{\w}_\edge \coloneqq \w[+]_{\cell}\normal[\cell] + \w[+]_{\neighborcell}\normal[\neighborcell], \qquad \jump{\bfw}_\edge 
	\coloneqq 
	\bfw[+]_{\cell}\cdot\normal[\cell] + \bfw[+]_{\neighborcell}\cdot\normal[\neighborcell].
\end{align*}
For faces on the Dirichlet portion of the boundary, we set
\begin{align*}
\jump{\w}_\edge 
	\coloneqq 
	\w[+]_{\cell}\normal[\cell], 
\qquad \jump{\bfw}_\edge 
	\coloneqq 
	\bfw[+]_{\cell}\cdot\normal[\cell],
\end{align*}
while on the Neumann portion we set
\begin{align*}
\jump{\w}_\edge 
	\coloneqq 
	\mathbf{0}, 
\qquad \jump{\bfw}_\edge 
	\coloneqq 
	0.
\end{align*}
In the same way, we define the average values of $\w$ and $\bfw$ 
on the face $\edge \subset \cellbdy$ as 
\begin{align*}
\average{\w}_{\edge}
	\coloneqq
	\half\left(\w[+]_\cell+\w[-]_\cell\right), \qquad
\average{\bfw}_{\edge}
	\coloneqq
	\half\left(\bfw[+]_\cell+\bfw[-]_\cell\right),
\end{align*}
while on all boundary faces we define
\begin{align*}
\average{\w}_{\edge}
	\coloneqq
	\w[+]_\cell, \qquad 
\average{\bfw}_{\edge}
	\coloneqq
	\bfw[+]_\cell.
\end{align*}
Finally we introduce the upwind-jump  across the boundary of $\cell$ given
by 
\begin{equation*}
\flowjump{\u}_{\cell}
\coloneqq
\left\{\begin{array}{cr}
\u[+]_\cell-\u[-]_\cell & \text{ on } \cellinflow\backslash\domainbdy,\\
\u[-]_\cell-\u[+]_\cell & \text{ on } \celloutflow\backslash\domainbdy.
\end{array}\right. 
\end{equation*}
Below, we often suppress the jump and average subscript when no confusion is likely.

With such notation at hand, we define for each $t\in I$ the interior penalty dG bilinear form $\bilinearah{\cdot}{\cdot}:\dgspace \times\dgspace \rightarrow\reals$ by
\begin{eqnarray*}
\bilinearah{\u}{\test}\hspace{-3mm}
&\coloneqq &\hspace{-3mm}
\sumovercells \cellinnerprod{\diffusivity\grad\u}{\grad\test} 
		+ \cellinnerprod{\conv\cdot\grad \u}{\test} \\
&+& \hspace{-3mm}\sumoverinternalanddirichletedges \left(-\edgeinnerprod{ \diffusivity\average{\grad\ltwoproj[k]\u}}{\jump{\test}}{\edge} 
		-  \edgeinnerprod{\diffusivity\average{\grad \ltwoproj[k]\test}}{\jump{\u}}{\edge} 
		+ \frac{\penal\diffusivity}{\edgediam}\edgeinnerprod{\jump{\u}}{\jump{\test}}{\edge}\right)\\
&-&\hspace{-3mm}\sumovercells  \left(\edgeinnerprod{(\conv\cdot\edgenormal)\u[+]}{\test[+]}{\cellinflow\intersect\dirbdy} 
		+ \edgeinnerprod{(\conv\cdot\edgenormal[\cell])\flowjump{\u}}{\test[+]}{\cellinflow\backslash\dirbdy}\right),
\end{eqnarray*}
noting the hidden dependence on $t$ through the coefficient $\conv$. Note that we use the inconsistent formulation obtained by inserting the $L^2$-projection inside the flux average terms. This is equivalent to the standard formulation over $\dgspace$ but has the advantage of allowing testing also in the space $\Hone$.

 Similarly, we introduce 
the linear functional $\linearlh{\cdot}:\dgspace \rightarrow\reals$ by 
\begin{eqnarray*}
\linearlh{\test} 
&\coloneqq &
\displaystyle
\innerprod{\heating}{\test} 
		+ \edgeinnerprod{\neudata}{\test}{\neubdy} 
		-\edgeinnerprod{\diffusivity\grad \test\cdot\bdynormal}{\dirdata}{\dirbdy} 
		+ \frac{\penal\diffusivity}{\edgediam}\edgeinnerprod{\dirdata}{\test}{\dirbdy} \\
&&\displaystyle
- \sumovercells \edgeinnerprod{(\conv\cdot\edgenormal)\dirdata}{\test[+]}{\cellinflow\intersect\dirbdy}.
\end{eqnarray*}
which depends on time also through $\heating$.
The spatially discrete interior penalty dG method, thus, reads: find $\uh \in C^{0,1}([0,T];V_h)$, such that, for each $t\in (0,T]$, we have 
\begin{align}\label{eqn:convdiffsemi-discrete}
\left({\timederiv{\uh}},{\testh}\right)
		+ \bilinearah{\uh}{\testh} 
= 
\linearlh{\testh}
\end{align}
for all $\testh \in \dgspace$, and $\uh(\starttime) = \ltwoproj[k]\initu$.

\subsection{Fully discrete implicit Euler-interior penalty dG method}
%time discretisation
We further discretise the problem in time by considering a discrete time-stepping and applying any finite difference method.
Here for simplicity we consider the first order implicit Euler time-stepping. To this end, 
let $\ntimesteps\in\naturals$ and let 
$\timevar[0]=\starttime, \timevar[1], \timevar[2], \ldots, \timevar[\ntimesteps]=\finaltime$ 
be a strictly increasing sequence of values in the interval 
$\timeinterval=(\starttime,\finaltime]$.
 We subdivide the time interval $\timeinterval$ into $\ntimesteps$ subintervals 
$\timeinterval_n$, $n\in\left\{1,\ldots,\ntimesteps\right\}$, with each subinterval 
defined by $\timeinterval_n \coloneqq (\timevar[n-1],\timevar[n]]$ and having timestep length $\timestep{n} := \timevar[n]-\timevar[n-1]$.

%mesh discretisation at each timestep
At each time interval $\timeinterval_n$, we define a triangulation $\tria[n]$ 
with the properties and notation given in the previous section, propagating the superscript notation to all mesh entities,
and introduce the corresponding discontinuous element-wise polynomial  spaces 
\[
\tdgspace[n]  := \dgspace[k](\tria[n]).
\]
The fully-discrete, implicit Euler-interior penalty dG method reads:
for $n=1,\ldots,\ntimesteps$, find $\uh[n] \in \tdgspace[n]$ such that 
\begin{align}\label{eqn:convdifffullydiscrete}
\left({\frac{\uh[n]-\uh[n-1]}{\timestep{n}}},{\testh}\right)
		+ \bilinearah{\uh[n]}{\testh} 
= 
\linearlh{\testh}, 
\end{align}
for all $\testh \in \tdgspace[n]$, with $\uh[0] = \ltwoproj[k]^m\initu$, where $\ltwoproj[k]^m$ indicates the $L^2$-projection with respect to the mesh $\tria[m]$, $m=0,\dots, \ntimesteps$.

\section{An \aposteriori{} bound for stationary problems}\label{sec:apost}

We first derive an  \aposteriori{} error bound for the stationary problem; then, using the elliptic 
reconstruction framework \cite{Makridakis2003, LM06, M07, Georgoulis2011, Bansch2012, Cangiani2013a,GM23}, we extend the analysis to the non-stationary problem.

Little previous work has been done on the \aposteriori{} analysis of the 
stationary convection-diffusion problem \emph{without} a reaction term, except where severe restrictions
are placed on the convection. Typically, the convection field is assumed to be exactly divergence-free or a sufficiently
large positive reaction term is assumed to ensure coercivity; see, for instance, ~\cite{Verfurth2005,zhu2011robust,Cangiani2013a} to mention just a few related works. In the presence of a non-negative reaction 
coefficient $\reac\in C(I,\linftyspace)$, the standard setting is indeed to assume that
\begin{equation}\label{eqn:coercivity_assumption}
-\half\div\conv(\bx,\timevar)+\reac(\bx,\timevar) \ge \gamma_0,
\end{equation}
for some constant $\gamma_0>0$, for almost all $\bx\in \domain$ and $t\in (0,T]$.

One approach to circumvent~\eqref{eqn:coercivity_assumption} is to employ a Gårding-type argument. Such an argument can be alternatively described as follows. We notionally add an artificial reaction term with
\emph{reaction coefficient} $\addreac_0$,
with $\addreac_0 > \half\div\conv$, so that we can satisfy \eqref{eqn:coercivity_assumption}
and, thus, reinstate coercivity. This can be unsatisfactory since, while we know $\div\conv\in\left[\linftyspace\right]^\spacedim$, we demand that $\addreac_0$ 
must be at least as large as $\half \div\conv$, and $\addreac_0$ ultimately leads to an exponential factor of the form ${\rm exp}(\delta_0 t_n)$ in the \aposteriori{} error bound for the non-stationary 
convection-diffusion problem, via a Gr\"onwall Lemma argument.

%Instead, since $\conv$ is approximately divergence-free, we instead assume the existence
%of a positive constant $\gamma$ such that 
%\begin{equation}
%\abs*{\div\conv} \leq \gamma,
%\end{equation}
%and 

An alternative approach, proposed in \cite{AyusoMarini2009,georgoulisSub}, is to use 
an exponential-fitting technique, testing against a modified test function to prove coercivity in a modified norm. However, this alone is not enough to guarantee coercivity in the modified norm in the absence of reaction, unless we assume $\div\conv\leq 0$. 

We proceed by combining the two approaches:
the exponential fitting technique modifies the norm, and the effective reaction term,
which is then supplemented by an additional reaction term, ensures coercivity. 
Once a coercive problem is obtained, we can simply adapt previous analyses to obtain an \aposteriori{} 
estimate; here, in particular, we follow the  analysis in~\cite{Cangiani2013a}. As we shall see, in this way a minimal amount of artificial reaction is introduced in all 
regimes. The benefit of combining these two approaches is that they can work together 
complementarily to give sharper results. By modifying the norm by an exponential-fitting
 technique, we are able to enlarge the set of convection fields under which no additional reaction
 is required to provide coercivity. However, for convection fields where this is not sufficient, 
 we still add enough reaction locally to ensure coercivity. In this manner, we reduce the
 additional reaction that must be added. This is important to minimise, since the corresponding non-stationary \aposteriori{} error bounds presented in the next section will  depend upon this additional reaction in an exponential fashion.

We consider the stationary convection-diffusion-reaction problem:
\begin{align}
- \diffusivity\laplacian\u + \conv\cdot\grad\u + \addreac\u
&= \heating(\bx) &&\text{ on } \domain, \label{eqn:statconvdiffreac}\\
\u 
&= 0 &&\text{ on } \dirbdy , \label{eqn:statconvdiffreacdirbdy}\\
\diffusivity\deriv{\u}{\normal} 
&= \neudata
&&\text{ on } \neubdy, \label{eqn:statconvdiffreacneubdy}
\end{align}
with $\addreac\in\linftyspace$, where we focus on the case of zero Dirichlet
boundary conditions without loss of generality, since this problem can always be reduced to such by altering $f$ and $\neudata$.

Introducing the relevant bilinear form $a_{\rm reac}:\Hone\times\Hone\rightarrow\reals$, given by 
%\begin{align*}%\label{eqn:defnbilinearA}
$\bilinearareac{\u}{\test} 
:= 
\bilineara{w}{\test} 
	+ \innerprod{\addreac w}{\test}$,
%\end{align*}
for all $\w,\test\in\Hone$, the weak formulation for the problem including reaction $\addreac$ then reads: %of \eqref{eqn:statconvdiffreac}--\eqref{eqn:statconvdiffreacneubdy}
%is:
find $\u\in\Honenonhomzero$ such that
\begin{align}\label{eqn:weak_pde}
\bilinearareac{\u}{\test} = \linearl{\test}\qquad\forall \test\in\Honenonhomzero.
\end{align}
Correspondingly, for for $\wh,\testh \in V_h+\Hone$, we define the bilinear form $\bilinearahreacblank$ as:
\[
\bilinearahreac{\wh}{\testh}\coloneqq\bilinearah{\wh}{\testh} +\innerprod{\addreac \wh}{\testh},
\]
and introduce the corresponding IPDG method: find $\uh \in \dgspace$ such that 
\begin{equation}\label{eqn:dgproblem}
\bilinearahreac{\uh}{\testh} = \linearlh{\testh}\qquad\forall\testh\in V_h.
\end{equation}

\subsection{Exponential fitting}

The exponential fitting approach is based on a Helmholtz decomposition of the convection field: for a convection field $\conv\in [\sobolevspace{1}{\infty}]^d$, there exist $\convgrad\in\Hone$ and $\convcurl\in\left[\Hone\right]^3$, such that
\begin{equation}\label{eq:helm}
\conv =  \grad\convgrad+\curl\convcurl,
\end{equation}
where, in the $d=2$ case, this should be interpreted as applied to a three-dimensional vector field with zero $z$-component; we refer, e.g.~\cite{temam1977navier,GiraultRaviart1986} for details. Moreover, 
given that $\domain$ is either a smooth or a convex polygonal or polyhedral domain,  we have that $\eta\in\sobolevspace{1}{\infty}$ 
and $\curl\convcurl \in [\linftyspace]^{\spacedim}$. Additionally, since $\conv\cdot\edgenormal = 0$ on $\neubdy$, 
we have $\grad\convgrad\cdot\edgenormal = 0$ on $\neubdy$ (cf. \cite[Theorem 3.2]{GiraultRaviart1986}).

\begin{remark}
	We note that the aforementioned regularity of $\eta$ and  $\convcurl$ follows from the sufficient assumptions on smoothness or convexity of the spatial computational domain $\Omega$. Alternatively, we can assume directly the regularity on  $\eta$ and$\convcurl$ instead of the domain $\Omega$. 
\end{remark}
We then define the \emph{weighting function} 
\begin{equation}
\weight \coloneqq \exp(-\alpha\convgrad),\label{eqn:weight_defn}
\end{equation} 
with $\alpha >0$ a constant to be determined later,
so that 
\begin{equation}\label{eqn:weight_grad}
\grad\weight = -\alpha\weight\grad\convgrad.
\end{equation} 
Since $\convgrad\in \sobolevspace{1}{\infty}$ 
we have that $\weight\in \sobolevspace{1}{\infty}$. 
Thus, $\weight\test \in \Hone$ for all $\test \in \Hone$, and 
$\weight\w \in \Honenonhomzero$ for all $\w \in \Honenonhomzero$.

We define the 
$\weight$-weighted $L^p$-norm $\lpsw{\cdot}{\gendomain}{p}$ by 
\begin{equation*}
\lpsw{ \test }{\gendomain}{p} \coloneqq \left(\int_\gendomain \weight {\test}^p \dx\right)^{1/p};
\end{equation*}
 we will suppress the $\gendomain$ subscript if $\gendomain = \domain$, 
 and suppress the $p$ subscript if $p=2$. For $p=\infty$, we set
$\linftysw{\test}{\gendomain} 
\coloneqq 
{\rm ess}\sup_\gendomain\abs*{\sqrt{\weight}\test}$.

We introduce the following helpful notation for later:
\begin{eqnarray}\label{eqn:L_defn}
\correctedltwo\coloneqq &\addreac + \half\ltwoweightexplicit,\\
\correctedltwofull\coloneqq &\hspace{-4mm}\addreac + \ltwoweightexplicit.\label{eqn:M_defn}
\end{eqnarray}

For appropriately large $\addreac$, depending on the nature of $\conv$, so that  $\correctedltwo \geq 0$, we  
define over $V_h+\Honenonhomzero$ the $\weight$-weighted dG norm 
\begin{equation}\label{eqn:dg_norm_defn}
\ndgp{\testh} \coloneqq \Big( \sumovercells \diffusivity\ltwosw{\grad\testh}{\cell}^2 + \sumovercells\ltwosw{\sqrt{\correctedltwo}\testh}{\cell}^2 + \sumoveredges\frac{\penal\diffusivity}{\edgediam}\ltwosw{\jump{\testh}}{\edge}^2\Big)^{1/2},
\end{equation}

The crucial feature of the $\weight$-weighted norm is the addition of the second term, which provides control in a (weighted) $L^2$-norm, possibly in the absence of reaction terms.

We note that, in the case of a divergence-free convection field, we may allow $\convgrad=0$, 
in which case $\correctedltwo=0$ if we also choose $\addreac=0$, whence the weighed $L^2$-norm control is lost. See Section
\ref{sect:relation_to_existing_results} and the numerical results for a discussion of this case. In the following analysis,
for simplicity of presentation, we assume $\correctedltwo \neq 0$, noting that all the results follow analogously
in the (simpler) case $\correctedltwo=0$ with the appropriate modifications.

\begin{assumption}
We assume that $\addreac$ is large enough so that  $\correctedltwo >0$. 
\end{assumption}

For $\mathbf{\w} \in \left[\ltwospace\right]^\spacedim$, we further define the semi-norm
\begin{equation*}
\ndgpstar{\mathbf{\w}} \coloneqq \sup_{\test \in \Honenonhomzero\backslash \{0\} } \frac{\int_\domain \mathbf{\w} \weight\cdot \nabla \test \dx}{\ndgp{\test}}.
\end{equation*}
Finally, we define 
\begin{equation}
\ndgpa{\testh} 
\coloneqq 
\big(\ndgpstar{\newconvwbraces\testh}^2 
+ \sumoveredges \frac{\edgediam\linftys{\newconv}{\edge}^2}{\diffusivity}\ltwosw{\jump{\testh}}{\edge}^2\Big)^{1/2}. \label{eqn:a_norm_defn}
\end{equation}
These norms will be used to bound the convective derivative, following the inf-sup argument in \cite{Verfurth2005, Schotzau2009}, described below.  
%We note that 
%\begin{equation*}
%\frac{\edgediam\linftys{\newconv }{\edge}^2}{\diffusivity} = \penal^{-1}Pe_L^2 \frac{\penal \diffusivity}{\edgediam},
%\end{equation*} 
%where $Pe_L$ is the modified local mesh P\'eclet number and 
%$\frac{\penal\diffusivity}{\edgediam}$ is the penalty term. 

Further, the following immediate observation will be useful below: for regular enough vector field $\bf b$ and scalar function $\w$:
\begin{align}
\ndgpstar{{\bf b}\w}
\leq
\frac{1}{\sqrt{\diffusivity}} \left( \sumovercells \left( \ltwosw{{\bf b}}{\cell,\infty}^2 \ltwos{\w}{\cell}^2 \right) \right)^\half.  \label{eqn:Hminusone_bound}
\end{align}
Also, define the modified mesh-Pecl\`et number by
\[
Pe_L:=	\frac{\edgediam\linftys{\newconv }{\edge}}{\sqrt{\diffusivity}}.
\]

For $w,\test\in\Hone$, using $\weight\test$ as test function in $\bilinearareacblank$ and applying the product rule, yields
\begin{equation*}% \label{eqn:bilinearA}
\bilinearareac{w}{\weight\test} 
=
\innerprod{\diffusivity\nabla w}{\weight \nabla\test} 
		+ \innerprod{\left(\conv - \alpha\diffusivity\nabla\convgrad\right) \cdot \nabla w}{\weight\test} 
		+ \innerprod{\addreac w}{\weight\test}.
\end{equation*}
Integration by parts, \eqref{eqn:weight_grad} along with $\grad\convgrad\cdot\edgenormal = 0$ on $\neubdy$, reveal
\begin{multline*}
\innerprod{\newconvwbraces  w}{\weight\nabla \test}
+ \innerprod{ \newconvwbraces \weight \test }{ \nabla w } \\
=
\innerprod{ \ltwoweightexplicit w}{\weight \test}
+ \edgeinnerprod{ \newconvwbraces \cdot \bdynormal w}{\weight \test}{\dirbdy}.
\end{multline*}
The latter allows us to write
\begin{multline}\label{eqn:bilinearAweight}
\bilinearareac{w}{\weight\test} 
= 
\innerprod{ \diffusivity \nabla w }{\weight \nabla \test }
+ \innerprod{ \left(\addreac + \ltwoweightexplicit \right)w}{\weight\test} \\
- \innerprod{ \newconvwbraces w }{\weight \nabla \test }
+ \edgeinnerprod{ \newconvwbraces \cdot \bdynormal w}{\weight \test}{\dirbdy},
\end{multline}
A similar argument applied to the interior penalty dG bilinear form yields
for $\wh,\testh\in V_h$,
\begin{equation}
\begin{aligned}\label{eqn:bilinearAweight_h}
&\ \bilinearahreac{\wh}{\weight\testh} \\
= &\ 
\sumovercells \cellinnerprod{\diffusivity\gradh\wh}{\weight\gradh\testh} 
		+ \cellinnerprod{\newconvwbraces\cdot\gradh\wh + \addreac \wh}{\weight\testh} \\
&\ - \sumoveredges \Big(\edgeinnerprod{\average{\diffusivity\grad\ltwoproj[k]\wh}}{\jump{\weight\testh}}{\edge} 
		+ \edgeinnerprod{\average{\diffusivity\grad\ltwoproj[k](\weight\testh)}}{\jump{\wh}}{\edge} 
	\Big) \\
&\ + \sumoveredges  \frac{\diffusivity\penal}{\edgediam}\edgeinnerprod{\jump{\wh}}{\jump{\weight\testh}}{\edge} \\
&\ - \sumovercells \Big(\edgeinnerprod{\newconvwbraces\cdot\edgenormal \wh}{\weight\testh}{\cellinflow\cap\dirbdy} \\
&\kern 2cm		+  \edgeinnerprod{\newconvwbraces \cdot \edgenormal[\cell] \flowjump{\weight\testh}}{\wh}{\cellinflow \backslash\dirbdy}\Big).
\end{aligned}
\end{equation}

We conclude this section establishing coercivity, continuity and an inf-sup stability bound  for~\eqref{eqn:bilinearAweight}.
\begin{lemma}%[Coercivity and continuity of $\bilinearareac{\cdot}{\weight\cdot}$]
\label{lemma:A_coercivity}
Let  $\addreac$ large enough so that  $\correctedltwo \geq 0$ with  $\correctedltwo $ defined in~\eqref{eqn:L_defn}. Then, for
$\w\in\Honenonhomzero$, 
\begin{equation*}
\bilinearareac{\w}{\weight\w} = \ndgp{\w}^2.
\end{equation*}
Moreover, under the assumption that, for a.e. $\bx\in\domain$,
\begin{equation}
\addreac(\bx) \geq \max\left\{0,-2\ltwoweightexplicit(\bx)\right\},\label{eqn:addreacchoice}
\end{equation} 
we have that, for $\wh\in V_h+\Honenonhomzero$, $\test\in\Honenonhomzero$,
\[
\begin{array}{ll}
\bilinearareac{\wh}{\weight\test} &\lesssim \left( \ndgp{\wh}+%\ndgpa{\wh}
\ndgpstar{\newconvwbraces\w}
\right)\ndgp{\test}\vspace{2mm}\\
&\lesssim \left( \ndgp{\wh}+\ndgpa{\wh}
\right)\ndgp{\test}
\end{array}
\]
\end{lemma}
\begin{proof}
Testing in~\eqref{eqn:bilinearAweight} with $\test=w\in \Honenonhomzero$ yields
\begin{multline}\label{eqn:bilinearAweightcoercivitybase}
\bilinearareac{w}{\weight w} 
= 
\innerprod{ \diffusivity \nabla w}{\weight \nabla w} 
+ \half\edgeinnerprod{ \newconvwbraces \cdot \bdynormal   w}{\weight w}{\dirbdy}\\
	+ \innerprod{ \left(\addreac + \half \ltwoweightexplicit \right)\w}{\weight w},
\end{multline}
from which the coercivity result immediately follows.

Let now $\wh\in V_h+\Honenonhomzero$ and $\test\in\Honenonhomzero$.
Assumption \eqref{eqn:addreacchoice} implies  
\[
\innerprod{ \diffusivity \nabla \w }{ \weight \nabla \test } + \innerprod{ \left(\addreac + \ltwoweightexplicit \right) \w}{\weight \test } \lesssim \ndgp{\w}\ndgp{\test},
\]
and inserting this into~\eqref{eqn:bilinearAweight}, we have   
\begin{align*}
\bilinearareac{\w}{\weight\test} 
=& \ 
%\innerprod{ \diffusivity \nabla \w }{\weight \nabla \test } 
%+ \innerprod{ \left(\addreac + \ltwoweightexplicit \right)\w}{\weight\test } \\
%&- \innerprod{ \newconvwbraces \w }{\weight \nabla \test } 
%+ \innerprod{ \newconvwbraces \cdot \bdynormal \w}{\weight\test}_\dirbdy \\
%\lesssim&
\innerprod{ \diffusivity \nabla \w }{\weight \nabla \test} 
+ \innerprod{ \left(\addreac + \ltwoweightexplicit \right)\w}{\weight\test } \\
&- \innerprod{ \newconvwbraces  \w }{\weight \nabla \test}  \\
\lesssim&
\left(\ndgp{\w}+\ndgpstar{\newconvwbraces\w} \right) \ndgp{\test}.
\end{align*}\qed
\end{proof}

\begin{remark}
We remark on the behaviour of the weight $\weight$ and 
the term $\correctedltwo$ based on the admissible values for the artificial reaction coefficient $\addreac$ and, thus, on the underlying flow pattern.
Recall $\weight \coloneqq \exp(-\alpha\convgrad)$ with  $\convgrad$  solution of the equation $\laplacian\convgrad = \div\conv$.
A negative divergence leads to a large weighting, a divergence-free field has weighting $\weight=1$ and a positive-divergence implies a reduced weighting.
Similarly, for small $\diffusivity$,  $\half \ltwoweightexplicit$  may be negative when $\div\conv$ is positive, 
and vice-versa. In turns, the size of $\correctedltwo$, which is non-negative by construction, is proportional to  the absolute size of the divergence, with $\correctedltwo=0$ a viable choice for divergence-free flows obtained picking   $\addreac=0$.
\end{remark}

\begin{lemma}\label{lemma:infsup}
There exists a constant $C>0$ such that 
\[ \inf_{\u \in \Honenonhomzero \backslash \{ 0 \}} \sup_{\test \in \Honenonhomzero \backslash \{ 0 \}} \frac{\bilinearareac{ \u}{ \weight\test}}{\left( \ndgp{\u} + \ndgpstar{\newconvwbraces\u}\right) \ndgp{\test}} \geq C > 0. \]
\end{lemma}
\begin{proof}
Let $\w \in \Honenonhomzero$ and $\Lambda \in (0,1)$. Then, there exists $\w_\Lambda \in \Honenonhomzero$ such that
\[
\ndgp{\w_\Lambda} = 1, \quad\text{and}\quad 
%\Oh{\w}{\w_\Lambda} = 
-\int_\domain \newconvwbraces \w \weight \cdot \nabla \w_\Lambda \dx \geq \Lambda \ndgpstar{\newconvwbraces\w}.
\]
From \eqref{eqn:bilinearAweight}, we have
\begin{align*}
\bilinearareac{\w}{\weight\w_\Lambda} 
%=& 
%\int_\domain \diffusivity \weight \nabla \w \cdot \nabla \w_\Lambda \dx + \int_\domain \left(\addreac + \ltwoweightexplicit \right)\weight\w\w_\Lambda \dx \\
%&- \int_\domain \newconvwbraces \weight \w \cdot \nabla \w_\Lambda \dx + \int_\dirbdy \newconvwbraces \cdot \bdynormal \weight \w\w_\Lambda \ds \\
=& 
\int_\domain \diffusivity \weight \nabla \w \cdot \nabla \w_\Lambda \dx \\
&+ \int_\domain \left(\addreac + \ltwoweightexplicit \right)\weight\w\w_\Lambda \dx \\
&- \int_\domain \newconvwbraces \weight \w \cdot \nabla \w_\Lambda \dx %+ \int_\neubdy \newconvwbraces \cdot \bdynormal \weight \w\w_\Lambda \ds
.
\end{align*}
Then, by Lemma \ref{lemma:A_coercivity}, we obtain
\begin{align*}
\bilinearareac{\w}{\weight\w_\Lambda} %= & \Dh{\u}{w_\Lambda} + \Jh{\u}{w_\Lambda} + \Oh{\u}{\w_\Lambda} + \addreac \left(\chardiff\w, \weight w_\Lambda \right)_{\domain}\\
&\geq \Lambda \ndgpstar{\newconvwbraces\w} - C_1 \ndgp{\w}\ndgp{\w_\Lambda}\\
&= \Lambda \ndgpstar{\newconvwbraces\w} - C_1 \ndgp{\w},
\end{align*}
for some positive constant $C_1$.

Define $\test_\Lambda = \w + \frac{\ndgp{\w}}{1+C_1} \w_\Lambda$. Obviously, $\ndgp{\test_\Lambda} \leq \left(1+\frac{1}{1+C_1}\right)\ndgp{\w}$.

So, using Lemma \ref{lemma:A_coercivity}, 
%$\bilinearareac{\w}{\weight\w} = \ndgp{\w}^2.$
%Therefore
\begin{align*}
\sup_{\test \in \Honenonhomzero\backslash\{0\}} \frac{\bilinearareac{\w}{\weight\test}}{\ndgp{\test}}  \geq & \frac{\bilinearareac{\w}{\weight\test_\Lambda}}{\ndgp{\test_\Lambda}} \\
 = & \frac{\bilinearareac{\w}{\weight\w} + \frac{\ndgp{\w}}{1+C_1} \bilinearareac{\w}{\weight\w_\Lambda}}{\ndgp{v_\Lambda}} \\
 \geq & \frac{\ndgp{\w}^2 + \frac{\ndgp{\w}}{1+C_1}\left(\Lambda\ndgpstar{\newconvwbraces\w} - C_1 \ndgp{\w}\right)}{\left(1 + \frac{1}{1+C_1}\right)\ndgp{\w}} \\
 = & \frac{\ndgp{\w} + \Lambda \ndgpstar{\newconvwbraces\w}}{2+C_1}.
\end{align*}
Since $\Lambda \in (0,1) $ and $\w \in \Honenonhomzero$ are arbitrary,
\begin{equation*}
\inf_{\w\in \Honenonhomzero \backslash \{ 0 \}} \sup_{\test\in \Honenonhomzero \backslash \{ 0 \}} \frac{\bilinearareac{\w}{\weight\test}}{\left( \ndgp{\w} + \ndgpstar{\newconvwbraces\w} \right)\ndgp{\test}}
\geq \frac{1}{2+C_1} > 0,
\end{equation*}
and the result follows.\qed
\end{proof}

\subsection{\emph{A posteriori} error analysis}

On each cell $\cell\in\tria$, we define the shorthand 
\begin{equation*}
\subweight{\cell} 
\coloneqq 
\left\{\begin{array}{cc} \diffusivity^{-\half} & \text{if } \weightmin{\cell} = \weightmax{\cell} = 1, \\
\max\left\{\frac{\gradweightmax{\cell}}{\sqrt{\correctedltwomin{\cell}}},\frac{\weightmax{\cell}}{\sqrt{\diffusivity}}\right\} & otherwise,
\end{array}\right.
\end{equation*}
where overline and underline denotes, respectively, the essential supremum and infimum of the Euclidean norm over the indicated cell; for instance, $\weightmin{\cell}=\esssup_\cell |\weight|$ and  $\weightmax{\cell}=\inf_\cell |\weight|$.
%
%We note that here we could define that if $\weight=1$ and $\correctedltwomin{\cell}=0$, then we say $\frac{\gradweightmax{\cell}}{\sqrt{\correctedltwomin{\cell}}} = \frac{0}{0} \coloneqq 0$, if that is a cleaner way to define this scenario. This is ONLY used in the case of identically-zero $\correctedltwo$ when $\grad\convgrad=0$ and we choose $\addreac=0$, as otherwise we would have a positive $\correctedltwo$.
%
Then, for each cell $\cell\in\tria$ and  $\edge\in\edges$ we introduce the local weighting functions
\begin{align}
\begin{array}{rlrl}
\cellweight \coloneqq &
\frac{1}{\sqrt{\weightmin{\cell}}}\min\left\{ \frac{\weightmax{\cell}}{\sqrt{\correctedltwomin{\cell}}}, \celldiam\subweight{\cell}
\right\}, 
&
%\edgepatchweight \coloneqq &
%\min_{\genericcell\in\edgepatch}
%\left\{ \frac{\genericcellweight^2}{\genericcelldiam} + \frac{\genericcelldiam}{\weightmin{\genericcell}} \subweight{\cell}^2
%\right\},
\edgepatchweight \coloneqq &
\min_{\genericcell\in\edgepatch}
\left\{\frac{\genericcelldiam}{\weightmin{\genericcell}} \subweight{\cell}^2
\right\},
\\
\cellvarweight \coloneqq &
\frac{\subweight{\cell}^2}{\weightmin{\cell}},
&
\edgepatchvarweight \coloneqq &
\max_{\genericcell\in\edgepatch}\genericcellvarweight.
\end{array}
\label{eqn:cell_and_edge_weights}
\end{align}

\begin{lemma}
With the above definitions, we observe the following estimates:
	\begin{align}\label{eqn:local_ltwo}
		\begin{split}
			\cellweight^{-1} \ltwos{\left(\identity-\ltwoproj[m]\right)\left(\weight\test\right)}{\cell} &\lesssim \ndgps{\test}{\cell},\\
			\edgepatchweight^{-\half} \ltwos{\left(\identity-\ltwoproj[m]\right)\left(\weight\test\right)}{\edge} 
			&\lesssim
			\ndgps{\test}{\edgepatch},
		\end{split} 
	\end{align}
	for any $\test \in \Honenonhomzero$, and any $\cell\tria$ and any face $F\in\edges$, for any $m=0,1,\dots, k$. The above imply also the global estimates 
	\begin{align}\label{eqn:global_ltwo}
		\begin{split}
			\bigg( \sumovercells \cellweight^{-2} \ltwos{\left(\identity-\ltwoproj[m]\right)\left(\weight\test\right)}{\cell}^2 \bigg)^\half &\lesssim \ndgp{\test},\\
			\bigg( \sumoveredges \edgepatchweight^{-1}\ltwos{\left(\identity-\ltwoproj[m]\right)\left(\weight\test\right)}{\edge}^2 \bigg)^\half &\lesssim \ndgp{\test},
		\end{split}
	\end{align}
	for any $\test \in \Honenonhomzero$. 
	
	\end{lemma}
	\begin{proof} We have
		\begin{equation}\label{magic}
		\begin{aligned}
	\ltwos{\left(\identity-\ltwoproj[m]\right)\left(\weight\test\right)}{\cell}
	 \lesssim &\ 
			\celldiam	\ltwos{\nabla\left(\weight\test\right)}{\cell}\\
			\lesssim &\  
				\frac{	\celldiam \gradweightmax{\cell}}{\sqrt{\correctedltwomin{\cell}\weightmin{\cell}}}\ltwosw{\sqrt{\correctedltwo}\test}{\cell}
				+
						\celldiam \sqrt{\weightmax{\cell}}\ltwosw{\nabla\test}{\cell}\\
		\lesssim &\ 
		 \celldiam \lambda_\cell	(\weightmin{\cell})^{-\half}\ndgps{\test}{\cell}.
		\end{aligned}
		\end{equation}
		At the same time, from the stability of orthogonal $L^2$-projection, we can also have 
		\[
			\ltwos{\left(\identity-\ltwoproj[m]\right)\left(\weight\test\right)}{\cell}
			\le 
						2 	\ltwos{\weight\test}{\cell}
			\lesssim   
			 \big( \weightmax{\cell}\correctedltwomin{\cell}^{-1}\big)^{\half}\ndgps{\test}{\cell}.
		\]
		Combining the above two estimates, we deduce the first bound in \eqref{eqn:local_ltwo}. 
		For the second bound, we start by observing the bound
		\[
		\ltwos{\left(\identity-\ltwoproj[m]\right)\left(\weight\test\right)}{\edge}
		\lesssim 
		\sqrt{\celldiam}	\ltwos{\nabla\left(\weight\test\right)}{\cell},
		\]
		and we conclude as in \eqref{magic}. 
		
		The global estimates \eqref{eqn:global_ltwo} follow by squaring, summation and the shape-regularity of the meshes which limits the amount of overlap occurring by the element patches.
\qed	
		\end{proof}
	%

%
%These bounds on the $L^2$-projector (see Appendix \ref{appendix_a} for proofs) are based 
%on the following result from \cite[3.5.22]{QuarteroniValli1994}:
%let $0 \leq l \leq k$. For any bounded, open set $\gendomain$ with diameter 
%$\diam{}$, and $\test \in \Hpspace{\gendomain}{l+1}$,
%\begin{equation}
%	\ltwos{\test-\ltwoproj\test}{\gendomain} + \diam{}\ltwos{\left(\test-\ltwoproj\test\right)}{\Hone[\gendomain]} \lesssim \diam{}^{l+1}\abs{\test}_{\Hpspace{\gendomain}{l+1}},\label{eqn:QuarteroniValli3.5.22}
%\end{equation}
%and the resulting bound:
%\begin{equation}\label{eqn:grad_bound}
%	\cellvarweight^{-1} \ltwos{\grad(\weight\test)}{\cell}^2 
%	\lesssim 
%	\ndgps{\test}{\cell}^2,
%\end{equation}
%along with two trace inequalities \cite{agmon2010lectures}: for $\testh\in\dgspace$, 
%and for any cell $\cell$ and edge $\edge\subset\cellbdy$, 
%\begin{equation*}
%	\ltwos{\testh}{\edge}^2 \lesssim \celldiam^{-1}\ltwos{\testh}{\cell}^2 + \ltwos{\testh}{\cell}\ltwos{\grad\testh}{\cell}%\label{eqn:mult_trace_ineq}
%\end{equation*}
%and
%\begin{equation*}
%	\ltwos{\testh}{\edge}^2 \lesssim \celldiam^{-1}\ltwos{\testh}{\cell}^2 + \celldiam \ltwos{\grad\testh}{\cell}^2.%\label{eqn:add_trace_ineq}
%\end{equation*}

%With such notation at hand, we are ready to introduce the following\aposteriori{} error estimator.
\begin{definition}%[A posteriori error estimator]
\label{def:apost_est_stationary}
Let $\uh\in V_h$. We define the \aposteriori{} error estimator is given by
\begin{equation} 
\statestimator 
\coloneqq 
\Big( \sumovercells \statestimatorcell^2 \Big)^\half, \label{eqn:a_posteriori_estimator_defn}
\end{equation} 
where, for each element $\cell \in \tria$ the  local error indicator 
$\statestimatorcell$ is defined by
\begin{equation*} 
\statestimatorcell^2 = \interiorresidualsquared + \edgeresidualsquared + \edgejumpestimatorsquared,
\end{equation*}
with the  following notation: the \emph{interior residual}  
\begin{equation*}
\interiorresidualsquared 
= 
\cellweight^2\ltwos{\heating + \diffusivity\laplacian\uh - \conv \cdot \nabla\uh - \delta\uh}{\cell}^2,
\end{equation*}
the \emph{face residual} $\edgeresidual$
\begin{equation}\label{normal_flux_jump}
\edgeresidualsquared 
= 
\half \sum_{\edge \in \cellbdy \backslash \domainbdy} \edgepatchweight \ltwos{\jump{\diffusivity \nabla \uh}}{\edge}^2,
\end{equation}
and the \emph{face jump indicator}  $\edgejumpestimator$ 
\begin{align*}
\edgejumpestimatorsquared 
&= 
\sum_{\edge \in \cellbdy } \left(
\frac{\penal\diffusivity}{\edgediam}\left(\weightmax{\edgepatch} + \edgepatchvarweight\penal\diffusivity %+ \edgepatchvarweight\diffusivity 
+\frac{\alpha^2\diffusivity\maxabs[\edge]{\grad\convgrad}^2}{\maxabs[\edge]{\weight}} \max_{\cell\in\edgepatch}\cellweight^2\right) +
\edgepatchweight\linftys{ \conv }{\edge}^2 \right. 
\\
&\left.\phantom{+\sumoveredges \left(\right.}+ \edgediam\linftysw{\correctedltwo}{\edgevertexpatch}
 +  \frac{\weightmax{\edgevertexpatch}\edgediam}{\diffusivity} \linftys{\conv-\alpha\diffusivity\grad\convgrad}{\edgevertexpatch}^2\right) \ltwos{\jump{\uh}}{\edge}^2,
\end{align*}
measuring the non-conformity of the function $\uh$.
\end{definition}
%
%We also introduce a data approximation term by
%\begin{equation*}
%\statdataosccell^2 
%= 
%\cellweight^2 \left(
%\ltwos{\heating - \heatingh}{\cell}^2 + \ltwos{\left(\conv - \convh\right) \cdot \nabla \uh}{\cell}^2
%\right)
%\end{equation*}
%
%The data approximation error is given by 
%\begin{equation} 
%\statdataosc = \left( \sumovercells \statdataosccell^2 \right)^\half. \label{eqn:data_approximation_error_defn}
%\end{equation}

The next step is to establish the robustness of~\eqref{eqn:a_posteriori_estimator_defn} in estimating the error between the interior penalty dG solution $\uh$ and the true solution $\u$ of~\eqref{eqn:weak_pde} in the weighted norm. A key technical tool used in the derivation of \aposteriori \, bounds below is the following  trivial extension to the case of weighted norms of a well-known stability
%conforming Karakashian-Pascal  approximation operator
result by Karakashian and Pascal~\cite{Karakashian2004}. 

\begin{theorem}%[Karakashian-Pascal operator]
\label{thm:KarakashianPascal}
Let $\dgspace[c] \coloneqq V_h \cap \Honenonhomzero$, the conforming subspace of 
$\dgspace$ which satisfies the Dirichlet boundary condition \eqref{eqn:statconvdiffreacdirbdy} and let a positive function $\xi\in\linftyspace$ be given. 
%Let $\KPapproxbase$ be the Karakashian-Pascal approximation operator, 
%$\KPapproxbase : \dgspace \rightarrow \dgspace[c]$, (cf. \cite{Karakashian2004}).
For any $\testh\in V_h$, there exists a function $\KPapprox{\testh}\in\dgspace[c]$,
satisfying  
\begin{align*}
\sumovercells \ltwosw{\xi\left(\testh - \KPapprox{\testh}\right)}{\cell}^2 \lesssim \sumoverinternalanddirichletedges \linftysw{\xi}{\edgevertexpatch}^2 \edgediam \ltwos{\jump{\testh}}{\edge}^2 %+ \sumoverdirichletedges \ltwosw{\xi}{\edgevertexpatch,\infty}^2 \edgediam \ltwos{\testh}{\edge}^2
,\\
\sumovercells \ltwosw{\xi\nabla \left(\testh - \KPapprox{\testh}\right)}{\cell}^2 \lesssim \sumoverinternalanddirichletedges \linftysw{\xi}{\edgevertexpatch}^2 \edgediam^{-1} \ltwos{\jump{\testh}}{\edge}^2 %+ \sumoverdirichletedges \ltwosw{\xi}{\edgevertexpatch,\infty}^2 \edgediam^{-1} \ltwos{\testh}{\edge}^2
.
\end{align*}
We refer to $\KPapproxbase : V_h \rightarrow \dgspace[c]$ as the KP
approximation operator.
\end{theorem}
\begin{proof}
	 We refer to \cite{Karakashian2004} for a constructive proof for $\xi=1$; the proof for general $\xi\in\linftyspace $ follows by the positivity and the boundedness of $\xi$. 
	 \qed
	 \end{proof}

In the spirit of \cite{Karakashian2004,HoustonPerugiaSchotzau2004, HoustonSchotzauWihler2007}, we decompose the discontinuous Galerkin solution into a conforming part and a non-conforming remainder:
\begin{equation*}
\uh = \uh[c] + \uh[d],
\end{equation*}
where $\uh[c] = \KPapprox{\uh} \in \dgspace[c] \coloneqq V_h \cap \Honenonhomzero$, with $\KPapproxbase$ the KP operator from Theorem \ref{thm:KarakashianPascal}, and $\uh[d] \coloneqq \uh - \uh[c]$. %We observe that $\jump{\uh[d]}_{\edge} = \jump{\uh}_{\edge}$ on each face $\edge$. 
Triangle inequality implies
\begin{equation}
\ndgp{\u-\uh} + \ndgpa{\u-\uh} \leq \ndgp{\u-\uh[c]} + \ndgpa{\u-\uh[c]} + \ndgp{\uh[d]} + \ndgpa{\uh[d]}. \label{eqn:full_norm_decomposition}
\end{equation}
To show that estimator bounds the true error, we proceed by bounding from above norms of both the nonconforming term $\uh[d]$ and the continuous error $\u -\uh[c]$.

\begin{lemma}\label{lemma:bound_nonconforming_norms}
We have the bound
\begin{multline*}
\ndgp{\uh[d]}^2 + \ndgpa{\uh[d]}^2 \\
\lesssim 
\sumoveredges\Big(\weightmax{\edge} \frac{\penal\diffusivity}{\edgediam} +\edgediam \linftysw{\correctedltwo}{\edgevertexpatch} + \frac{\weightmax{\edgevertexpatch}\edgediam}{\diffusivity} \linftys{\conv-\alpha\diffusivity\grad\convgrad}{\edgevertexpatch}^2 \Big) 
\ltwos{\jump{\uh}}{\edge}^2.
\end{multline*}
\end{lemma}
\begin{proof}
Since $\jump{\uh[d]}=\jump{\uh} $, we have
\begin{align*}
\ndgp{\uh[d]}^2 + \ndgpa{\uh[d]}^2 
=& 
\sumovercells \Big( \diffusivity \ltwosw{\nabla \uh[d]}{\cell}^2
+
\ltwosw{\sqrt{\correctedltwo}\uh[d]}{\cell}^2 \Big)
+
\ndgpstar{\newconvwbraces\uh[d]}^2 \\
&
+
\sumoveredges \left( \frac{\penal\diffusivity}{\edgediam} + \frac{\edgediam \linftys{\newconv}{\edge}^2}{\diffusivity} \right) \ltwosw{\jump{\uh}}{\edge}^2.
%\lesssim& 
%\sumovercells \left( \diffusivity \ltwosw{\nabla \uh[d]}{\cell}^2
%+
%\ltwosw{\sqrt{\correctedltwo}\uh[d]}{\cell}^2 \right)
%+
%\ndgpstar{\newconvwbraces\uh[d]}^2 \\
%&
%+
%\sumovercells \edgejumpestimatorsquared
\end{align*}
Theorem~\ref{thm:KarakashianPascal} yields 
\begin{equation*}
\sumovercells \diffusivity \ltwosw{\nabla \uh[d]}{\cell}^2 
\lesssim 
\penal^{-1} \sumoverinternalanddirichletedges \frac{\penal\diffusivity}{\edgediam} \weightmax{\edge}\ltwos{\jump{\uh}}{\edge}^2, 
\end{equation*}
and
\begin{equation*}
\sumovercells \ltwosw{\sqrt{\correctedltwo} \uh[d]}{\cell}^2 
%&\leq 
%\sumovercells \ltwosw{\sqrt{\correctedltwo}}{\cell,\infty}^2 \ltwos{\uh[d]}{\cell}^2 \\
\lesssim 
\sumoverinternalanddirichletedges \edgediam \linftysw{\correctedltwo}{\edgevertexpatch} \ltwos{\abs*{\jump{\uh}}}{\edge}^2.
\end{equation*}
To estimate $\ndgpstar{\left(\conv-\alpha\diffusivity\grad\convgrad\right)\uh[d]}$, we apply Theorem~\ref{thm:KarakashianPascal} once more, with the bound %
\eqref{eqn:Hminusone_bound}, and obtain
\begin{align*}
\ndgpstar{\left(\conv-\alpha\diffusivity\grad\convgrad\right)\uh[d]}^2 \leq & \diffusivity^{-1} \sumovercells\linftysw{\newconv}{\cell}^2 \ltwos{\uh[d]}{\cell}^2 \\
\lesssim & \sumoverinternalanddirichletedges \diffusivity^{-1}\edgediam \weightmax{\edgevertexpatch}\linftys{\newconv}{\edgevertexpatch}^2 \ltwos{\abs*{\jump{\uh}}}{\edge}^2.
\end{align*}
Finally,
\begin{multline*}
\sumoveredges \left( \frac{\penal\diffusivity}{\edgediam} + \diffusivity^{-1}\edgediam \linftys{\newconv}{\edge}^2 \right) \ltwosw{\jump{\uh}}{\edge}^2 \\
\leq
\sumoveredges \weightmax{\edge} \left( \frac{\penal\diffusivity}{\edgediam} + \diffusivity^{-1}\edgediam \linftys{\newconv}{\edge}^2 \right) \ltwos{\jump{\uh}}{\edge}^2.
\end{multline*}
%
%Thus, we have
%%
%\begin{align*}
%\ndgp{\uh[d]}^2 + \ndgpa{\uh[d]}^2
%\lesssim 
%&\sumoveredges \weightmax{\edge} \frac{\penal\diffusivity}{\edgediam} \ltwos{\jump{\uh}}{\edge}^2
%\\
%&+ \sumoverinternalanddirichletedges\left(\edgediam \linftysw{\correctedltwo}{\edgevertexpatch} + \frac{\weightmax{\edgevertexpatch}\edgediam}{\diffusivity} \linftys{\conv-\alpha\diffusivity\grad\convgrad}{\edgevertexpatch}^2 \right)\ltwos{\jump{\uh}}{\edge}^2.
%\end{align*}
%This finishes the proof. %We also use the fact that $1+\sqrt{\weightfrac} \thicksim \sqrt{1+\weightfrac} $.
Collecting together these bounds and noting that 
$\weightmax{\edge} \leq \weightmax{\edgevertexpatch}$ yields the result.\qed
\end{proof}

To bound the conforming error, we begin by noting that $\ndgpa{\u - \uh[c]}  = \ndgpstar{\newconvwbraces\left(\u - \uh[c]\right)}$, cf. \eqref{eqn:a_norm_defn}. Then, the inf-sup Lemma \ref{lemma:infsup} yields:
\begin{equation*}
\ndgp{\u - \uh[c]} + \ndgpstar{\newconvwbraces\left(\u - \uh[c]\right)} 
\lesssim 
\sup_{\test \in \Honenonhomzero \backslash \{ 0 \} } \frac{\bilinearareac{ \u - \uh[c] }{ \weight\test}}{\ndgp{\test}},
\end{equation*}
for any $\test \in \Honenonhomzero$, since $\weight \in \sobolevspace{1}{\infty}$, we have that $\weight\test \in \Honenonhomzero$.
Noting that $\bilinearahreac{w}{\test}=\bilinearareac{w}{\test}
$ for all $w, \test \in \Honenonhomzero$, and using  \eqref{eqn:weak_pde} and \eqref{eqn:dgproblem}, gives, respectively, for any $\test \in \Honenonhomzero$,
\begin{align*}
&\bilinearareac{ \u - \uh[c] }{ \weight\test}\\
%&= \bilinearareac{ \u}{\weight\test} - \bilinearareac{ \uh[c]}{\weight\test}\\%\tag*{bilinearity of $\bilinearareacblank$,}\\
&= \bilinearareac{ \u}{\weight\test} - \bilinearahreac{ \uh[c]}{\weight\test} \\%\tag*{by \eqref§, since $\uh[c],\weight\test \in \Honenonhomzero$,}\\
&= \bilinearareac{ \u }{ \weight\test} - \bilinearahreac{ \uh}{\weight\test} + \bilinearahreac{ \uh[d]}{\weight\test} \\%\tag*{bilinearity of $\bilinearahreacblank$,}\\
&= \left( \heating, \weight\test\right) - \bilinearahreac{ \uh}{\weight\test} + \bilinearahreac{ \uh[d] }{\weight\test} \\%\tag*{by \eqref{eqn:weak_pde},}\\
&= \left( \heating, \left(\identity-\ltwoproj[0]\right)\left(\weight\test\right)\right) + \left( \heating, \ltwoproj[0]\left(\weight\test\right)\right) - \bilinearahreac{ \uh }{\weight\test} + \bilinearahreac{ \uh[d] }{ \weight\test}\\
%&= \left( \heating, \left(\identity-\ltwoproj[k]\right)\left(\weight\test\right)\right) + \bilinearahreac{\uh }{ \ltwoproj[k]\left(\weight\test\right)} - \bilinearahreac{ \uh }{ \weight\test} + \bilinearahreac{ \uh[d] }{ \weight\test} \\%\tag*{by \eqref{eqn:dgproblem},}\\
&= \left( \heating, \left(\identity-\ltwoproj[0]\right)\left(\weight\test\right)\right) - \bilinearahreac{ \uh }{\left(\identity-\ltwoproj[0]\right)\left(\weight\test\right)} + \bilinearahreac{ \uh[d] }{ \weight\test}. %\tag*{bilinearity of $\bilinearahreacblank$.}
\end{align*}
We tackle the above terms in turns in the following lemmata.

\begin{lemma}\label{lemma:bound_projection_part}
For any $\test \in \Honenonhomzero$, we have
\begin{align*}
&\left( \heating, \left(\identity-\ltwoproj[0]\right)\left(\weight\test\right)\right) - \bilinearahreac{ \uh}{\left(\identity-\ltwoproj[0]\right)\left(\weight\test\right)} \\
\lesssim & 
\bigg( \sumovercells \left( \interiorresidualsquared + \edgeresidualsquared \right) +
\sumoveredges \left( \edgepatchvarweight \frac{\penal^2\diffusivity^2}{\edgediam}
+
\edgepatchweight\linftys{ \conv }{\edge}^2 \right) \ltwos{\jump{\uh}}{\edge}^2 \bigg)^\half \ndgp{\test}.
\end{align*}
\end{lemma}
\begin{proof}
Set
\begin{equation*}
T = \left( \heating, \left(\identity-\ltwoproj[0]\right)\left(\weight\test\right)\right) - \bilinearahreac{ \uh}{\left(\identity-\ltwoproj[0]\right)\left(\weight\test\right)}.
\end{equation*}
Then,  employing integration by parts,
\begin{align*}
T =& \sumovercells \cellinnerprod{\heating + \diffusivity \Delta \uh - \conv \cdot \nabla \uh - \delta\uh}{ \left(\identity-\ltwoproj[0]\right)\left(\weight\test\right) }\\
&-
\sumovercells \edgeinnerprod{ \diffusivity \nabla \uh \cdot \normal[\cell]}{\left(\identity-\ltwoproj[0]\right)\left(\weight\test\right)}{\cellbdy} \\
&+
\sumoverinternalanddirichletedges \edgeinnerprod{\average{\diffusivity \nabla \uh } }{ \jump{\left(\identity-\ltwoproj[0]\right)\left(\weight\test\right)} }{\edge}
\\
&-
\sumoverinternalanddirichletedges \frac{\penal\diffusivity}{\edgediam} \edgeinnerprod{ \jump{\uh} }{ \jump{\ltwoproj[0]\left(\weight\test\right)} }{\edge}
\\
&+ \sumovercells \edgeinnerprod{\conv \cdot \bdynormal[\cell] \uh }{ \left(\identity-\ltwoproj[0]\right)\left(\weight\test\right) }{\cellinflow\cap \dirbdy} \\
&+
\sumovercells \edgeinnerprod{\conv \cdot \normal[\cell] \flowjump{\uh} }{ \left(\identity-\ltwoproj[0]\right)\left(\weight\test\right) }{\cellinflow\backslash \dirbdy} 
\\
 =&\  T_1 + T_2 + T_3 + T_4 + T_5 + T_6.
\end{align*}
%For $T_1$, add and subtract data approximation terms:
%\begin{multline*}
%T_1 = \sumovercells \cellinnerprod{ \heatingh + \diffusivity \Delta \uh - \convh \cdot \nabla \uh - \delta\uh }{ \left(\identity - \ltwoproj[0]\right)\left(\weight\test\right) }
%\\+
%\sumovercells \cellinnerprod{(\heating - \heatingh) - ( \conv - \convh) \cdot \nabla \uh}{ \left(\identity - \ltwoproj[0]\right)\left(\weight\test\right)}
%\end{multline*}
%Then 
For $T_1$, using \eqref{eqn:global_ltwo}, we have
\begin{equation*}
%\begin{split}
T_1 \lesssim  \bigg( \sumovercells \interiorresidualsquared\bigg)^\half \bigg( \sumovercells  \cellweight^{-2} \ltwos{\left(\identity-\ltwoproj[0]\right)\left(\weight\test\right)}{\cell}^2 \bigg)^\half 
%&+
%\left( \sumovercells \statdataosccell^2 \right)^\half \left( \sumovercells  \cellweight^{-2} \ltwos{\left(\identity-\ltwoproj[0]\right)\left(\weight\test\right)}{\cell}^2 \right)^\half 
%\\ 
\lesssim  \bigg( \sumovercells \interiorresidualsquared \bigg)^\half \ndgp{\test}.
%\end{split}
\end{equation*}
%Next, rewriting $T_2+T_3$ in terms of jumps yields
%\begin{equation*}
%T_2 + T_3 
%= 
%-
%\sumovercells \edgeinnerprod{ \diffusivity \nabla \uh \cdot \normal[\cell] }{ \left(\identity-\ltwoproj[0]\right)\left(\weight\test\right) }{\cellbdy} 
%+
%\sumoverinternalanddirichletedges \edgeinnerprod{ \average{\diffusivity \nabla \uh} }{ \jump{\left(\identity-\ltwoproj[0]\right)\left(\weight\test\right)} }{\edge}.
%\end{equation*}
%Then by rearrangement,
$T_2+T_3$ can be written in terms of jumps and averages as follows
%\begin{equation*}
%T_2 + T_3 
%= 
%-
%\sumoverinternalanddirichletedges \edgeinnerprod{ \jump{\diffusivity \nabla \uh} }{ \average{\left(\identity-\ltwoproj[k]\right)\left(\weight\test\right)} }{\edge}
%+
%\sumoverneumannedges \edgeinnerprod{ \jump{\diffusivity \nabla \uh} }{ \average{\left(\identity-\ltwoproj[k]\right)\left(\weight\test\right)} }{\edge},
%\end{equation*}
%upon performing standard estimation.
%The Cauchy-Schwarz inequality and \eqref{eqn:global_ltwo} yield 
\begin{align*}
T_2 + T_3 
= &
-
\sumoverinternalanddirichletedges \! \edgeinnerprod{ \jump{\diffusivity \nabla \uh} }{ \average{\left(\identity-\ltwoproj[0]\right)\left(\weight\test\right)} }{\edge}\\
& +
\sumoverneumannedges \!\edgeinnerprod{ \jump{\diffusivity \nabla \uh} }{ \average{\left(\identity-\ltwoproj[0]\right)\left(\weight\test\right)} }{\edge}\\
\lesssim 
& \bigg( \sumoveredges \edgepatchweight \ltwos{\jump{\diffusivity \nabla \uh}}{\edge}^2 \bigg)^\half \bigg( \sumoveredges \edgepatchweight^{-1} \ltwos{\left(\identity-\ltwoproj[0]\right)\left(\weight\test\right)}{\edge}^2 \bigg)^\half 
\\
\lesssim& \bigg( \sumoveredges \edgepatchweight \ltwos{\jump{\diffusivity \nabla \uh}}{\edge}^2 \bigg)^\half \ndgp{\test}
\lesssim\bigg( \sumovercells \edgeresidual \bigg)^\half \ndgp{\test},
%\lesssim & \left( \sumovercells \frac{1}{\weightmin_\highlight{\cell}}\edgeresidualsquared \right)^\half \ndgp{\test}
\end{align*}
employing again \eqref{eqn:global_ltwo} in the penultimate inequality.  

To bound $T_4$, 
we begin by noting $\jump{\weight\test}=0$ a.e.~on each $\edge\in\edges$, and we have
%and invoking in turn a trace-inverse inequality, the triangle inequality, and the respective optimal approximation properties for the Cl\'ement interpolant and the $L^2$-projection operator~\cite{schwab1998p,warburton2003constants}, yields
\begin{align*}
T_4 =&-
\sumoverinternalanddirichletedges \frac{\penal\diffusivity}{\edgediam} \edgeinnerprod{ \jump{\uh} }{ \jump{\ltwoproj[0]\left(\weight\test\right)} }{\edge}  =-\sumoverinternalanddirichletedges \frac{\penal\diffusivity}{\edgediam} \edgeinnerprod{ \jump{\uh} }{ \jump{(\identity-\ltwoproj[0])\left(\weight\test\right)} }{\edge} 
\\ 
\lesssim 
&\bigg(\sumoveredges \edgepatchvarweight \frac{\penal^2\diffusivity^2}{\edgediam} \ltwos{\jump{\uh}}{\edge}^2 \bigg)^\half
\bigg( \sumoveredges \edgepatchvarweight^{-1} \edgediam^{-1} \ltwos{\jump{\left(\identity -\ltwoproj[0]\right)\left(\weight\test\right)}}{\edge}^2 \bigg)^\half 
%\\
%\lesssim &\bigg(\sumoveredges \edgepatchvarweight \frac{\penal^2\diffusivity^2}{\edgediam} \ltwos{\jump{\uh}}{\edge}^2 \bigg)^\half 
%\bigg( \sumovercells \cellvarweight^{-1} \ltwos{\grad\left(\weight\test\right)}{\cell}^2 \bigg)^\half  %&& \text{(Clement bound and Q+V 3.5.22)}
\\
\lesssim &\bigg(\sumoveredges \edgepatchvarweight \frac{\penal^2\diffusivity^2}{\edgediam} \ltwos{\jump{\uh}}{\edge}^2 \bigg)^\half \ndgp{\test},
\end{align*}
using \eqref{eqn:global_ltwo} and \eqref{eqn:cell_and_edge_weights}.

To bound the final terms $T_5 + T_6$, we again use \eqref{eqn:global_ltwo} and work as above:
\begin{align*}
T_5 + T_6 &= \sumovercells \edgeinnerprod{\conv \cdot \bdynormal[\cell] \uh }{ \left(\identity-\ltwoproj[0]\right)\left(\weight\test\right) }{\cellinflow\cap \dirbdy} \\
&\phantom{\sumovercells}+
\sumovercells \edgeinnerprod{\conv \cdot \normal[\cell] \flowjump{\uh} }{ \left(\identity-\ltwoproj[0]\right)\left(\weight\test\right) }{\cellinflow\backslash \dirbdy} \\
&= \sumoveredges \edgeinnerprod{ \jump{\conv \uh} }{ \left(\identity - \ltwoproj[0] \right) \left(\weight\test\right)}{\edge} 
\\
& \lesssim \bigg( \sumoveredges \edgepatchweight \ltwos{\jump{ \conv\uh}}{\edge} ^2\bigg)^\half 
\bigg( \sumoveredges \edgepatchweight^{-1} \ltwos{\left(\identity-\ltwoproj[0]\right)\left(\weight\test\right)}{\edge}^2 \bigg)^\half 
\\
& \lesssim \bigg( \sumoveredges \edgepatchweight\linftys{ \conv }{\edge}^2 \ltwos{\jump{ \uh}}{\edge}^2 \bigg)^\half \ndgp{\test},
%& \lesssim \left( \sumovercells \edgejumpestimatorsquared \right)^\half \ndgp{\test}
\end{align*}
from the continuity of $\conv$ in the normal direction.\qed
%This finishes the proof. We note that the sharper form would be
%\begin{align*}
%\int_\domain \heating \weight (\test-\Clem\test) \dx - &\bilinearAhtildepsi{\uh,\test-\Clem\test} \\
%\lesssim& 
%\sqrt{\weightfrac} \left(\left(\sumovercells\interiorresidualsquared\right)^\half + \left(\sumovercells\edgeresidualsquared\right)^\half + \left( \sumoveredges \frac{\edgediam \ltwos{ \conv }{\edge,\infty}^2}{\diffusivity} \ltwosw{\jump{\uh}}{\edge}^2 \right)^\half + \Theta_{\weight}\right) \ndgp{\test}
%\end{align*}
\end{proof}

\begin{lemma}\label{lemma:bound_nonconforming_part}
For any $\test \in \Honenonhomzero$, there holds
\begin{align*}
\bilinearahreac{\uh[d]}{\weight\test} 
\lesssim& 
\bigg( \sumoveredges \bigg( \frac{\penal\diffusivity}{\edgediam}\bigg(\weightmax{\edgepatch} + \edgepatchvarweight\diffusivity  +\frac{\alpha^2\diffusivity\maxabs[\edge]{\grad\convgrad}^2}{\maxabs[\edge]{\weight}} \max_{\cell\in\edgepatch}\cellweight^2\bigg)
\\
&\ + \edgediam\linftysw{\correctedltwofull}{\edgevertexpatch}+  \vphantom{\left( \frac{\weightmax{\edge}\alpha^2\diffusivity\maxabs[\edge]{\grad\convgrad}^2}{\correctedltwomin{\edgepatch}}\right)}\frac{\edgediam}{\diffusivity} \linftysw{\conv-\alpha\diffusivity\grad\convgrad}{\edgevertexpatch}^2 \bigg) \ltwos{\jump{\uh}}{\edge}^2 \bigg)^\half
\ndgp{\test},
\end{align*}
with $\correctedltwofull$ defined as in \eqref{eqn:M_defn}.
\end{lemma}
\begin{proof}
Recalling the definition of $\bilinearahreacblank$, we have 
\begin{align*}
\bilinearahreac{\uh[d]}{\weight\test} 
=
& \sumovercells \cellinnerprod{\diffusivity\gradh\uh[d]}{\gradh\left(\weight\test\right)} - \cellinnerprod{\uh[d]}{\gradh\cdot\left(\conv\weight\test\right)} + \cellinnerprod{\addreac \uh[d]}{\weight\test} \\
&- \sumoveredges \edgeinnerprod{ \average{\diffusivity\grad\ltwoproj[k]\left(\weight\test\right)}}{\jump{\uh[d]}}{\edge} \\
=& \sumovercells \cellinnerprod{\diffusivity\weight\grad\uh[d]}{\grad\test} - \cellinnerprod{\left(\conv - \alpha\diffusivity\grad\convgrad\right)\weight\uh[d]}{\grad\test} \\
&+ \sumovercells \cellinnerprod{\correctedltwofull\uh[d]}{\weight\test}\\
&-
\sumoveredges \edgeinnerprod{ \average{\diffusivity\grad\ltwoproj[k]\left(\weight\test\right)} }{ \jump{\uh[d]} }{\edge} - \sumovercells \alpha\diffusivity\edgeinnerprod{ \grad\convgrad \cdot \normal[\cell]\uh[d]}{\weight\test}{\cellbdy}\\
=& S_1 + S_2 + S_3 + S_4 + S_5,
\end{align*}
by the product rule, and integration by parts.
By the Cauchy-Schwarz inequality 
and Theorem \ref{thm:KarakashianPascal},
\begin{align*}
S_1 
% &
%\leq \bigg(\sumovercells\int_\cell \diffusivity \weight \abs*{\grad\uh[d]}^2\dx\bigg)^\half \bigg(\sumovercells\int_\cell\diffusivity\weight\left|\grad\test\right|^2\dx\bigg)^\half 
%\\
&\leq \bigg(\sumovercells\int_\cell \diffusivity \weight \abs*{\grad\uh[d]}^2\dx\bigg)^\half \ndgp{\test} 
\lesssim \penal^{-\half}\bigg(\sumoveredges\weightmax{\edgepatch} \frac{\penal\diffusivity}{\edgediam} \ltwos{\jump{\uh}}{\edge}^2\bigg)^\half \ndgp{\test}.
\end{align*}
Using the definition of the semi-norm $\ndgpstar{\cdot}$, Theorem~\ref{thm:KarakashianPascal} 
and \eqref{eqn:Hminusone_bound},
\begin{align*}
S_2 \leq&\  \ndgpstar{\left(\conv-\alpha\diffusivity\grad\convgrad\right)\uh[d]}\ndgp{\test}
\\
\leq& \ \diffusivity^{-\half} \bigg( \sumovercells \linftysw{\conv-\alpha\diffusivity\grad\convgrad}{\cell}^2 \ltwos{\uh[d]}{\cell}^2 \bigg)^\half \ndgp{\test} 
\\
\lesssim& \bigg( \sumoveredges \frac{\edgediam}{\diffusivity} \linftysw{\conv-\alpha\diffusivity\grad\convgrad}{\edgevertexpatch}^2 \ltwos{\jump{\uh}}{\edge}^2 \bigg)^\half \ndgp{\test}.
\end{align*}
From Theorem~\ref{thm:KarakashianPascal}, and by the identity $\correctedltwofull+\addreac = 2\correctedltwo$ (see \eqref{eqn:L_defn}), and the choice of $\addreac$ from \eqref{eqn:addreacchoice}, we have, respectively,
\begin{align*}
S_3 
%&\leq \bigg(\sumovercells\int_\cell\weight\correctedltwofull\left(\uh[d]\right)^2 \dx \bigg)^\half \left(\sumovercells \int_\cell\weight\correctedltwofull\test[2]\dx\right)^\half 
%\\
&\leq \bigg(\sumovercells\int_\cell\weight\correctedltwofull\left(\uh[d]\right)^2 \dx \bigg)^\half \bigg(\sumovercells \int_\cell\weight\left(\correctedltwofull+\addreac\right)\test[2]\dx\bigg)^\half 
\\
&\lesssim \bigg(\sumovercells\int_\cell\weight\correctedltwofull\left(\uh[d]\right)^2 \dx \bigg)^\half \ndgp{\test} 
\\
&\lesssim \bigg( \sumoveredges \edgediam \linftysw{\correctedltwofull}{\edgevertexpatch}^2 \ltwos{\jump{\uh}}{\edge}^2 \bigg)^\half \ndgp{\test}.
\end{align*}
Employing standard inverse estimates, we have, respectively,
\begin{align*}
\ltwos{\nabla\ltwoproj[k](\weight\test)}{\edge}^2=&\ \ltwos{\nabla\ltwoproj[k](\identity-\ltwoproj[0])(\weight\test)}{\edge}^2\lesssim \celldiam^{-3}\ltwos{\ltwoproj[k](\identity-\ltwoproj[0])(\weight\test)}{\cell}^2\\
\le&\ \celldiam^{-3}\ltwos{(\identity-\ltwoproj[0])(\weight\test)}{\cell}^2
%\lesssim \celldiam^{-1}\ltwos{\nabla (\weight\test)}{\cell}^2,
\end{align*}
from the stability of the $L^2$-projection, so that
\begin{align*}
S_4 &\leq \penal^{-\half}\bigg( \sumoveredges \int_\edge \edgepatchvarweight\frac{\penal\diffusivity^2}{\edgediam}|\jump{\uh}|^2\ds\bigg)^\half
 \bigg(\sumoveredges\edgepatchvarweight^{-1}\edgediam\int_\edge|\average{\grad\ltwoproj[k]\left(\weight\test\right)}|^2\ds\bigg)^\half
%\\
%&\lesssim \penal^{-\half}\left( \sumoveredges \edgepatchvarweight\frac{\penal\diffusivity^2}{\edgediam}\ltwos{\jump{\uh}}{\edge}^2\right)^\half
% \left(\sumoveredges\edgepatchvarweight^{-1}\edgediam\int_\edge\average{\grad\left(\weight\test\right)}^2\ds\right)^\half
\\
&\lesssim \penal^{-\half}\bigg( \sumoveredges \edgepatchvarweight\frac{\penal\diffusivity^2}{\edgediam}\ltwos{\jump{\uh}}{\edge}^2\bigg)^\half
 \bigg(\sumovercells\cellvarweight^{-1}h_{\cell}^{-2}\ltwos{(\identity-\ltwoproj[0])\weight\test}{\cell}^2\bigg)^\half
\\
&\lesssim \penal^{-\half}\bigg( \sumoveredges \edgepatchvarweight\frac{\penal\diffusivity^2}{\edgediam}\ltwos{\jump{\uh}}{\edge}^2\bigg)^\half \ndgp{\test}.
\end{align*}
from \eqref{eqn:global_ltwo}.
Finally, straightforward estimation and a trace estimate imply, respectively,
\begin{align*}
	S_5 =&
	-\sumovercells \int_\cellbdy \alpha\diffusivity\grad\convgrad \cdot \normal[\cell]\weight\uh[d]\test\ds
	=-\sumoveredges \int_\edge \alpha\diffusivity \grad\convgrad\cdot\jump{\uh[d]} \weight\test \ds \\
	\le &
\sumoveredges  \alpha\diffusivity\maxabs[\edge]{\grad\convgrad}\sqrt{\maxabs[\edge]{\weight}}\ltwosw{\jump{\uh}}{\edge}\ltwos{\test}{\edge}  \\
\lesssim&
\sumoveredges  \alpha\diffusivity\maxabs[\edge]{\grad\convgrad}\sqrt{\maxabs[\edge]{\weight}}\ltwosw{\jump{\uh}}{\edge}\celldiam^{-\half}\big(\ltwos{\test}{\cell} +\celldiam\ltwos{\nabla\test}{\cell}\big)\\
\lesssim&
\sumoveredges  \alpha\diffusivity\maxabs[\edge]{\grad\convgrad}\big(\maxabs[\edge]{\weight}\minabs[\cell]{\weight}^{-1}\celldiam^{-1}\big)^\half \ltwosw{\jump{\uh}}{\edge}\big(\correctedltwomin{\cell}^{-\half}\ltwosw{\sqrt{\correctedltwo}\test}{\cell} +\celldiam\ltwosw{\nabla\test}{\cell}\big), 
\end{align*}
for  an element $\cell$ with $\edge\subset \partial\cell$. Continuing with application of the Cauchy-Schwarz inequality and \eqref{eqn:cell_and_edge_weights}, we get
\begin{align*}
	S_5 
	\lesssim& \ \penal^{-\half}
	\bigg(\sumoveredges  \frac{\penal\alpha^2\diffusivity^2\maxabs[\edge]{\grad\convgrad}^2}{\edgediam\maxabs[\edge]{\weight}} \max_{\cell\in\edgepatch}\cellweight^2\ltwosw{\jump{\uh}}{\edge}^2\bigg)^\half
\ndgp{\test}
\end{align*}

%
%\begin{align*}
%S_5 =&
%-\sumovercells \int_\cellbdy \alpha\diffusivity\grad\convgrad \cdot \normal[\cell]\weight\uh[d]\test\ds
%=-\sumoveredges \int_\edge \alpha\diffusivity \grad\convgrad\cdot\jump{\uh[d]} \weight\test \ds \\
%\leq&
%\penal^{-\half}\left( \sumoveredges \frac{\penal\alpha^2\diffusivity^2\maxabs[\edge]{\grad\convgrad}^2}{\edgediam\correctedltwomin{\edgepatch}}\ltwosw{\jump{\uh}}{\edge}^2 \right)^\half 
%\left(\sumoveredges \int_\edge \edgediam\correctedltwomin{\edgepatch}\weight\test[2] \ds \right)^\half \\
%\leq&
%\penal^{-\half}\left( \sumoveredges \frac{\penal\alpha^2\diffusivity^2\maxabs[\edge]{\grad\convgrad}^2}{\edgediam\correctedltwomin{\edgepatch}}\ltwosw{\jump{\uh}}{\edge}^2 \right)^\half 
%\left(\sumovercells \int_\cell \correctedltwomin{\cell}\weight\test[2] \dx \right)^\half \\
%\leq&
%\penal^{-\half}\left( \sumoveredges \frac{\penal\alpha^2\diffusivity^2\maxabs[\edge]{\grad\convgrad}^2}{\edgediam\correctedltwomin{\edgepatch}}\ltwosw{\jump{\uh}}{\edge}^2 \right)^\half 
%\left(\sumovercells \ltwosw{ \sqrt{\correctedltwo}\test}{\cell}^2 \right)^\half \\
%\leq&
%\penal^{-\half}\left( \sumoveredges \frac{\penal\alpha^2\diffusivity^2\maxabs[\edge]{\grad\convgrad}^2}{\edgediam\correctedltwomin{\edgepatch}}\ltwosw{\jump{\uh}}{\edge}^2 \right)^\half 
%\ndgp{\test}\\
%\leq&
%\penal^{-\half}\left( \sumoveredges \frac{\weightmax{\edge}\penal\alpha^2\diffusivity^2\maxabs[\edge]{\grad\convgrad}^2}{\edgediam\correctedltwomin{\edgepatch}}\ltwos{\jump{\uh}}{\edge}^2 \right)^\half 
%\ndgp{\test}.
%\end{align*}
\end{proof}
%\begin{remark}
%In the case of a constant $\convgrad$,  we have no $S_5$ term in the above proof, and thus in the following results the term 
%$\frac{\weightmax{\edge}\penal\alpha^2\diffusivity^2\maxabs[\edge]{\grad\convgrad}^2}{\edgediam\correctedltwomin{\edgepatch}}$
%should then not appear, being treated as $0$ rather than $\frac{0}{\minabs[\edgepatch]{\addreac}}$, 
%which may not be defined in the case $\minabs[\edgepatch]{\addreac}=0$.
%\end{remark}

Collecting together the above developments immediately yields a bound on the conforming error as follows.
\begin{lemma}\label{lemma:conforming_bound}
There holds:
\begin{align*}
\ndgp{\u - \uh[c]} &+ \ndgpa{\u - \uh[c]}
\lesssim
\left( \sumovercells \left( \interiorresidualsquared + \edgeresidualsquared\right) \right.
\\
&+
\sumoveredges \left( %\edgepatchvarweight \frac{\penal\diffusivity}{\edgediam}
\frac{\penal\diffusivity}{\edgediam}\bigg(\weightmax{\edgepatch} + \edgepatchvarweight\penal\diffusivity %+ \edgepatchvarweight\diffusivity
 +\frac{\alpha^2\diffusivity\maxabs[\edge]{\grad\convgrad}^2}{\maxabs[\edge]{\weight}} \max_{\cell\in\edgepatch}\cellweight^2\bigg) +
\edgepatchweight\linftys{ \conv }{\edge}^2 \right. 
\\
&\left.\phantom{\sumoveredges \left(\right.}\left.\vphantom{\frac{\weightmax{\edge}\alpha^2\diffusivity\maxabs[\edge]{\grad\convgrad}^2}{\correctedltwomin{\edgepatch}}} \edgediam\linftysw{\correctedltwofull}{\edgevertexpatch}
 +  \frac{\edgediam}{\diffusivity} \linftysw{\conv-\alpha\diffusivity\grad\convgrad}{\edgevertexpatch}^2\right) \ltwos{\jump{\uh}}{\edge}^2 \right)^\half.
\end{align*}\qed
\end{lemma}

Finally, combining \eqref{eqn:full_norm_decomposition} with Lemma~\ref{lemma:bound_nonconforming_norms} and Lemma~\ref{lemma:conforming_bound}, and noting that $\correctedltwofull \lesssim \correctedltwo$, we are able to establish an upper bound for the\aposteriori{} error estimator.

\begin{theorem}\label{thm:stationary_estimate}
Let $\u$ be the solution of \eqref{eqn:statconvdiffreac}--\eqref{eqn:statconvdiffreacneubdy} 
and $\uh$ its discontinuous Galerkin approximation, the solution of \eqref{eqn:dgproblem}.
Then, the following bound holds:
\begin{equation*}
\ndgp{\u-\uh} + \ndgpa{\u-\uh} \\
\lesssim
\bigg( \sumovercells \left( \interiorresidualsquared + \edgeresidualsquared+\edgejumpestimatorsquared \right) \bigg)^\half.
\end{equation*}\qed
\end{theorem}

\section{\emph{A posteriori} error analysis for the semi-discrete  method}\label{sec:apost_time}
Having proven an \aposteriori{} error bound on the stationary convection-diffusion-reaction 
equation in the above modified norm, we are ready to consider the non-stationary model convection-diffusion problem~\eqref{eqn:convdiffweak}. We shall do that in two steps: first, we derive an \aposteriori{} error bound for the semi-discrete method to highlight the issues  specific to the interior penalty dG discretisation, and then we will complete the analysis for the fully-discrete implicit Euler dG method.

For the proof of the  \aposteriori{} error bound, our strategy is to reframe it as a convection-diffusion-reaction 
problem by means of the observation that we may rewrite the equation
\begin{equation*}
\timederiv{\u} - \diffusivity\laplacian\u + \conv\cdot\grad\u 
= \heating, %\tag{\ref{eqn:convdiff}}
\end{equation*}
as
\begin{equation*}
\timederiv{\u} - \diffusivity\laplacian\u + \conv\cdot\grad\u + \addreac\u
= \heating + \addreac\u .
\end{equation*}
Then, using the elliptic reconstruction framework \cite{Makridakis2003, LM06, M07, Georgoulis2011, Bansch2012, Cangiani2013a,GM23}, and a Gr\"onwall 
inequality, we arrive at an error bounds upon converting the reaction term into an exponential factor in the final error bound.

We  consider the spatially discrete scheme: find $\uh \in C^{0,1}([0,T];V_h)$ such that 
\begin{align}\label{eqn:convdiffsemi-discrete_r}
\left({\timederiv{\uh}},{\testh}\right)
+	\bilinearahreac{\uh}{\testh}  
= 
\left(\heating+\addreac\uh,\testh\right)
\end{align}
for all $\testh \in \dgspace$, with $\uh(\starttime) = \ltwoproj[k]\initu$.

\begin{definition}
For each $t\in (0,T]$, the  \emph{elliptic reconstruction} of $\uh(t)$ is the unique $\semielliprecon\in\Honenonhomzero$, such that
\begin{equation}
\bilinearareac{\semielliprecon}{\test}  
= 
\left(\heating + \addreac\uh - \timederiv{\uh},\test\right) \quad\forall \test\in\Honenonhomzero.\label{eqn:semi_ellip_recon}
\end{equation}
\end{definition}
%It is easy to see that the elliptic reconstruction $\semielliprecon$ is the exact solution to the parabolic PDE whose dG approximation is $\uh$ \cite{Makridakis2003}.
%
The interior penalty dG discretisation of the above elliptic reconstruction problem reads: find $\semiellipreconh\in V_h$, such that
%such that at each time $\timevar\in\timeinterval$,
\begin{equation*}
\bilinearahreac{\semiellipreconh}{\testh}  
= 
\left(\heating + \addreac\uh - \timederiv{\uh},\testh\right) \quad\forall \testh\in V_h.
\end{equation*}
Then, the uniqueness of the solution to the above problem and \eqref{eqn:convdiffsemi-discrete_r} implies that $\semiellipreconh = \uh$.
We can, therefore, apply the stationary case bound of 
Theorem \ref{thm:stationary_estimate}, to conclude that
\begin{multline}
%\begin{split}
\ndgp{\semielliprecon-\uh} + \ndgpa{\semielliprecon-\uh}   \\
\lesssim \sumovercells %+ \statdataosc
\left( \cellweight^2\ltwos{\heating - \timederiv{\uh}+ \diffusivity\laplacian\uh - \conv \cdot \nabla\uh}{\cell}^2 + \edgeresidualsquared\right)
+
\sumoveredges \edgejumpestimatorsquared.\label{eqn:semiidentity}
%\end{split}
\end{multline}

%\subsection{A posteriori error analysis}
We introduce the following splitting of the error $\error := \u-\uh$:
\begin{equation*}
\error = \confdiff + \dgdiff \quad \text{ with }\quad \confdiff \coloneqq \u-\semielliprecon, \quad \dgdiff \coloneqq \semielliprecon-\uh,
\end{equation*}
along with the extra notation $\errorconf \coloneqq \u- \KPapprox{\uh}$ and $\dgdiffconf \coloneqq \semielliprecon- \KPapprox{\uh}$, nothing that $\errorconf , \dgdiffconf\in\Honenonhomzero$.

\begin{theorem}\label{thm:semi-discrete-apost}
Let $\u$ be the solution of  \eqref{eqn:convdiff} and $\uh$ its semi-discrete approximation satisfying \eqref{eqn:convdiffsemi-discrete}. 
Then, we have the following \aposteriori{} error bound:
\begin{align*}
&\ltwosw{\error}{\extendedlinftyspace{\starttime}{\timevar}{\ltwospace}}^2 + \int_\starttime^{\timevar}  \ndgp{\error}^2 \ds\\
&\phantom{+}\lesssim
\exp\left(\int_{\starttime}^{\timevar} \max_{\domain}\frac{\addreac[2]}{\correctedltwo}(s) \ds\right)
\left( \ltwow{\error(\starttime)}^2
		+ \int_\starttime^{\timevar} \sdeSsub{1}^2 + \sdeSsub{2}^2 \ds + \max_{0\leq s\leq\timevar} \sdeSsub{3}^2 \right),
\end{align*}
whereby
\begin{align*}
\sdeSsub{1}^2 &\coloneqq  \sumovercells \cellweight^2\ltwos{\heating -\timederiv{\uh} + \diffusivity\laplacian\uh - \conv \cdot \nabla\uh}{\cell}^2
% + \ltwos{\heating - \heatingh}{\cell}^2 + \ltwos{\left(\conv - \convh\right) \cdot \nabla \uh}{\cell}^2 \right) 
%\\
% &
%\phantom{\coloneqq}
+  \sumoverinternaledges \edgepatchweight \ltwos{\jump{\diffusivity \nabla \uh}}{\edge}^2 
\\
&\ +
\sumoveredges \bigg(
\frac{\penal\diffusivity}{\edgediam}\bigg(\weightmax{\edgepatch} + \edgepatchvarweight\penal\diffusivity %+ \edgepatchvarweight\diffusivity 
+ \frac{\alpha^2\diffusivity\maxabs[\edge]{\grad\convgrad}^2}{\maxabs[\edge]{\weight}} \max_{\cell\in\edgepatch}\cellweight^2\bigg) +
\edgepatchweight\linftys{ \conv }{\edge}^2 
\\
&\phantom{\coloneqq+\sumoveredges \left(\right.}+ \edgediam\linftysw{\correctedltwo}{\edgevertexpatch}
 +  \frac{\weightmax{\edgevertexpatch}\edgediam}{\diffusivity} \linftys{\conv-\alpha\diffusivity\grad\convgrad}{\edgevertexpatch}^2\bigg) \ltwos{\jump{\uh}}{\edge}^2 ,
\\
\sdeSsub{2}^2 &\coloneqq \sumoveredges \min\left\{\linftysw{\correctedltwo^{-\half}}{\edgevertexpatch}^2, \frac{\weightmax{\edgevertexpatch}}{\diffusivity}\right\} \edgediam \ltwos{\jump{\timederiv{\uh}}}{\edge}^2,
\\
\sdeSsub{3}^2 &\coloneqq \sumoveredges \weightmax{\edgevertexpatch} \edgediam \ltwos{\jump{\uh}}{\edge}^2.
\end{align*}
\end{theorem}
\begin{proof}
We begin by observing that $\u$ satisfies
\begin{equation*}
\innerprod{\timederiv{\u}}{\weight\test} + \bilinearareac{\u}{\weight\test} = \innerprod{\heating + \addreac\u}{\weight\test} \quad \forall\test\in\Honenonhomzero,
\end{equation*}
so, upon rearrangement and recalling \eqref{eqn:semi_ellip_recon}, we can show that 
\begin{equation*}
\innerprod{\timederiv{\error}}{\weight\test} + \bilinearareac{\confdiff}{\weight\test} = \innerprod{\addreac\error}{\weight\test}  \quad \forall\test\in\Honenonhomzero.
\end{equation*}
Testing with $\test=\errorconf$, and noting that $\error=\errorconf-\uh[d]$ and $\confdiff=\errorconf-\dgdiffconf$, gives
\begin{equation*}
\innerprod{\timederiv{\errorconf}}{\weight\errorconf} + \bilinearareac{\errorconf}{\weight\errorconf} = \innerprod{\timederiv{\uh[d]}}{\weight\errorconf} + \bilinearareac{\dgdiffconf}{\weight\errorconf} + \innerprod{\addreac\error}{\weight\errorconf}  .
\end{equation*}
In the following, we note that in the case of constant $\convgrad$ and $\addreac=0$, we  have $\correctedltwo=0$. In this case, the result carries through in the natural way, 
resulting in a bound on the quantity
\begin{equation*}
\ltwosw{\error}{\extendedlinftyspace{\starttime}{\timevar}{\ltwospace}}^2 + \int_\starttime^{\timevar}  \ndgp{\error}^2 \ds,
\end{equation*}
with the $\ndgp{\cdot}$ norm containing only an $H^1$ term.

By the Cauchy-Schwarz inequality, Poincare-Friedrichs inequality, and the coercivity and continuity of $\bilinearareac{\cdot}{\cdot}$ from Lemma~\ref{lemma:A_coercivity},
\begin{multline*}
\timederiv{\left(\ltwow{\errorconf}^2\right)} + \ndgp{\errorconf}^2 
\lesssim 
\min\left\{\ltwow{\correctedltwo^{-\half}\timederiv{\left(\uh[d]\right)}}, \ltwow{\diffusivity^{-\half}\timederiv{\left(\uh[d]\right)}} \right\}\ndgp{\errorconf}\\
+ \left( \ndgp{\dgdiffconf} + \ndgpa{\dgdiffconf} \right)\ndgp{\errorconf} 
+ \ltwow{\frac{\addreac}{\sqrt{\correctedltwo}}\error}\ltwow{\sqrt{\correctedltwo}\errorconf}.
\end{multline*}
Using Young's inequality, we arrive to
\begin{align*}
\timederiv{\left(\ltwow{\errorconf}^2\right)} + \ndgp{\errorconf}^2 
%&\lesssim 
%\min\left\{\ltwow{\correctedltwo^{-\half}\timederiv{\left(\uh[d]\right)}}, \ltwow{\diffusivity^{-\half}\timederiv{\left(\uh[d]\right)}} \right\}\ndgp{\errorconf}\\
%&\phantom{\lesssim}+ \half\left( \ndgp{\dgdiffconf} + \ndgpa{\dgdiffconf} \right)^2
%+ \ltwow{\frac{\addreac}{\sqrt{\correctedltwo}}\error}\ltwow{\sqrt{\correctedltwo}\errorconf}\\
%&\lesssim 
%\left( \ndgp{\dgdiffconf} + \ndgpa{\dgdiffconf} \right)^2
%+ \min\left\{\ltwow{\frac{1}{\sqrt{\correctedltwo}}\timederiv{\left(\uh[d]\right)}}, \ltwow{\frac{1}{\sqrt{\diffusivity}}\timederiv{\left(\uh[d]\right)}} \right\}^2 
%\\&\phantom{\lesssim}
%+ \ltwow{\frac{\addreac}{\sqrt{\correctedltwo}}\error}^2 
%+ \ltwow{\sqrt{\correctedltwo}\errorconf}^2\\
&\lesssim 
\left( \ndgp{\dgdiffconf} + \ndgpa{\dgdiffconf} \right)^2\\
&\quad +\min\left\{\ltwow{\correctedltwo^{-\half}\timederiv{\left(\uh[d]\right)}}, \ltwow{\diffusivity^{-\half}\timederiv{\left(\uh[d]\right)}} \right\}^2
%\\&\phantom{\lesssim}
+ \ltwow{\frac{\addreac}{\sqrt{\correctedltwo}}\error}^2.
\end{align*}
Thus, by the triangle inequality,
\begin{align*}
\timederiv{\left(\ltwow{\error}^2\right)} + \ndgp{\error}^2 
&\lesssim 
\left( \ndgp{\dgdiff} + \ndgpa{\dgdiff} \right)^2
+\min\left\{\ltwow{\correctedltwo^{-\half}\timederiv{\left(\uh[d]\right)}}, \ltwow{\diffusivity^{-\half}\timederiv{\left(\uh[d]\right)}} \right\}^2\\
&\phantom{\lesssim} + \ltwow{\frac{\addreac}{\sqrt{\correctedltwo}}\error}^2 + \timederiv{\left(\ltwow{\uh[d]}^2\right)}
+ \ndgp{\uh[d]}^2
+ \ndgpa{\uh[d]}^2.
\end{align*}
Using Gr\"onwall's Lemma (see, e.g.,  
\cite[Appendix B, p.624]{evans1998partial} for a convenient reference) %(See Evans ``Partial Differential Equations'' Appendix B, p624),
we have that, for $\timevar\in\timeinterval$,
\begin{align*}
&\ltwosw{\error}{\extendedlinftyspace{\starttime}{\timevar}{\ltwospace}}^2 + \int_\starttime^{\timevar}  \ndgp{\error}^2 \ds\\
&\phantom{+}\lesssim
\exp\left(\int_{\starttime}^{\timevar} \max_{\domain}\frac{\addreac[2]}{\correctedltwo}(s) \ds\right)
\Big( \ltwow{\error(\starttime)}^2 
		+ \int_\starttime^{\timevar} \left( \ndgp{\dgdiff} + \ndgpa{\dgdiff} \right)^2 \ds\\
		&\qquad\qquad\qquad\qquad\qquad\quad+ \int_\starttime^{\timevar} +\min\left\{\ltwow{\correctedltwo^{-\half}\timederiv{\left(\uh[d]\right)}}, \ltwow{\diffusivity^{-\half}\timederiv{\left(\uh[d]\right)}} \right\}^2
		\\&\qquad\qquad\qquad\qquad\qquad\quad+\ltwosw{\uh[d]}{\extendedlinftyspace{\starttime}{\timevar}{\ltwospace}}^2 + \ndgp{\uh[d]}^2 + \ndgpa{\uh[d]}^2\ds 
\Big).
\end{align*}
Finally, using \eqref{eqn:semiidentity}, Theorem~\ref{thm:stationary_estimate}, 
Theorem~\ref{thm:KarakashianPascal}, and Lemma \ref{lemma:bound_nonconforming_norms}, the result follows. \qed
\end{proof}

\section{A posteriori error analysis for the fully-discrete  scheme}\label{sec:apost_fully}

We can now discuss the analogous bound for the fully discrete problem.

Once again, we start by reformulating the fully-descrete problem~\eqref{eqn:convdifffullydiscrete} as a convection-diffusion-reaction problem letting $\uh[n]\in\tdgspace[n]$, $n=0,\ldots,\ntimesteps$, satisfy 
\begin{align}\label{eqn:discrete_defn}
\innerprod{\frac{\uh[n]-\uh[n-1]}{\timestep{n}}}{\testh[n]} 
		+ \bilinearahreac{\uh[n]}{\testh[n]} 
= 
\innerprod{\heating[n] + \addreac[n]\uh[n]}{\testh[n]} \quad \forall \testh[n]\in\tdgspace[n],
\end{align}
with $\uh[0] = \ltwoproj[k]^0\initu$. We note that the dependence of the bilinear form $\bilinearahreac{\cdot}{\cdot} $ on the $n$-th mesh is suppressed for brevity, but it is taken into account in what follows. 
We then define $\Ak[n] \in \tdgspace[n]$, $n\geq 1$ to be the Riesz representer defined as
\begin{equation*}%\label{eqn:Ak_defn}
\innerprod{\Ak[n]}{\testh[n]}= \bilinearahreac{\uh[n]}{\testh[n]}  \quad \forall \testh[n]\in\tdgspace[n]
,
\end{equation*}
noting that, from the method~\eqref{eqn:discrete_defn} it follows that
\begin{equation}
\Ak[n] = \ltwoproj[k]^{n}\left(\heating[n]+\addreac[n]\uh[n]\right) - \left(\uh[n]-\ltwoproj[k]^{n}\uh[n-1]\right)/\timestep{n}.\label{eqn:Ak_equality}
\end{equation}
%where $\ltwoproj[n]$ is the $L^2$-projection into $\pwpolyspace[n]{1}$.
\begin{definition}
The \emph{elliptic reconstruction} of $\uh[n]$, $n=1,\dots, \ntimesteps$, is the unique  $\ellrecon[n]\in\Honenonhomzero$ such that
\begin{equation*}\label{eqn:disc_Ak}
\bilinearareac{\ellrecon[n]}{\test} 
= 
\innerprod{\Ak[n]}{\test} \quad \forall \test\in\Honenonhomzero.
\end{equation*}
\end{definition}

We extend continuously in time the discrete solution $\uh[n]$ by linear interpolation on each time-interval, setting
\begin{equation*}
\uh(\timevar) \coloneqq \lbasis{n}(\timevar)\uh[n] + \lbasis{n-1}(\timevar)\uh[n-1],
\end{equation*}
on each interval $[\timek[n-1],\timek[n]]\ni t$, $n=1,\dots, \ntimesteps$, where $\{\lbasis{n-1},\lbasis{n}\}$ is the standard 
linear Lagrange basis on $[\timek[n-1],\timek[n]]$. We similarly extend
the definition of the elliptic reconstruction $\ellrecon[n]$ linearly and thus, as in the semi-discrete case, we deecompose the error $\error := \u-\uh$ as
%the difference $\dgdiff[n] \coloneqq \ellrecon[n]-\uh[n]$ as defined at all times. Then, the error  $\error = \u-\uh$
\[
\error = \confdiff + \dgdiff \quad \text{ with }\quad \confdiff \coloneqq \u-\semielliprecon, \quad \dgdiff \coloneqq \semielliprecon-\uh.
\]

\begin{theorem}%[An a posteriori error bound on the fully discrete convection-d\nolinebreak iffusion problem]
Let  $\u$ be the solution of  
\eqref{eqn:convdiff}, and $\uh$ its dG approximation satisfying \eqref{eqn:discrete_defn}. 
Then, we have the \aposteriori{}  bound on the error $\error := \u-\uh$:
\begin{align}\label{eqn:fullydiscreteapostbound}
\begin{split}
&\ltwosw{\error}{\extendedlinftyspace{\starttime}{\finaltime}{\ltwospace}}^2 + \int_\starttime^{\finaltime}%^{\timevar}  
\ndgp{\error}^2 \ds\\
&\phantom{+}\lesssim
\exp\left(\int_{\starttime}^{\finaltime}%^{\timevar}
\max_{\domain}\frac{\addreac[2]}{\correctedltwo}(s) \ds\right)\\
&\phantom{+}\qquad\left( \ltwow{\error(\starttime)}^2
		+ \sum_{n=1}^{\ntimesteps}  \int_{\timek[n-1]}^{\timek[n]} \big(\deSsub{1}{n}^2 + \deSsub{1}{n-1}^2 + \deSsub{2}{n}^2 + \deSsub{4}{n}^2\big) \ds \right.\\
		&\phantom{+}\qquad\qquad\qquad+\left. \sum_{n=1}^{\ntimesteps} \int_{\timek[n-1]}^{\timek[n]}  \deTsub{1}{n}^2 +  \deTsub{2}{n}^2 \ds 
		+ \max_{0\leq n\leq\ntimesteps} \deSsub{3}{n}^2 \right),
\end{split}
\end{align}
whereby, for $n\geq 1$, 
\begin{align*}
\deSsub{1}{n}^2 &\coloneqq  \sumovercellsk{n} \cellweight^2 \ltwos{\Ak[n] + \diffusivity\laplacian\uh[n] - \conv[n] \cdot \nabla\uh[n] -\addreac[n]\uh[n]}{\cell}^2
% + \ltwos{\heating[n] - \heatingh[n]}{\cell}^2 + \ltwos{\left(\conv[n] - \convh[n]\right) \cdot \nabla \uh[n]}{\cell}^2 \right) 
+  \sumoverinternaledgesk{n} \edgepatchweight \ltwos{\jump{\diffusivity \nabla \uh[n]}}{\edge}^2 
\\
&\phantom{\coloneqq}+
\sumoveredgesk{n} \bigg(
\frac{\penal\diffusivity}{\edgediam}\bigg(\weightmax{\edgepatch} + \edgepatchvarweight\penal\diffusivity %+ \edgepatchvarweight\diffusivity 
+  \frac{\alpha^2\diffusivity\maxabs[\edge]{\grad\convgrad}^2}{\maxabs[\edge]{\weight}} \max_{\cell\in\edgepatch}\cellweight^2\bigg) +
\edgepatchweight\linftys{ \conv }{\edge}^2 
\\
&\phantom{\coloneqq+\sumoveredges \left(\right.}+ \edgediam\linftysw{\correctedltwo}{\edgevertexpatch}
 +  \frac{\weightmax{\edgevertexpatch}\edgediam}{\diffusivity} \linftys{\conv-\alpha\diffusivity\grad\convgrad}{\edgevertexpatch}^2\bigg) \ltwos{\jump{\uh[n]}}{\edge}^2 ,
\\
\deSsub{2}{n}^2 &\coloneqq \sumovercellskunion{n} \cellweight^2 \ltwos{\left(\identity-\ltwoproj[n]\right) 
																	     \left(\heating[n] + \addreac[n]\uh[n] + \frac{\uh[n-1]}{\timestep{n}}\right)
																	   }{\cell}^2,
\\
\deSsub{3}{n}^2 &\coloneqq \sumoveredgesk{n} \weightmax{\edgevertexpatch} \edgediam \ltwos{\jump{\uh[n]}}{\edge}^2,
\\
\deSsub{4}{n}^2 &\coloneqq \sumoveredgeskunion{n} \min\left\{\linftysw{\correctedltwo^{-\half}}{\edgevertexpatch}^2, \weightmax{\edgevertexpatch}\diffusivity^{-1}\right\} \edgediam \ltwos{\jump{\frac{\uh[n]-\uh[n-1]}{\timestep{n}}}}{\edge}^2,
\\
\deTsub{1}{n}^2 &\coloneqq \sumovercellskunion{n} \diffusivity^{-1} \ltwosw{\lbasis{n}\left(\conv[n]-\conv\right)\uh[n] 
																			 + \lbasis{n-1}\left(\conv[n-1]-\conv\right)\uh[n-1]
																			}{\cell}^2,
\\
\deTsub{2}{n}^2 &\coloneqq \sumovercellskunion{n}  \ltwoswleft{\min\left\{\correctedltwo^{-\half},\diffusivity^{-\half}\right\}\left(\heating -\heating[n] +\addreac\uh-\addreac[n]\uh[n] + \lbasis{n-1}\left(\Ak[n]-\Ak[n-1]\right)\right.} \\
&\qquad\qquad\qquad\quad															\ltwoswright{\left.+\lbasis{n}\longestimatorterm{n}\uh[n] + \lbasis{n-1}\longestimatorterm{n-1}\uh[n-1] \right)
														  }{\cell}^2,
\end{align*}
where $\longestimatorterm{n} \coloneqq \longestimatortermwobrackets{n}.$
\end{theorem}

\begin{proof}
By rearrangement we can show that for $\test\in\Honenonhomzero$ and $\timevar\in(\timek[n-1],\timek[n]]$,
\begin{equation}\label{eqn:discrete_est_error}
\begin{split}
&\innerprod{\timederiv{\error}}{\weight\test} + \bilinearareac{\error}{\weight\test}\\
=&\ 
\innerprod{\timederiv{\u}}{\weight\test} - \innerprod{\timederiv{\uh}}{\weight\test} + \bilinearareac{\u}{\weight\test} - \bilinearareac{\uh}{\weight\test} \\
=&\  \innerprod{\heating -\heating[n] +\addreac\u-\addreac[n]\uh[n]}{\weight\test} + \innerprod{\heating[n] +\addreac[n]\uh[n]- \timederiv{\uh} - \Ak[n]}{\weight\test} \\
	& + \bilinearareac{\dgdiff[n]}{\weight\test} 
%	+ \innerprod{\addreac\error}{\weight\test} + \innerprod{\addreac\uh}{\weight\test}
	+ \bilinearareac{\uh[n]}{\weight\test} - \bilinearareac{\uh}{\weight\test}\\
=&\  \innerprod{\heating[n] + \addreac[n]\uh[n]- \timederiv{\uh} - \Ak[n]}{\weight\test} \\
	&\ + \innerprod{\heating -\heating[n] +\addreac\uh-\addreac[n]\uh[n] + \lbasis{n-1}\left(\Ak[n]-\Ak[n-1]\right)}{\weight\test}\\
	&\ +\big( \lbasis{n}\bilinearareac{\uh[n]}{\weight\test} + \lbasis{n-1}\bilinearareac{\uh[n-1]}{\weight\test} - \bilinearareac{\uh}{\weight\test} \big)\\
    &\ +\big( \lbasis{n}\bilinearareac{\dgdiff[n]}{\weight\test} + \lbasis{n-1}\bilinearareac{\dgdiff[n-1]}{\weight\test} \big)+ \innerprod{\addreac\error}{\weight\test}\\
    =:&\  A_1+A_2+A_3+A_4+ \innerprod{\addreac\error}{\weight\test}.
    \end{split}
\end{equation}

By using \eqref{eqn:Ak_equality} and the property \eqref{eqn:local_ltwo} we have
\begin{align*}
A_1
&= 
\innerprod{\heating[n] +\addreac[n]\uh[n] - \timederiv{\uh} - \Ak[n]}{\left(\identity-\ltwoproj[k]^n\right)(\weight\test)} 
\lesssim
\deSsub{2}{n}%\left( \sumovercells\cellweight^2 \ltwosw{\left(\identity-\ltwoproj[n]\right)\left(\heating^n  +\addreac^{n}\uh[n] + \frac{\uh[n-1]}{\timestep{n}}\right)}{\cell}^2\right)^\half 
\ndgp{\test}.
\end{align*}
Also, we have
\begin{equation*}
\begin{aligned}
A_2+A_3
=  &\ 
\innerprod{\heating -\heating[n] +\addreac\uh-\addreac[n]\uh[n] + \lbasis{n-1}\left(\Ak[n]-\Ak[n-1]\right)}{\weight\test}\\
&\ +
\lbasis{n}\innerprod{\longestimatorterm{n}\uh[n] }{\weight\test}
 +
\lbasis{n-1}\innerprod{\longestimatorterm{n-1}\uh[n-1] }{\weight\test}\\
&\ -\innerprod{\lbasis{n}\left(\conv[n]-\conv\right)\uh[n] 
 + \lbasis{n-1}\left(\conv[n-1]-\conv\right)\uh[n-1]}{\weight\grad\test} \\
 \lesssim&\  \deTsub{2}{n} \ndgp{\test} + \deTsub{1}{n} \ndgp{\test}.
\end{aligned}
\end{equation*}
In a similar fashion to the semi-discrete case, by Lemma \ref{lemma:A_coercivity}, we have 
\begin{align*}
&\lbasis{n}\bilinearareac{\dgdiff[n]}{\weight\test} + \lbasis{n-1}\bilinearareac{\dgdiff[n-1]}{\weight\test}\\
&\lesssim \lbasis{n}^2\left(\ndgp{\dgdiff[n]} + \ndgpa{\dgdiff[n]}\right)^2 + \lbasis{n-1}^2\left(\ndgp{\dgdiff[n-1]} + \ndgpa{\dgdiff[n-1]}\right)^2 + \ndgp{\test}^2\\
&\lesssim \lbasis{n}^2\deSsub{1}{n}^2 + \lbasis{n-1}^2\deSsub{1}{n-1}^2 + \ndgp{\test}^2
. 
\end{align*}
Once again the dG solution $\uh[n]$ may be decomposed into its conforming and 
nonconforming parts, $\uh[n,c]\in\Honenonhomzero\intersect\tdgspace[n]$ and $\uh[n,d]\in\tdgspace[n]$,  with $\uh[n,c] = \KPapprox{\uh[n]  } \in \dgspace[c]$ and $\uh[n,d]=\uh[n]-\uh[n,c]$, respectively. Returning to \eqref{eqn:discrete_est_error}, and testing with $\test=\errorconf$ we have, via Young's inequality,
\begin{equation}\label{eqn:discrete_penultimate}
\begin{split}
\innerprod{\timederiv{\error}}{\weight\errorconf} + \bilinearareac{\error}{\weight\errorconf}
&\lesssim
\lbasis{n}^2\deSsub{1}{n}^2 + \lbasis{n-1}^2\deSsub{1}{n-1}^2 + \deSsub{2}{n}^2 \\
	&\phantom{\lesssim}+ \deTsub{1}{n}^2 + \deTsub{2}{n}^2 
	%\\&\phantom{\lesssim}
	+ \ndgp{\errorconf}^2 + \innerprod{\addreac\error}{\weight\errorconf},
\end{split}
\end{equation}
and, thus,
\begin{equation}\label{eqn:discrete_error_pre_integration}
\begin{split}
\timederiv{\left(\ltwow{\errorconf}^2\right)} + \ndgp{\errorconf}^2
&\lesssim
\lbasis{n}^2\deSsub{1}{n}^2 + \lbasis{n-1}^2\deSsub{1}{n-1}^2 + \deSsub{2}{n}^2 
	+ \deTsub{1}{n}^2 + \deTsub{2}{n}^2 \\
	&%\phantom{\lesssim}
	\hspace{-2.7cm}+ \min\left\{\ltwow{\correctedltwo^{-\half}\timederiv{\left(\uh[d]\right)}}, \ltwow{\diffusivity^{-\half}\timederiv{\left(\uh[d]\right)}} \right\}^2 + \left(\ndgp{\uh[d]}+\ndgpa{\uh[d]}\right)^2
	%\\&\phantom{\lesssim}
	+\ltwow{\frac{\addreac}{\sqrt{\correctedltwo}}\error}^2 \\
&\hspace{-2.7cm}\lesssim
\lbasis{n}^2\deSsub{1}{n}^2 + \lbasis{n-1}^2\deSsub{1}{n-1}^2 + \deSsub{2}{n}^2 
	+ \deTsub{1}{n}^2 + \deTsub{2}{n}^2 
	%\\&\phantom{\lesssim}
	+ \deSsub{4}{n}^2 + \deSsub{1}{n}^2
	%\\&\phantom{\lesssim}
	+ \ltwow{\frac{\addreac}{\sqrt{\correctedltwo}}\error}^2.
\end{split}
\end{equation}
The result now follows by completely analogous argument to the semi-discrete case.\qed
\end{proof}

For simplicity, we stated the above result for the final time $T$, but clearly it applies up to any timestep.

\section{Discussion and implementation of the estimators}\label{sect:relation_to_existing_results}

We continue with a few remarks on the derived \aposteriori{} error estimator and on the tuning of the involved parameters. 

\subsection{Properties of the estimators}

We begin by highlighting the effect that the use of the Gr\"onwall inequality  (cf., proof of Theorem~\ref{thm:semi-discrete-apost}) may have upon the sharpness of the resulting bound and, thus, on the quality of the resulting error bound as an adaptivity indicator. The argument requires the estimation $\ltwow{\frac{\addreac}{\sqrt{\correctedltwo}}\error}\le \|\frac{\addreac}{\sqrt{\correctedltwo}}\|_{\infty} \ltwow{\error}$ and, so, we lose the 
local dependence of the inequality upon ${\addreac}/{\sqrt{\correctedltwo}}$. This may reduce the \emph{local} sharpness of the bound in some cases. 
%This in turn means that the possibility exists that the resulting error indicator may not generate local values that directly correspond in order to the local contribution to the true error. 
However, we argue that the estimator can still be used as an effective error indicator in practice. %The bound \eqref{eqn:discrete_error_pre_integration} is an \apriori~bound on the error before the time integration step,  in terms of a number of terms including the term $\ltwow{\frac{\addreac}{\sqrt{\correctedltwo}}\error}$. 
Indeed, unless this is the dominant term locally, most of the information is encoded in the remaining 
terms whose sum will act as an appropriate adaptivity indicator.
In cases when $|\frac{\addreac}{\sqrt{\correctedltwo}}|\ll\|\frac{\addreac}{\sqrt{\correctedltwo}}\|_{\infty} $ locally,  the adaptivity indicator
will not act in an optimal manner, ranking cells in an order different to their true local contribution to the error. To minimise this effect, it is important to fix judiciously the parameters $\alpha$ and $\addreac$, characterising the magnitude of the weighting function and of the artificial reaction term, respectively.

Lemma \ref{lemma:A_coercivity} implies that $\addreac(\bx)$ is required to be large enough to assert continuity. Since \eqref{eqn:fullydiscreteapostbound} contains an exponential term
of $\max_\domain\left({\addreac[2]}/{\correctedltwo}\right)$, it is of paramount
importance to reduce the value of $\addreac$ wherever possible. Thus, based on~\eqref{eqn:addreacchoice}, the ideal choice is to fix
\begin{equation*}
\addreac(\bx) = \max\left\{0,-2\ltwoweightexplicit(\bx)\right\},
\end{equation*}
to ensure continuity while also minimising the magnitude
of added reaction.

Good choices of $\alpha$ are less clear. %Dimensional analysis shows that $\alpha$ must have units of seconds per metre, but further analysis and calibration of the choice of value has not been  undertaken. However, 
Two main concerns should guide its definition.
Firstly, as above, we wish to reduce the magnitude of $\addreac$ wherever possible. In some circumstances,
a judicious choice of the value of $\alpha$ may lead to the method requiring no $\addreac$ anywhere, in which 
case no exponential term will be incurred; see also the comments below about previous results. 
Secondly, the choice of $\alpha$ affects the weight $\weight$ and, thus, the weighted norm used to derive the error bound. It also affects the value of 
$\correctedltwo$. Through these quantities, an injudicious choice of $\alpha$ may have the undesirable effect
of misleading weighting of the error norm, rendering the resulting estimators not useful for our purposes. For instance, if a very large value of $\alpha$ is used, 
such that the weight $\weight = \exp(-\alpha\convgrad)$ is very small in most areas, and a larger value 
in only a small area, then the resulting norm informs us little about the global behaviour of the solution.
%
%In conclusion,  the optimal choice of $\alpha$ is a non-trivial  problem, and is dependent upon the desired weighted norm which we wish to bound, and upon the behaviour of the given convection field, while the validity of the above analysis relies purely on a choice of constant $\alpha \in \reals^{+}$ and $\addreac$ satisfying  \eqref{eqn:addreacchoice}..
%
%It is, thus, clear that the optimal choice of $\alpha$ is a non-trivial  problem, and is dependent upon the desired weighted norm which we wish to bound, and upon the behaviour of the given convection field.
%
%However, it is also clear that the validity of the results of this chapter rely purely on a choice of constant $\alpha \in \reals^{+}$, and $\addreac$ satisfying  \eqref{eqn:addreacchoice}.

For example, if the field $\conv$ is exactly the curl of another field, i.e., $\grad\convgrad = 0$, then we may choose $\convgrad=0$ and, thus, we have $\weight=1$. That is, we recover the unweighted norm case. Further, we may also fix $\addreac=0$, 
removing the need to employ Gr\"onwall's Lemma, (cf., \eqref{eqn:discrete_penultimate},) and the resulting addition of an 
exponential term. In this case, we recover the bound of \cite{Cangiani2013a}.

On the other hand, consider the case of negative divergence, e.g., suppose $\Omega=\left[0,1\right]^2$ and  $\conv =\left(1,\half-\half y-x\right)^\transpose$. In this case, $\div\conv=-\half$, and so we should have 
little difficulty in deriving a bound as shown in \cite{Schotzau2009}:
since this flow is characterised by $
\conv = \grad\left(x-\frac{y^2}{4}\right) + \curl\left(-x+\frac{x^2}{2} + y\right),
$
we have that 
\begin{equation*}
\ltwoweightexplicit \geq 1-\frac{3}{2}\diffusivity,
\end{equation*}
everywhere in $\Omega$ and, thus,  for small enough $\diffusivity$, we can again fix $\addreac=0$, that is no  artificial reaction term is required. Note, however, that we are still deriving an error bound in a weighted dG norm, with 
$
\weight=\exp\left(-\alpha\left(x-\frac{y^2}{4}\right)\right).
$
Hence, we may view the new bound as an alternative to that proven in \cite{Cangiani2013a}.

Finally, for convection fields for which the introduction of the weighted norm is not sufficient, such as in presence of positive divergence, we can  add enough reaction locally to ensure coercivity and thus obtain an\aposteriori{} error estimator for a regime out of reach for standard approaches. 

Concluding, the above analysis improves upon and refines known results,
while offering the possibility of reduced dependence upon the \emph{worst case} Gr\"onwall
constant for a number of relevant scenarios.

%minimising the incurred penalty of a Gronwall-type  exponential factor where possible.

\subsection{Implementation considerations}\label{sec:imple}

We comment on the practical implementation of the terms composing the \aposteriori{} \emph{error estimate}
\eqref{eqn:fullydiscreteapostbound} as local error indicators within a mesh adaptive algorithm. 

In view of the following application to a coupled problem whereby the convective field is also approximated numerically, we assume that such field is a discrete function with respect to the same mesh and time-steps used for the computation of $\uh$. Hence, we consider the  solution pair $(\uh[n],\convh[n])$ to be defined on the triangulation $\tria[n]$, for $n=0,1,\ldots,\ntimesteps$.  

While most terms involved are standard and are computable (up to an approximation for patchwise-defined quantities) from the solution pair $(\uh[n],\convh[n])$, some, less standard, terms require special considerations. We refer specifically to the assembly of $\grad\convgrad$ and $\weight$, arising by the use of the Helmholtz decomposition, and the integration-in-time of quantities that are nonlinear or non-polynomial in time, e.g., the weighting function $\weight$.

The computation of the weighting function $\weight[n] = \exp(-\alpha\convgrad[n])$ at each time-step requires the evaluation of the function $\convgrad[n]$ from the Helmholtz decomposition $\convh[n] =  \grad\convgrad[n]+\curl\convcurl[n]$. 
Since $\div\curl\convcurl[n] = 0$, $\convgrad[n]$ satisfies 
$%\begin{equation*}
\div\convh[n] = \laplacian\convgradhat[n]
$. %\end{equation*}
 Thus, we are able to compute the approximate field $\convgradh[n]$ by solving the FEM problem:
find $\convgradh[n]\in\convgradspace[n]$ such that 
\begin{equation}\label{eq:auxiliary}
\innerprod{\grad\convgradh[n]}{\grad\testh[n]}=\innerprod{\div\convh[n]}{\testh[n]}\qquad \forall\testh[n]\in\pwpolyspace[n]{k},
\end{equation}
using the standard, continuous finite element spaces
\[
\pwpolyspace[n]{k}	\coloneqq %\pwpolyspace[n]{k}
\dgspace[k](\tria[n])\intersect\Czerospace[\domain],\quad \convgradspace[n] \coloneqq \pwpolyspace[n]{k} \intersect\left\{\testh\in\ltwospace \suchthat \testh\evalat{\domainbdy} = 0\right\},
\]
with $k$ the polynomial degree of the velocity field. Thus, the evaluation of the weighting function requires the solution of the auxiliary problem~\eqref{eq:auxiliary} at each time-step, which allows to compute, at least approximately, $\weight[n]$ and $\correctedltwo^n$.

Another difficulty in the evaluation of the estimator~\eqref{eqn:fullydiscreteapostbound} is the computation of maxima over patches for the terms $\weightmax{\edgevertexpatch}$, 
$\linftys{\conv[n]-\alpha^n\diffusivity\grad\convgrad[n]}{\edgevertexpatch}^2$, 
and $\linftysw{\correctedltwo^n}{\edgevertexpatch}$ in $\deSsub{1}{n}$,
$\weightmax{\edgevertexpatch}$ in $\deSsub{3}{n}$,
and $\linftysw{\correctedltwo^{-\half}}{\edgevertexpatch}$, $\weightmax{\edgevertexpatch}$ 
in $\deSsub{4}{n}$. 
Each of these requires the calculation of a maximum over 
$\edgevertexpatch$%, namely the set of all cells that share at least a vertex with the edge $\edge$
. 
However, typical discontinuous Galerkin assembly  works by iterating over all cells, and all 
faces of each cell,  hence, the knowledge of vertex-neighbours is not immediately available. A simple solution is to 
approximate this quantity by computing instead the maximum over the edge
patch $\edgepatch \subset \edgevertexpatch$ comprising only the two cells sharing $\edge$ as an edge.

A second approximation is required to simplify integration in time of the non-polynomial functions appearing, for instance, in term $\deSsub{2}{n}$.
The cell weight $\cellweight^2$ featuring therein is varying in time, cf. \eqref{eqn:cell_and_edge_weights},
% since
%\begin{equation*}
%\cellweight =
%\frac{1}{\sqrt{\weightmin{\cell}}}\min\left\{ \frac{\weightmax{\cell}}{\sqrt{\correctedltwomin{\cell}}}, \celldiam\max\left\{\frac{\gradweightmax{\cell}}{\sqrt{\correctedltwomin{\cell}}},\frac{\weightmax{\cell}}
%{\sqrt{\diffusivity}}
%\right\} \right\},
%\end{equation*}
and, due to the presence of the exponential function in the weight $\weight$, it is, in general, non-polynomial.
Nevertheless, even if its exact integration is often unavailable, it is typically smoothly varying and, thus, not challenging. We take different approaches 
to computing this quantity in the terms $\deSsub{1}{n}^2$ and $\deSsub{2}{n}^2$ for 
simplicity of implementation. Since $\deSsub{1}{n}^2$ is defined on a single mesh,
we evaluate this term only at the end of the time interval. In contrast, 
for the term $\deSsub{2}{n}^2$, the implementation has access to the union mesh, 
and the values of the necessary quantities at both ends of each time interval. 
As such, in $\deSsub{2}{n}^2$ we can take the approximation that 
\begin{equation*}
\int_{\timek[n-1]}^{\timek[n]} \cellweight^2 \ds \approx \timestep{n} \max\left\{ \cellweight^2\evalat{\timek[n-1]},\cellweight^2\evalat{\timek[n]}\right\},
\end{equation*}
with little extra effort. The coefficient $\linftysw{\correctedltwo^{-\half}}{\edgevertexpatch}^2$ 
in $\deSsub{4}{n}^2$ can be treated completely analogously.

Also, the evaluation of the estimator terms $\deSsub{2}{n}^2$, $\deSsub{4}{n}^2$, $\deTsub{1}{n}^2$, and $\deTsub{2}{n}^2$ requires projection, viz.,  $\ltwoproj[k]^n\uh[n-1]$. This can be conveniently computed by forming the \emph{union mesh} $\tria[n-1]\union\tria[n]$. However, keeping in memory three different meshes, $\tria[n-1]$, $\tria[n]$ and $\tria[n-1]\union\tria[n]$, can be challenging for large scale problems. To avoid this, we proceed as follows.
The union mesh $\tria[n-1]\union\tria[n]$ is exactly the mesh generated by  \emph{only} applying the modification operations required to move from $\tria[n-1]$ to $\tria[n]$. Thus, instead of making a copy of the triangulation at each timestep, we keep an \emph{auxiliary  triangulation} $\secondtria[n]$ throughout the simulation which follows the main triangulation. By saving  and re-using the refinement and coarsening flags used on the main triangulation, we can ensure that the auxiliary triangulation follows exactly the same pattern of refinement and coarsening as the main triangulation, but at a delayed time in the simulation process. This is implemented as follows. First, the auxiliary triangulation $\secondtria[n-1]$ is held in the unadapted state while  the main triangulation is adapted.  Then, we apply only the refinement process to $\secondtria[n-1]$, yielding $\secondtria[n-\half]$. Note that this may not be exactly the union triangulation, as in principle a cell may be refined and then its children be coarsened during the same step. However, $\secondtria[n-\half]$ is at least as refined as the union mesh. Thus, interpolation to $\secondtria[n-\half]$ of all the finite element functions from $\secondtria[n-1]$ amounts to the identity operator. After the estimator is computed in this way, the new auxiliary mesh is updated as $\secondtria[n]=\tria[n]$ and the adaptive step is complete. The above process results to only two meshes required to be stored at any one time, at the expense of a slight modification of the projection operation given that we project over $\secondtria[n-\half]$ rather than $\tria[n]$ and, as noted above, these meshes may differ slightly.

\section{Numerical experiments}\label{sec:numerics}

%\subsection{Experimental analysis of the estimator on uniform meshes}
We examine the behaviour of the full error estimate \eqref{eqn:fullydiscreteapostbound} on the convection-diffusion problem~\eqref{eqn:convdiff}-\eqref{eqn:convdiffinit} with prescribed convection. In the following examples, 
the initial temperature field is given by 
\begin{equation*}
\u_0(x,y)=1 - (1-y+0.15\sin(4\pi x)\sin(2\pi y)),
\end{equation*}
on a box domain $\domain = \left[0,1\right]^2$, with Dirichlet boundary conditions 
enforced on all boundaries, with values compatible with the initial temperature 
field. The diffusion is constant, $\diffusivity=$\num{1e-6}, and a uniform mesh 
is used.

In the following, we repeatedly make use of the shorthand for $z$-independent vector fields, that is, we may denote a vector field of the form
$\twovectorflow \coloneqq \left(0, 0, g(x, y)\right)^\transpose$, where $g(x, y)$ is constant
in the $z$-direction, by $g(x, y)$.
Further, we use the notation 
\begin{equation*}
\de{S,k}{}{}^2 \coloneqq \sum_{n=1}^{k} \int_{\timek[n-1]}^{\timek[n]} \left( \deSsub{1}{n}^2 + \deSsub{1}{n-1}^2 + \deSsub{2}{n}^2 + \deSsub{4}{n}^2 \right) \ds + \max_{0\leq n\leq\ntimesteps} \deSsub{3}{n}^2,
\end{equation*} 
and 
\begin{equation*}
\de{T,k}{}{}^2 \coloneqq \sum_{n=1}^{k} \int_{\timek[n-1]}^{\timek[n]}  \deTsub{1}{n}^2 +  \deTsub{2}{n}^2 \ds,
\end{equation*}
to refer to the full spatial estimate, and time estimate, respectively.
Furthermore, we use the notation $\de{k}{}{}^2$ to refer to the full
on the right-hand side of \eqref{eqn:fullydiscreteapostbound}, 
excluding the initial discretisation error $\ltwow{\error(\starttime)}^2$.

We consider different cases, depending on the flow field $\conv$ with different characteristics. In each case, we report the value of the leading terms in the estimator at each time-step and the time accumulation of the space, time, and full error estimators  $\de{S,k}{}{}$, $\de{T,k}{}{}$, and $\de{k}{}{}$, respectively.

\subsection*{Case 1.} 
%This data comes from Data/case1a_5
We impose the divergence-free flow $\conv = \curl\convcurl$, where $\convcurl = \frac{x^2+y^2}{2}$. Thus, $\conv = (y,-x)^\top$ and $\convgrad=0$. In this case, the weight $\weight$ is equal to 1 and we recover an un-weighted dG norm. 
Under these circumstances, we have $\correctedltwo=\addreac$, and so we may choose
$\addreac=0$ to remove the exponential term in the estimator, but have only an $H^1$-seminorm bound.
\begin{figure}
%\centering
\hspace{0.7cm}
%\input{case1a_5.tex}
% GNUPLOT: LaTeX picture with Postscript
\begingroup
  \makeatletter
  \providecommand\color[2][]{%
    \GenericError{(gnuplot) \space\space\space\@spaces}{%
      Package color not loaded in conjunction with
      terminal option `colourtext'%
    }{See the gnuplot documentation for explanation.%
    }{Either use 'blacktext' in gnuplot or load the package
      color.sty in LaTeX.}%
    \renewcommand\color[2][]{}%
  }%
  \providecommand\includegraphics[2][]{%
    \GenericError{(gnuplot) \space\space\space\@spaces}{%
      Package graphicx or graphics not loaded%
    }{See the gnuplot documentation for explanation.%
    }{The gnuplot epslatex terminal needs graphicx.sty or graphics.sty.}%
    \renewcommand\includegraphics[2][]{}%
  }%
  \providecommand\rotatebox[2]{#2}%
  \@ifundefined{ifGPcolor}{%
    \newif\ifGPcolor
    \GPcolorfalse
  }{}%
  \@ifundefined{ifGPblacktext}{%
    \newif\ifGPblacktext
    \GPblacktexttrue
  }{}%
  % define a \g@addto@macro without @ in the name:
  \let\gplgaddtomacro\g@addto@macro
  % define empty templates for all commands taking text:
  \gdef\gplbacktext{}%
  \gdef\gplfronttext{}%
  \makeatother
  \ifGPblacktext
    % no textcolor at all
    \def\colorrgb#1{}%
    \def\colorgray#1{}%
  \else
    % gray or color?
    \ifGPcolor
      \def\colorrgb#1{\color[rgb]{#1}}%
      \def\colorgray#1{\color[gray]{#1}}%
      \expandafter\def\csname LTw\endcsname{\color{white}}%
      \expandafter\def\csname LTb\endcsname{\color{black}}%
      \expandafter\def\csname LTa\endcsname{\color{black}}%
      \expandafter\def\csname LT0\endcsname{\color[rgb]{1,0,0}}%
      \expandafter\def\csname LT1\endcsname{\color[rgb]{0,1,0}}%
      \expandafter\def\csname LT2\endcsname{\color[rgb]{0,0,1}}%
      \expandafter\def\csname LT3\endcsname{\color[rgb]{1,0,1}}%
      \expandafter\def\csname LT4\endcsname{\color[rgb]{0,1,1}}%
      \expandafter\def\csname LT5\endcsname{\color[rgb]{1,1,0}}%
      \expandafter\def\csname LT6\endcsname{\color[rgb]{0,0,0}}%
      \expandafter\def\csname LT7\endcsname{\color[rgb]{1,0.3,0}}%
      \expandafter\def\csname LT8\endcsname{\color[rgb]{0.5,0.5,0.5}}%
    \else
      % gray
      \def\colorrgb#1{\color{black}}%
      \def\colorgray#1{\color[gray]{#1}}%
      \expandafter\def\csname LTw\endcsname{\color{white}}%
      \expandafter\def\csname LTb\endcsname{\color{black}}%
      \expandafter\def\csname LTa\endcsname{\color{black}}%
      \expandafter\def\csname LT0\endcsname{\color{black}}%
      \expandafter\def\csname LT1\endcsname{\color{black}}%
      \expandafter\def\csname LT2\endcsname{\color{black}}%
      \expandafter\def\csname LT3\endcsname{\color{black}}%
      \expandafter\def\csname LT4\endcsname{\color{black}}%
      \expandafter\def\csname LT5\endcsname{\color{black}}%
      \expandafter\def\csname LT6\endcsname{\color{black}}%
      \expandafter\def\csname LT7\endcsname{\color{black}}%
      \expandafter\def\csname LT8\endcsname{\color{black}}%
    \fi
  \fi
    \setlength{\unitlength}{0.0500bp}%
    \ifx\gptboxheight\undefined%
      \newlength{\gptboxheight}%
      \newlength{\gptboxwidth}%
      \newsavebox{\gptboxtext}%
    \fi%
    \setlength{\fboxrule}{0.5pt}%
    \setlength{\fboxsep}{1pt}%
\begin{picture}(8640.00,4320.00)%
    \gplgaddtomacro\gplbacktext{%
      \csname LTb\endcsname%%
      \put(300,864){\makebox(0,0)[r]{\strut{}$0$}}%
      \put(300,1469){\makebox(0,0)[r]{\strut{}$5$}}%
      \put(300,2073){\makebox(0,0)[r]{\strut{}$10$}}%
      \put(300,2678){\makebox(0,0)[r]{\strut{}$15$}}%
      \put(335,3282){\makebox(0,0)[r]{\strut{}$\times10^{4}$}}%
      \put(432,644){\makebox(0,0){\strut{}$0$}}%
      \put(916,644){\makebox(0,0){\strut{}$0.5$}}%
      \put(1399,644){\makebox(0,0){\strut{}$1$}}%
      \put(1883,644){\makebox(0,0){\strut{}$1.5$}}%
      \put(2366,644){\makebox(0,0){\strut{}$2$}}%
      \put(2850,644){\makebox(0,0){\strut{}$2.5$}}%
    }%
    \gplgaddtomacro\gplfronttext{%
      \csname LTb\endcsname%%
      \put(1995,3007){\makebox(0,0)[r]{\strut{}$\zeta_{S_1,k}$}}%
      \csname LTb\endcsname%%
      \put(1995,2699){\makebox(0,0)[r]{\strut{}$\zeta_{S_4,k}$}}%
      \csname LTb\endcsname%%
      \put(1641,314){\makebox(0,0){\strut{}Time (s)}}%
      \put(1541,3612){\makebox(0,0){\strut{}Leading error estimator terms}}%
    }%
    \gplgaddtomacro\gplbacktext{%
      \csname LTb\endcsname%%
      \put(4015,864){\makebox(0,0)[r]{\strut{}$0$}}%
      \put(4015,1348){\makebox(0,0)[r]{\strut{}$5$}}%
      \put(4015,1831){\makebox(0,0)[r]{\strut{}$10$}}%
      \put(4015,2315){\makebox(0,0)[r]{\strut{}$15$}}%
      \put(4015,2798){\makebox(0,0)[r]{\strut{}$20$}}%
      \put(4050,3282){\makebox(0,0)[r]{\strut{}$\times10^{4}$}}%
      \put(4147,644){\makebox(0,0){\strut{}$0$}}%
      \put(4631,644){\makebox(0,0){\strut{}$0.5$}}%
      \put(5114,644){\makebox(0,0){\strut{}$1$}}%
      \put(5598,644){\makebox(0,0){\strut{}$1.5$}}%
      \put(6081,644){\makebox(0,0){\strut{}$2$}}%
      \put(6565,644){\makebox(0,0){\strut{}$2.5$}}%
    }%
    \gplgaddtomacro\gplfronttext{%
      \csname LTb\endcsname%%
      \put(5710,2161){\makebox(0,0)[r]{\strut{}$\zeta_{S,k}$}}%
      \csname LTb\endcsname%%
      \put(5710,1853){\makebox(0,0)[r]{\strut{}$\zeta_{T,k}$}}%
      \csname LTb\endcsname%%
      \put(5710,1545){\makebox(0,0)[r]{\strut{}$\zeta_{k}$}}%
      \csname LTb\endcsname%%
      \put(5356,314){\makebox(0,0){\strut{}Time (s)}}%
      \put(5258,3612){\makebox(0,0){\strut{}Time-accumulated error estimators}}%
    }%
    \gplbacktext
    \put(0,0){\includegraphics[width={432.00bp},height={216.00bp}]{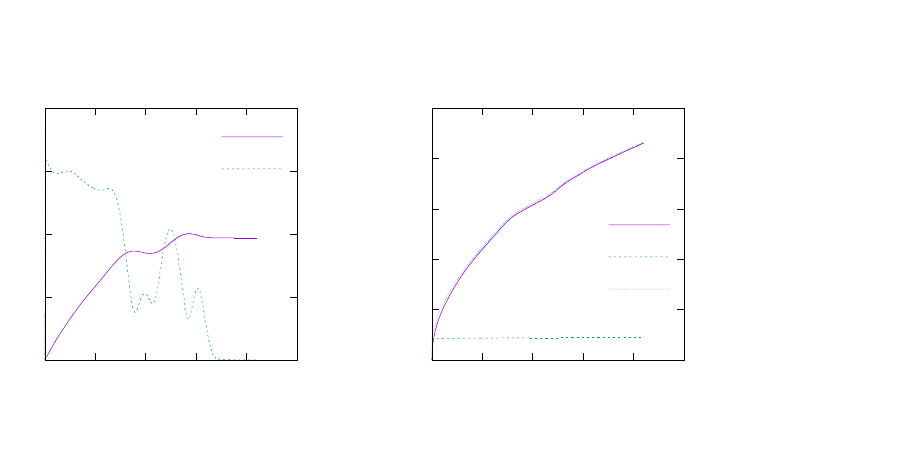}}%
    \gplfronttext
  \end{picture}%
\endgroup
\caption{Estimator terms in Case 1 with $\addreac=0$.}\label{fig:case1a_5}
\end{figure}
Then, we  fix $\addreac=0.1$, resulting in $\correctedltwo=\addreac$.  
In this case, the error estimate has an exponential term of $e^{0.1\finaltime}$ but it
includes the term $0.1\ltwos{\error}{\cell}^2$. 

Figure \ref{fig:case1a_5} and Figure  \ref{fig:case1b_5}  show the results corresponding to $\addreac=0$ and $\addreac=0.1$, respectively. %In both cases, the spatial error estimator is orders of magnitute larger than the time error, with the term $\zeta_{S_4,k}^2$ dominating in the initial phase, before being mostly overtaken by the residual term  $\zeta_{S_1,k}^2$ in the later stages.
The lack of an $L^2$-term when $\addreac=0$ forces the estimator to rely on inequalities related to the diffusion $\diffusivity$. 
This leads to an instant factor of $10^6$ in several estimator terms, and so this 
estimator has a large absolute value, but exhibits only linear growth after $\timevar=1.5$. Indeed, the estimator is initially dominated by the term $\zeta_{S_4,k}$ scaling as $1/\diffusivity$, until this tails off due to a reduction of the solution's jumps across the mesh faces as the solution becomes smoother over time, cf. the left panel in Figure~\ref{fig:case1a_5}.
On the other hand, fixing $\addreac=0.1$ yields control on the full dG norm, including a weighted $L^2$-norm term, and we rely on inequalities involving 
$\correctedltwo=0.1$, leading to a much smaller absolute value for the estimator at small times. Although the term $\zeta_{S_4,k}$ is still dominant in the initial stages, it is much reduced in magnitude, clearly showing that a better balance is obtained between the various controlling mechanisms. The exponential nature of the error bound begins to show at later times. Since the exponent is only $0.1\timevar$, this example exhibits very slow exponential growth, but will eventually overwhelm the estimate in the case $\addreac=0$.

\begin{figure}
%\centering
\hspace{0.5cm}
%\input{case1b_5.tex}
% GNUPLOT: LaTeX picture with Postscript
\begingroup
  \makeatletter
  \providecommand\color[2][]{%
    \GenericError{(gnuplot) \space\space\space\@spaces}{%
      Package color not loaded in conjunction with
      terminal option `colourtext'%
    }{See the gnuplot documentation for explanation.%
    }{Either use 'blacktext' in gnuplot or load the package
      color.sty in LaTeX.}%
    \renewcommand\color[2][]{}%
  }%
  \providecommand\includegraphics[2][]{%
    \GenericError{(gnuplot) \space\space\space\@spaces}{%
      Package graphicx or graphics not loaded%
    }{See the gnuplot documentation for explanation.%
    }{The gnuplot epslatex terminal needs graphicx.sty or graphics.sty.}%
    \renewcommand\includegraphics[2][]{}%
  }%
  \providecommand\rotatebox[2]{#2}%
  \@ifundefined{ifGPcolor}{%
    \newif\ifGPcolor
    \GPcolorfalse
  }{}%
  \@ifundefined{ifGPblacktext}{%
    \newif\ifGPblacktext
    \GPblacktexttrue
  }{}%
  % define a \g@addto@macro without @ in the name:
  \let\gplgaddtomacro\g@addto@macro
  % define empty templates for all commands taking text:
  \gdef\gplbacktext{}%
  \gdef\gplfronttext{}%
  \makeatother
  \ifGPblacktext
    % no textcolor at all
    \def\colorrgb#1{}%
    \def\colorgray#1{}%
  \else
    % gray or color?
    \ifGPcolor
      \def\colorrgb#1{\color[rgb]{#1}}%
      \def\colorgray#1{\color[gray]{#1}}%
      \expandafter\def\csname LTw\endcsname{\color{white}}%
      \expandafter\def\csname LTb\endcsname{\color{black}}%
      \expandafter\def\csname LTa\endcsname{\color{black}}%
      \expandafter\def\csname LT0\endcsname{\color[rgb]{1,0,0}}%
      \expandafter\def\csname LT1\endcsname{\color[rgb]{0,1,0}}%
      \expandafter\def\csname LT2\endcsname{\color[rgb]{0,0,1}}%
      \expandafter\def\csname LT3\endcsname{\color[rgb]{1,0,1}}%
      \expandafter\def\csname LT4\endcsname{\color[rgb]{0,1,1}}%
      \expandafter\def\csname LT5\endcsname{\color[rgb]{1,1,0}}%
      \expandafter\def\csname LT6\endcsname{\color[rgb]{0,0,0}}%
      \expandafter\def\csname LT7\endcsname{\color[rgb]{1,0.3,0}}%
      \expandafter\def\csname LT8\endcsname{\color[rgb]{0.5,0.5,0.5}}%
    \else
      % gray
      \def\colorrgb#1{\color{black}}%
      \def\colorgray#1{\color[gray]{#1}}%
      \expandafter\def\csname LTw\endcsname{\color{white}}%
      \expandafter\def\csname LTb\endcsname{\color{black}}%
      \expandafter\def\csname LTa\endcsname{\color{black}}%
      \expandafter\def\csname LT0\endcsname{\color{black}}%
      \expandafter\def\csname LT1\endcsname{\color{black}}%
      \expandafter\def\csname LT2\endcsname{\color{black}}%
      \expandafter\def\csname LT3\endcsname{\color{black}}%
      \expandafter\def\csname LT4\endcsname{\color{black}}%
      \expandafter\def\csname LT5\endcsname{\color{black}}%
      \expandafter\def\csname LT6\endcsname{\color{black}}%
      \expandafter\def\csname LT7\endcsname{\color{black}}%
      \expandafter\def\csname LT8\endcsname{\color{black}}%
    \fi
  \fi
    \setlength{\unitlength}{0.0500bp}%
    \ifx\gptboxheight\undefined%
      \newlength{\gptboxheight}%
      \newlength{\gptboxwidth}%
      \newsavebox{\gptboxtext}%
    \fi%
    \setlength{\fboxrule}{0.5pt}%
    \setlength{\fboxsep}{1pt}%
\begin{picture}(8640.00,4320.00)%
    \gplgaddtomacro\gplbacktext{%
      \csname LTb\endcsname%%
      \put(300,864){\makebox(0,0)[r]{\strut{}$0$}}%
      \put(300,1469){\makebox(0,0)[r]{\strut{}$50$}}%
      \put(300,2073){\makebox(0,0)[r]{\strut{}$100$}}%
      \put(300,2678){\makebox(0,0)[r]{\strut{}$150$}}%
      \put(300,3282){\makebox(0,0)[r]{\strut{}$200$}}%
      \put(432,644){\makebox(0,0){\strut{}$0$}}%
      \put(916,644){\makebox(0,0){\strut{}$0.5$}}%
      \put(1399,644){\makebox(0,0){\strut{}$1$}}%
      \put(1883,644){\makebox(0,0){\strut{}$1.5$}}%
      \put(2366,644){\makebox(0,0){\strut{}$2$}}%
      \put(2850,644){\makebox(0,0){\strut{}$2.5$}}%
    }%
    \gplgaddtomacro\gplfronttext{%
      \csname LTb\endcsname%%
      \put(1995,3007){\makebox(0,0)[r]{\strut{}$\zeta_{S_1,k}$}}%
      \csname LTb\endcsname%%
      \put(1995,2699){\makebox(0,0)[r]{\strut{}$\zeta_{S_4,k}$}}%
      \csname LTb\endcsname%%
      \put(1641,314){\makebox(0,0){\strut{}Time (s)}}%
      \put(1541,3612){\makebox(0,0){\strut{}Leading error estimator terms}}%
    }%
    \gplgaddtomacro\gplbacktext{%
      \csname LTb\endcsname%%
      \put(4015,864){\makebox(0,0)[r]{\strut{}$0$}}%
      \put(4015,1348){\makebox(0,0)[r]{\strut{}$50$}}%
      \put(4015,1831){\makebox(0,0)[r]{\strut{}$100$}}%
      \put(4015,2315){\makebox(0,0)[r]{\strut{}$150$}}%
      \put(4015,2798){\makebox(0,0)[r]{\strut{}$200$}}%
      \put(4015,3282){\makebox(0,0)[r]{\strut{}$250$}}%
      \put(4147,644){\makebox(0,0){\strut{}$0$}}%
      \put(4631,644){\makebox(0,0){\strut{}$0.5$}}%
      \put(5114,644){\makebox(0,0){\strut{}$1$}}%
      \put(5598,644){\makebox(0,0){\strut{}$1.5$}}%
      \put(6081,644){\makebox(0,0){\strut{}$2$}}%
      \put(6565,644){\makebox(0,0){\strut{}$2.5$}}%
    }%
    \gplgaddtomacro\gplfronttext{%
      \csname LTb\endcsname%%
      \put(5710,2161){\makebox(0,0)[r]{\strut{}$\zeta_{S,k}$}}%
      \csname LTb\endcsname%%
      \put(5710,1853){\makebox(0,0)[r]{\strut{}$\zeta_{T,k}$}}%
      \csname LTb\endcsname%%
      \put(5710,1545){\makebox(0,0)[r]{\strut{}$\zeta_{k}$}}%
      \csname LTb\endcsname%%
      \put(5356,314){\makebox(0,0){\strut{}Time (s)}}%
      \put(5258,3612){\makebox(0,0){\strut{}Time-accumulated error estimators}}%
    }%
    \gplbacktext
    \put(0,0){\includegraphics[width={432.00bp},height={216.00bp}]{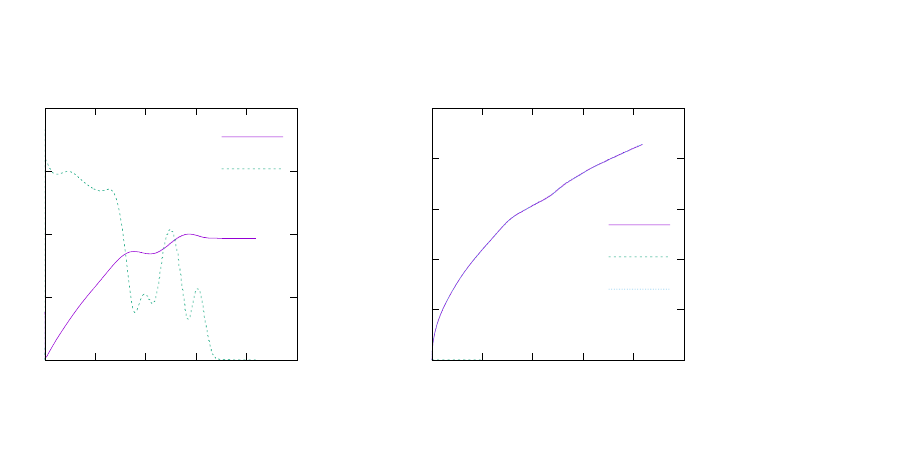}}%
    \gplfronttext
  \end{picture}%
\endgroup
\caption{Estimator terms in Case 1 with $\addreac=0.1$. The time error estimator $\de{T,k}{}{}$ is in this case orders of magnitude smaller than $\de{S,k}{}{}$, hence the latter appears superimposed to the full estimator $\de{k}{}{}$.}\label{fig:case1b_5}
\end{figure}

\subsection*{Case 2.}
%This data comes from Data/case2_5
We now set 
\begin{equation*}
\conv = \left(\begin{array}{c}e^x \sin y + y\\e^x \cos y -x\end{array}\right) = \grad (e^x \sin y) + \curl\frac{x^2+y^2}{2}.
\end{equation*}
This flow field can no longer be characterised as $\conv=\curl\convcurl$, but it is still divergence-free, 
and $\convgrad$ is harmonic but not zero.
Since $\div\conv=0$, then 
\begin{equation*}
\correctedltwo=\addreac + \half\alpha\left(\conv\cdot\grad\convgrad - \alpha\diffusivity \abs*{\grad\convgrad}^2\right) = \addreac + \half\alpha e^x\left(\left(1-\alpha\diffusivity\right)e^x + y \sin y - x \cos y\right).
\end{equation*}
We have $\correctedltwo > \addreac$ in the domain of interest and thus we can choose $\addreac=0$.
This results once again in no exponential term, but we do also have an $L^2$-like term in the norm.
The behaviour of the estimator is shown in Figure \ref{fig:case2_5}. In this case, the residual type term $\zeta_{S_1,k}$ is dominant throughout the computation. Note that solution largely reaches stability by $t=1$ due to the imposed velocity field and, thus, $\zeta_{S_1,k}$ as well as all the contributing factors become near-constant, leading to a linearly-increasing time-integrated error bound thereon.
\begin{figure}
%\centering
\hspace{.7cm}
%\input{case2_5.tex}
% GNUPLOT: LaTeX picture with Postscript
\begingroup
  \makeatletter
  \providecommand\color[2][]{%
    \GenericError{(gnuplot) \space\space\space\@spaces}{%
      Package color not loaded in conjunction with
      terminal option `colourtext'%
    }{See the gnuplot documentation for explanation.%
    }{Either use 'blacktext' in gnuplot or load the package
      color.sty in LaTeX.}%
    \renewcommand\color[2][]{}%
  }%
  \providecommand\includegraphics[2][]{%
    \GenericError{(gnuplot) \space\space\space\@spaces}{%
      Package graphicx or graphics not loaded%
    }{See the gnuplot documentation for explanation.%
    }{The gnuplot epslatex terminal needs graphicx.sty or graphics.sty.}%
    \renewcommand\includegraphics[2][]{}%
  }%
  \providecommand\rotatebox[2]{#2}%
  \@ifundefined{ifGPcolor}{%
    \newif\ifGPcolor
    \GPcolorfalse
  }{}%
  \@ifundefined{ifGPblacktext}{%
    \newif\ifGPblacktext
    \GPblacktexttrue
  }{}%
  % define a \g@addto@macro without @ in the name:
  \let\gplgaddtomacro\g@addto@macro
  % define empty templates for all commands taking text:
  \gdef\gplbacktext{}%
  \gdef\gplfronttext{}%
  \makeatother
  \ifGPblacktext
    % no textcolor at all
    \def\colorrgb#1{}%
    \def\colorgray#1{}%
  \else
    % gray or color?
    \ifGPcolor
      \def\colorrgb#1{\color[rgb]{#1}}%
      \def\colorgray#1{\color[gray]{#1}}%
      \expandafter\def\csname LTw\endcsname{\color{white}}%
      \expandafter\def\csname LTb\endcsname{\color{black}}%
      \expandafter\def\csname LTa\endcsname{\color{black}}%
      \expandafter\def\csname LT0\endcsname{\color[rgb]{1,0,0}}%
      \expandafter\def\csname LT1\endcsname{\color[rgb]{0,1,0}}%
      \expandafter\def\csname LT2\endcsname{\color[rgb]{0,0,1}}%
      \expandafter\def\csname LT3\endcsname{\color[rgb]{1,0,1}}%
      \expandafter\def\csname LT4\endcsname{\color[rgb]{0,1,1}}%
      \expandafter\def\csname LT5\endcsname{\color[rgb]{1,1,0}}%
      \expandafter\def\csname LT6\endcsname{\color[rgb]{0,0,0}}%
      \expandafter\def\csname LT7\endcsname{\color[rgb]{1,0.3,0}}%
      \expandafter\def\csname LT8\endcsname{\color[rgb]{0.5,0.5,0.5}}%
    \else
      % gray
      \def\colorrgb#1{\color{black}}%
      \def\colorgray#1{\color[gray]{#1}}%
      \expandafter\def\csname LTw\endcsname{\color{white}}%
      \expandafter\def\csname LTb\endcsname{\color{black}}%
      \expandafter\def\csname LTa\endcsname{\color{black}}%
      \expandafter\def\csname LT0\endcsname{\color{black}}%
      \expandafter\def\csname LT1\endcsname{\color{black}}%
      \expandafter\def\csname LT2\endcsname{\color{black}}%
      \expandafter\def\csname LT3\endcsname{\color{black}}%
      \expandafter\def\csname LT4\endcsname{\color{black}}%
      \expandafter\def\csname LT5\endcsname{\color{black}}%
      \expandafter\def\csname LT6\endcsname{\color{black}}%
      \expandafter\def\csname LT7\endcsname{\color{black}}%
      \expandafter\def\csname LT8\endcsname{\color{black}}%
    \fi
  \fi
    \setlength{\unitlength}{0.0500bp}%
    \ifx\gptboxheight\undefined%
      \newlength{\gptboxheight}%
      \newlength{\gptboxwidth}%
      \newsavebox{\gptboxtext}%
    \fi%
    \setlength{\fboxrule}{0.5pt}%
    \setlength{\fboxsep}{1pt}%
\begin{picture}(8640.00,4320.00)%
    \gplgaddtomacro\gplbacktext{%
      \csname LTb\endcsname%%
      \put(300,864){\makebox(0,0)[r]{\strut{}$0$}}%
      \put(300,1267){\makebox(0,0)[r]{\strut{}$10$}}%
      \put(300,1670){\makebox(0,0)[r]{\strut{}$20$}}%
      \put(300,2073){\makebox(0,0)[r]{\strut{}$30$}}%
      \put(300,2476){\makebox(0,0)[r]{\strut{}$40$}}%
      \put(300,2879){\makebox(0,0)[r]{\strut{}$50$}}%
      \put(300,3282){\makebox(0,0)[r]{\strut{}$60$}}%
      \put(432,644){\makebox(0,0){\strut{}$0$}}%
      \put(916,644){\makebox(0,0){\strut{}$0.5$}}%
      \put(1399,644){\makebox(0,0){\strut{}$1$}}%
      \put(1883,644){\makebox(0,0){\strut{}$1.5$}}%
      \put(2366,644){\makebox(0,0){\strut{}$2$}}%
      \put(2850,644){\makebox(0,0){\strut{}$2.5$}}%
    }%
    \gplgaddtomacro\gplfronttext{%
      \csname LTb\endcsname%%
      \put(1995,1919){\makebox(0,0)[r]{\strut{}$\zeta_{S_1,k}$}}%
      \csname LTb\endcsname%%
      \put(1641,314){\makebox(0,0){\strut{}Time (s)}}%
      \put(1541,3612){\makebox(0,0){\strut{}Leading error estimator terms}}%
    }%
    \gplgaddtomacro\gplbacktext{%
      \csname LTb\endcsname%%
      \put(4015,864){\makebox(0,0)[r]{\strut{}$0$}}%
      \put(4015,1469){\makebox(0,0)[r]{\strut{}$25$}}%
      \put(4015,2073){\makebox(0,0)[r]{\strut{}$50$}}%
      \put(4015,2678){\makebox(0,0)[r]{\strut{}$75$}}%
      \put(4015,3282){\makebox(0,0)[r]{\strut{}$100$}}%
      \put(4147,644){\makebox(0,0){\strut{}$0$}}%
      \put(4631,644){\makebox(0,0){\strut{}$0.5$}}%
      \put(5114,644){\makebox(0,0){\strut{}$1$}}%
      \put(5598,644){\makebox(0,0){\strut{}$1.5$}}%
      \put(6081,644){\makebox(0,0){\strut{}$2$}}%
      \put(6565,644){\makebox(0,0){\strut{}$2.5$}}%
    }%
    \gplgaddtomacro\gplfronttext{%
      \csname LTb\endcsname%%
      \put(5710,1919){\makebox(0,0)[r]{\strut{}$\zeta_{S,k}$}}%
      \csname LTb\endcsname%%
      \put(5710,1611){\makebox(0,0)[r]{\strut{}$\zeta_{T,k}$}}%
      \csname LTb\endcsname%%
      \put(5710,1303){\makebox(0,0)[r]{\strut{}$\zeta_{k}$}}%
      \csname LTb\endcsname%%
      \put(5356,314){\makebox(0,0){\strut{}Time (s)}}%
      \put(5258,3612){\makebox(0,0){\strut{}Time-accumulated error estimators}}%
    }%
    \gplbacktext
    \put(0,0){\includegraphics[width={432.00bp},height={216.00bp}]{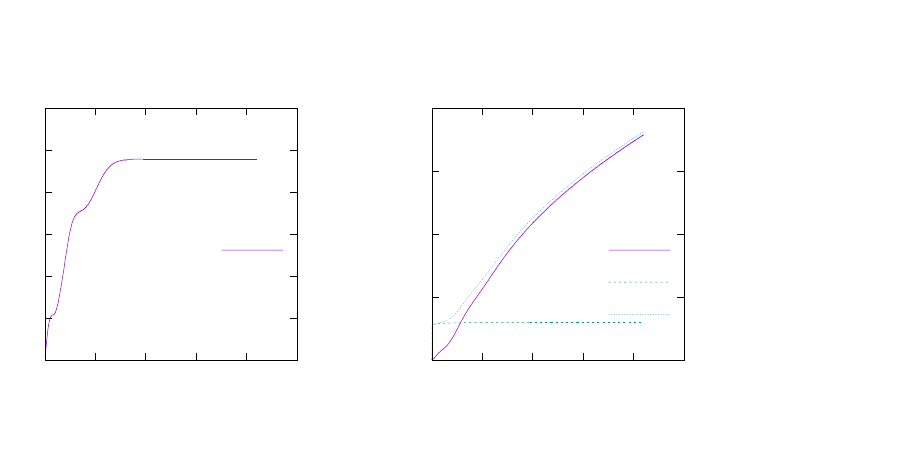}}%
    \gplfronttext
  \end{picture}%
\endgroup
\caption{Estimator terms in Case 2.}\label{fig:case2_5}
\end{figure}

\subsection*{Case 3.}
%this data comes from Data/case3a_5
To consider a case in which the existing literature is not well equipped, we impose the flow 
\begin{equation*}
\conv=\left(\begin{array}{c}x\\y\end{array}\right) = \grad\left(\frac{x^2+y^2}{2}\right),
\end{equation*}
which has positive divergence as $\div\conv\equiv 2$. Then, 
\begin{equation*}
\half\left(\alpha\grad\convgrad-\grad\right)\cdot\left(\conv-\alpha\diffusivity\grad\convgrad\right) = \half\left(1-\alpha\diffusivity\right)\left(\alpha\left(x^2+y^2\right) - 2\right).
\end{equation*}
Thus, we add an artificial reaction term with 
$\addreac = 2\left(1-\alpha\diffusivity\right)\left(2-\alpha\left(x^2+y^2\right)\right)$ 
to satisfy \eqref{eqn:addreacchoice}. 

We consider the two approaches offered by the error estimate. We first take the simple choice of $\alpha=1$. Then the minimal artificial reaction we can impose is
$\addreac = 2\left(1-\diffusivity\right)\left(2-x^2-y^2\right)$.
This leads to an exponential term 
\begin{equation*}
\exp\left(\int_{\starttime}^{\finaltime} \max_{\domain}\frac{\addreac[2]}{\correctedltwo}\dt\right) 
= 
\exp\left(\frac{8}{3}\left(1-\diffusivity\right)\timevar\right),
\end{equation*}
in the error estimator. See Figure \ref{fig:case3_5} for the corresponding results.
~\begin{figure}
%\centering
\hspace{.7cm}
%\input{case3_5.tex}
% GNUPLOT: LaTeX picture with Postscript
\begingroup
  \makeatletter
  \providecommand\color[2][]{%
    \GenericError{(gnuplot) \space\space\space\@spaces}{%
      Package color not loaded in conjunction with
      terminal option `colourtext'%
    }{See the gnuplot documentation for explanation.%
    }{Either use 'blacktext' in gnuplot or load the package
      color.sty in LaTeX.}%
    \renewcommand\color[2][]{}%
  }%
  \providecommand\includegraphics[2][]{%
    \GenericError{(gnuplot) \space\space\space\@spaces}{%
      Package graphicx or graphics not loaded%
    }{See the gnuplot documentation for explanation.%
    }{The gnuplot epslatex terminal needs graphicx.sty or graphics.sty.}%
    \renewcommand\includegraphics[2][]{}%
  }%
  \providecommand\rotatebox[2]{#2}%
  \@ifundefined{ifGPcolor}{%
    \newif\ifGPcolor
    \GPcolorfalse
  }{}%
  \@ifundefined{ifGPblacktext}{%
    \newif\ifGPblacktext
    \GPblacktexttrue
  }{}%
  % define a \g@addto@macro without @ in the name:
  \let\gplgaddtomacro\g@addto@macro
  % define empty templates for all commands taking text:
  \gdef\gplbacktext{}%
  \gdef\gplfronttext{}%
  \makeatother
  \ifGPblacktext
    % no textcolor at all
    \def\colorrgb#1{}%
    \def\colorgray#1{}%
  \else
    % gray or color?
    \ifGPcolor
      \def\colorrgb#1{\color[rgb]{#1}}%
      \def\colorgray#1{\color[gray]{#1}}%
      \expandafter\def\csname LTw\endcsname{\color{white}}%
      \expandafter\def\csname LTb\endcsname{\color{black}}%
      \expandafter\def\csname LTa\endcsname{\color{black}}%
      \expandafter\def\csname LT0\endcsname{\color[rgb]{1,0,0}}%
      \expandafter\def\csname LT1\endcsname{\color[rgb]{0,1,0}}%
      \expandafter\def\csname LT2\endcsname{\color[rgb]{0,0,1}}%
      \expandafter\def\csname LT3\endcsname{\color[rgb]{1,0,1}}%
      \expandafter\def\csname LT4\endcsname{\color[rgb]{0,1,1}}%
      \expandafter\def\csname LT5\endcsname{\color[rgb]{1,1,0}}%
      \expandafter\def\csname LT6\endcsname{\color[rgb]{0,0,0}}%
      \expandafter\def\csname LT7\endcsname{\color[rgb]{1,0.3,0}}%
      \expandafter\def\csname LT8\endcsname{\color[rgb]{0.5,0.5,0.5}}%
    \else
      % gray
      \def\colorrgb#1{\color{black}}%
      \def\colorgray#1{\color[gray]{#1}}%
      \expandafter\def\csname LTw\endcsname{\color{white}}%
      \expandafter\def\csname LTb\endcsname{\color{black}}%
      \expandafter\def\csname LTa\endcsname{\color{black}}%
      \expandafter\def\csname LT0\endcsname{\color{black}}%
      \expandafter\def\csname LT1\endcsname{\color{black}}%
      \expandafter\def\csname LT2\endcsname{\color{black}}%
      \expandafter\def\csname LT3\endcsname{\color{black}}%
      \expandafter\def\csname LT4\endcsname{\color{black}}%
      \expandafter\def\csname LT5\endcsname{\color{black}}%
      \expandafter\def\csname LT6\endcsname{\color{black}}%
      \expandafter\def\csname LT7\endcsname{\color{black}}%
      \expandafter\def\csname LT8\endcsname{\color{black}}%
    \fi
  \fi
    \setlength{\unitlength}{0.0500bp}%
    \ifx\gptboxheight\undefined%
      \newlength{\gptboxheight}%
      \newlength{\gptboxwidth}%
      \newsavebox{\gptboxtext}%
    \fi%
    \setlength{\fboxrule}{0.5pt}%
    \setlength{\fboxsep}{1pt}%
    \definecolor{tbcol}{rgb}{1,1,1}%
\begin{picture}(8640.00,4320.00)%
    \gplgaddtomacro\gplbacktext{%
      \csname LTb\endcsname%%
      \put(300,864){\makebox(0,0)[r]{\strut{}$0$}}%
      \put(300,1267){\makebox(0,0)[r]{\strut{}$20$}}%
      \put(300,1670){\makebox(0,0)[r]{\strut{}$40$}}%
      \put(300,2073){\makebox(0,0)[r]{\strut{}$60$}}%
      \put(300,2476){\makebox(0,0)[r]{\strut{}$80$}}%
      \put(300,2879){\makebox(0,0)[r]{\strut{}$100$}}%
      \put(300,3282){\makebox(0,0)[r]{\strut{}$120$}}%
      \put(432,644){\makebox(0,0){\strut{}$0$}}%
      \put(916,644){\makebox(0,0){\strut{}$0.5$}}%
      \put(1399,644){\makebox(0,0){\strut{}$1$}}%
      \put(1883,644){\makebox(0,0){\strut{}$1.5$}}%
      \put(2366,644){\makebox(0,0){\strut{}$2$}}%
      \put(2850,644){\makebox(0,0){\strut{}$2.5$}}%
    }%
    \gplgaddtomacro\gplfronttext{%
      \csname LTb\endcsname%%
      \put(1995,1516){\makebox(0,0)[r]{\strut{}$\zeta_{S_1,k}$}}%
      \csname LTb\endcsname%%
      \put(1641,314){\makebox(0,0){\strut{}Time (s)}}%
      \put(1450,3612){\makebox(0,0){\strut{}Leading error estimator terms}}%
    }%
    \gplgaddtomacro\gplbacktext{%
      \csname LTb\endcsname%%
      \put(4015,864){\makebox(0,0)[r]{\strut{}$0$}}%
      \put(4015,1209){\makebox(0,0)[r]{\strut{}$25$}}%
      \put(4015,1555){\makebox(0,0)[r]{\strut{}$50$}}%
      \put(4015,1900){\makebox(0,0)[r]{\strut{}$75$}}%
      \put(4015,2246){\makebox(0,0)[r]{\strut{}$100$}}%
      \put(4015,2591){\makebox(0,0)[r]{\strut{}$125$}}%
      \put(4015,2937){\makebox(0,0)[r]{\strut{}$150$}}%
      \put(4015,3282){\makebox(0,0)[r]{\strut{}$175$}}%
      \put(4147,644){\makebox(0,0){\strut{}$0$}}%
      \put(4631,644){\makebox(0,0){\strut{}$0.5$}}%
      \put(5114,644){\makebox(0,0){\strut{}$1$}}%
      \put(5598,644){\makebox(0,0){\strut{}$1.5$}}%
      \put(6081,644){\makebox(0,0){\strut{}$2$}}%
      \put(6565,644){\makebox(0,0){\strut{}$2.5$}}%
    }%
    \gplgaddtomacro\gplfronttext{%
      \csname LTb\endcsname%%
      \put(5710,1884){\makebox(0,0)[r]{\strut{}$\zeta_{S,k}$}}%
      \csname LTb\endcsname%%
      \put(5710,1576){\makebox(0,0)[r]{\strut{}$\zeta_{T,k}$}}%
      \csname LTb\endcsname%%
      %\put(5710,1268){\makebox(0,0)[r]{\strut{}$\zeta_{k}$}}%
      \csname LTb\endcsname%%
      \put(5356,314){\makebox(0,0){\strut{}Time (s)}}%
      \put(5158,3612){\makebox(0,0){\strut{}Time-accumulated error estimators}}%
    }%
    \gplbacktext
    \put(0,0){\includegraphics[width={432.00bp},height={216.00bp}]{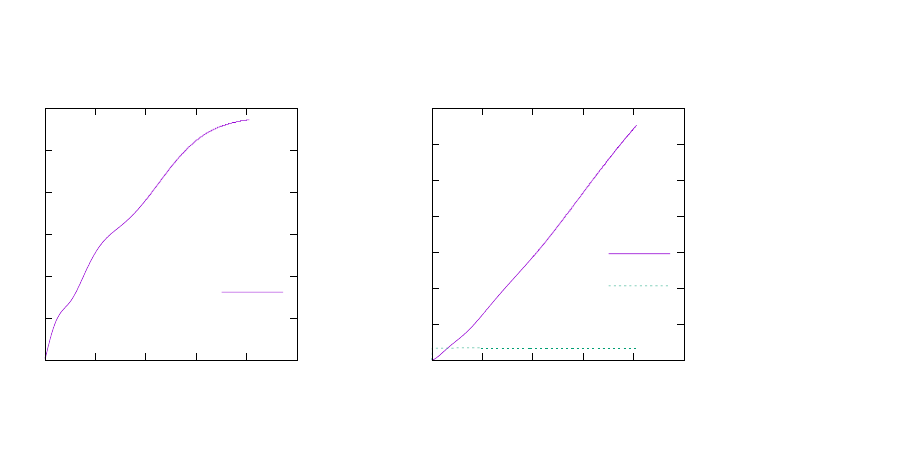}}%
    \gplfronttext
  \end{picture}%
\endgroup
\caption{Estimator terms in Case 3; $\zeta_k$ is not plotted since it grows exponentially.}\label{fig:case3_5}
\end{figure}
The full error bound is not shown in the plot as,  it grows exponentially, becoming too large for double precision arithmetic to represent already at $t=0.5$.

We remark that, if we had not used the exponential fitting technique, then we would have been
required to add enough reaction $\addreac$ to handle $\half\div\conv$, i.e., we would have 
required $\addreac = 4$, leading to an exponential term $\exp\left(\frac{8}{3}\timevar\right)$, 
and so the exponential fitting here has enabled us to slightly reduce the factor in the 
exponential. We note that there exist examples where this difference is
more substantial, particularly when $\conv \neq \grad\convgrad$ and $\div\conv\neq0$.
In this case, we could use the other freedom afforded us by the estimator, and alter the value of $\alpha$ to improve this behaviour. 
However, in our experience this is not usually useful in the case of a small diffusion
coefficient -- to have a measurable effect on the exponential term requires $\alpha$ to be very large and, in particular, to be of order $ \diffusivity^{-1}$.
%, but since the weight $\weight$ depends upon $\alpha$, we experience terms of the type $\alpha \exp\left(-\alpha\right)$, which for large $\alpha$ quickly stretches the ability of double precision arithmetic routines.

\subsection*{Case 4}
%This data comes from Data/case4_5
Finally, we look at the case of a positive-divergence field with a non-zero curl part.
Taking 
\begin{equation*}
\conv = \left(\begin{array}{c}x\\x^2+y^2\end{array}\right) = \grad\left(\frac{x^2}{2} + x^2y\right) + \curl\left(-xy^2\right),
\end{equation*}
and choosing $\alpha=1$, we have that
\begin{multline*}
\half\left(\alpha\grad\convgrad-\grad\right)\cdot\left(\conv-\alpha\diffusivity\grad\convgrad\right) \\
= \half\left(\left(1-\diffusivity\right)\left(x^4+x^2 - 1 -2y\right) + \left(2-4\diffusivity\right)x^2y + \left(1-4\diffusivity\right)x^2y^2\right),
\end{multline*}
and, so, we add reaction 
\begin{equation*}
-2\left(\left(1-\diffusivity\right)\left(x^4+x^2 - 1 -2y\right) - \left(2-4\diffusivity\right)x^2y - \left(1-4\diffusivity\right)x^2y^2\right).
\end{equation*}
This leads to an exponential term of $\exp\left(8\left(1-\diffusivity\right)\timevar\right)$, resulting in the full estimator $\de{k}{}{}^2$ growing exponentially fast. However, the estimator terms discounted by this factor as shown in Figure~\ref{fig:case4_5} give a meaningful representation of the error. 

\begin{figure}
%\centering
\hspace{.7cm}
%\input{case4_5.tex}
% GNUPLOT: LaTeX picture with Postscript
\begingroup
  \makeatletter
  \providecommand\color[2][]{%
    \GenericError{(gnuplot) \space\space\space\@spaces}{%
      Package color not loaded in conjunction with
      terminal option `colourtext'%
    }{See the gnuplot documentation for explanation.%
    }{Either use 'blacktext' in gnuplot or load the package
      color.sty in LaTeX.}%
    \renewcommand\color[2][]{}%
  }%
  \providecommand\includegraphics[2][]{%
    \GenericError{(gnuplot) \space\space\space\@spaces}{%
      Package graphicx or graphics not loaded%
    }{See the gnuplot documentation for explanation.%
    }{The gnuplot epslatex terminal needs graphicx.sty or graphics.sty.}%
    \renewcommand\includegraphics[2][]{}%
  }%
  \providecommand\rotatebox[2]{#2}%
  \@ifundefined{ifGPcolor}{%
    \newif\ifGPcolor
    \GPcolorfalse
  }{}%
  \@ifundefined{ifGPblacktext}{%
    \newif\ifGPblacktext
    \GPblacktexttrue
  }{}%
  % define a \g@addto@macro without @ in the name:
  \let\gplgaddtomacro\g@addto@macro
  % define empty templates for all commands taking text:
  \gdef\gplbacktext{}%
  \gdef\gplfronttext{}%
  \makeatother
  \ifGPblacktext
    % no textcolor at all
    \def\colorrgb#1{}%
    \def\colorgray#1{}%
  \else
    % gray or color?
    \ifGPcolor
      \def\colorrgb#1{\color[rgb]{#1}}%
      \def\colorgray#1{\color[gray]{#1}}%
      \expandafter\def\csname LTw\endcsname{\color{white}}%
      \expandafter\def\csname LTb\endcsname{\color{black}}%
      \expandafter\def\csname LTa\endcsname{\color{black}}%
      \expandafter\def\csname LT0\endcsname{\color[rgb]{1,0,0}}%
      \expandafter\def\csname LT1\endcsname{\color[rgb]{0,1,0}}%
      \expandafter\def\csname LT2\endcsname{\color[rgb]{0,0,1}}%
      \expandafter\def\csname LT3\endcsname{\color[rgb]{1,0,1}}%
      \expandafter\def\csname LT4\endcsname{\color[rgb]{0,1,1}}%
      \expandafter\def\csname LT5\endcsname{\color[rgb]{1,1,0}}%
      \expandafter\def\csname LT6\endcsname{\color[rgb]{0,0,0}}%
      \expandafter\def\csname LT7\endcsname{\color[rgb]{1,0.3,0}}%
      \expandafter\def\csname LT8\endcsname{\color[rgb]{0.5,0.5,0.5}}%
    \else
      % gray
      \def\colorrgb#1{\color{black}}%
      \def\colorgray#1{\color[gray]{#1}}%
      \expandafter\def\csname LTw\endcsname{\color{white}}%
      \expandafter\def\csname LTb\endcsname{\color{black}}%
      \expandafter\def\csname LTa\endcsname{\color{black}}%
      \expandafter\def\csname LT0\endcsname{\color{black}}%
      \expandafter\def\csname LT1\endcsname{\color{black}}%
      \expandafter\def\csname LT2\endcsname{\color{black}}%
      \expandafter\def\csname LT3\endcsname{\color{black}}%
      \expandafter\def\csname LT4\endcsname{\color{black}}%
      \expandafter\def\csname LT5\endcsname{\color{black}}%
      \expandafter\def\csname LT6\endcsname{\color{black}}%
      \expandafter\def\csname LT7\endcsname{\color{black}}%
      \expandafter\def\csname LT8\endcsname{\color{black}}%
    \fi
  \fi
    \setlength{\unitlength}{0.0500bp}%
    \ifx\gptboxheight\undefined%
      \newlength{\gptboxheight}%
      \newlength{\gptboxwidth}%
      \newsavebox{\gptboxtext}%
    \fi%
    \setlength{\fboxrule}{0.5pt}%
    \setlength{\fboxsep}{1pt}%
    \definecolor{tbcol}{rgb}{1,1,1}%
\begin{picture}(8640.00,4320.00)%
    \gplgaddtomacro\gplbacktext{%
      \csname LTb\endcsname%%
      \put(300,864){\makebox(0,0)[r]{\strut{}$0$}}%
      \put(300,1267){\makebox(0,0)[r]{\strut{}$20$}}%
      \put(300,1670){\makebox(0,0)[r]{\strut{}$40$}}%
      \put(300,2073){\makebox(0,0)[r]{\strut{}$60$}}%
      \put(300,2476){\makebox(0,0)[r]{\strut{}$80$}}%
      \put(300,2879){\makebox(0,0)[r]{\strut{}$100$}}%
      \put(300,3282){\makebox(0,0)[r]{\strut{}$120$}}%
      \put(432,644){\makebox(0,0){\strut{}$0$}}%
      \put(916,644){\makebox(0,0){\strut{}$0.5$}}%
      \put(1399,644){\makebox(0,0){\strut{}$1$}}%
      \put(1883,644){\makebox(0,0){\strut{}$1.5$}}%
      \put(2366,644){\makebox(0,0){\strut{}$2$}}%
      \put(2850,644){\makebox(0,0){\strut{}$2.5$}}%
    }%
    \gplgaddtomacro\gplfronttext{%
      \csname LTb\endcsname%%
      \put(1995,1516){\makebox(0,0)[r]{\strut{}$\zeta_{S_1,k}$}}%
      \csname LTb\endcsname%%
      \put(1641,314){\makebox(0,0){\strut{}Time (s)}}%
      \put(1450,3612){\makebox(0,0){\strut{}Leading error estimator terms}}%
    }%
    \gplgaddtomacro\gplbacktext{%
      \csname LTb\endcsname%%
      \put(4015,864){\makebox(0,0)[r]{\strut{}$0$}}%
      \put(4015,1209){\makebox(0,0)[r]{\strut{}$25$}}%
      \put(4015,1555){\makebox(0,0)[r]{\strut{}$50$}}%
      \put(4015,1900){\makebox(0,0)[r]{\strut{}$75$}}%
      \put(4015,2246){\makebox(0,0)[r]{\strut{}$100$}}%
      \put(4015,2591){\makebox(0,0)[r]{\strut{}$125$}}%
      \put(4015,2937){\makebox(0,0)[r]{\strut{}$150$}}%
      \put(4015,3282){\makebox(0,0)[r]{\strut{}$175$}}%
      \put(4147,644){\makebox(0,0){\strut{}$0$}}%
      \put(4631,644){\makebox(0,0){\strut{}$0.5$}}%
      \put(5114,644){\makebox(0,0){\strut{}$1$}}%
      \put(5598,644){\makebox(0,0){\strut{}$1.5$}}%
      \put(6081,644){\makebox(0,0){\strut{}$2$}}%
      \put(6565,644){\makebox(0,0){\strut{}$2.5$}}%
    }%
    \gplgaddtomacro\gplfronttext{%
      \csname LTb\endcsname%%
      \put(5710,2023){\makebox(0,0)[r]{\strut{}$\zeta_{S,k}$}}%
      \csname LTb\endcsname%%
      \put(5710,1715){\makebox(0,0)[r]{\strut{}$\zeta_{T,k}$}}%
      \csname LTb\endcsname%%
      %\put(5710,1407){\makebox(0,0)[r]{\strut{}$\zeta_{k}$}}%
      \csname LTb\endcsname%%
      \put(5356,314){\makebox(0,0){\strut{}Time (s)}}%
      \put(5158,3612){\makebox(0,0){\strut{}Time-accumulated error estimators}}%
    }%
    \gplbacktext
    \put(0,0){\includegraphics[width={432.00bp},height={216.00bp}]{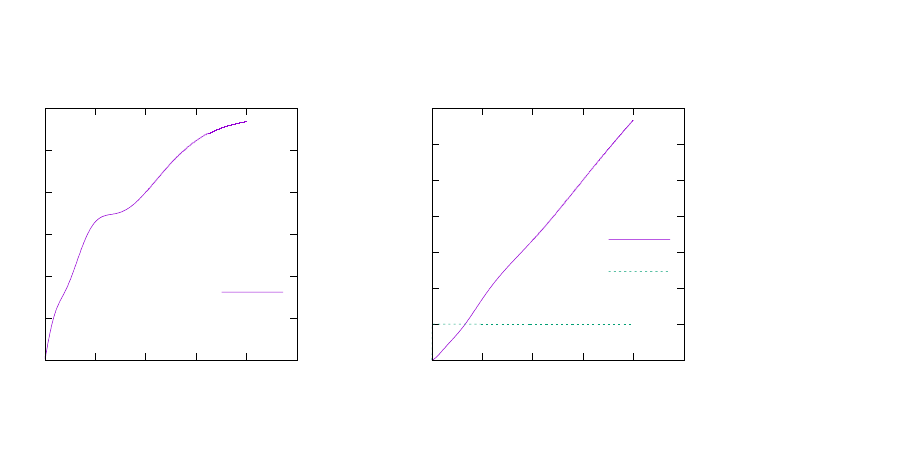}}%
    \gplfronttext
  \end{picture}%
\endgroup
\caption{Estimator terms in Case 4; $\zeta_k$ is not plotted since it grows exponentially.}\label{fig:case4_5}
\end{figure}

\section{The Boussinesq system and mantle convection simulations}\label{sec:Boussinesq}

The study of numerical modelling of mantle convection began in the late 1960s and 
early 1970s, with 2D finite difference codes such as those of Minear and Toks\"oz 
\cite{MinearToksoz1970}, Torrance and Turcotte \cite{Torrance1971}, Mckenzie et 
al. \cite{MckenzieRobertsWeiss1974}, and Schmeling and Jacoby \cite{SchmelingJacoby1981}. 
These approaches typically use the stream function formulation to eliminate the pressure 
from the Navier-Stokes equations and reduce 2D velocity vectors to scalars. More 
recent attempts to use finite differences have used staggered grids, e.g., \cite{gerya2003characteristics}. Spectral methods have been employed in mantle simulations as early as 1974 \cite{Young1974}, 
and enjoyed much popularity during the 1980s and early 1990s for both 3D 
Cartesian and spherical geometries, due to their power in splitting a 3D problem 
into several 1D problems, e.g., \cite{BercoviciSchubertGlatzmeier1989,TackleyStevensonGlatzmaierSchubert1993}. They have since largely fallen out of favour due to difficulties in handling 
large lateral heterogeneities in viscosity. Finite volume methods enjoyed a lot of popularity from the early 1990s, and continue to be used, e.g., the Stag3D code of Tackley \cite{Tackley2008}, 
but not to the same extent as finite element methods.

Finite element methods (FEM) have been used since the early 1980s, often solving 
for a stream function, e.g., \cite{HansenEbel1984}. Most FEM codes now solve instead for the primary variables of temperature, velocity,  and pressure. There are a growing number of codes that are well documented and  have been widely used in the mantle convection modelling community, as well as several newer codes that are relevant to this work. We refer the interested reader to  \cite{May2013} for an excellent discussion of the history of the FEM and the use of mesh adaptivity in geodynamics.

The problem that we consider here is derived from the infinite-Prandtl number limit of the Navier-Stokes equations, 
with the Boussinesq approximation, in which the buoyancy term arises only from the density variations caused by 
temperature variations. It is the most widely used basis model of the dynamics of the mantle temperature, velocity, and pressure.

Given an initial temperature field $\initu(\bx)$
and time- and position-dependent forcing term $\heating(\bx, \timevar)$, 
%viscosity $\viscosity\left(\u,\bx\right)$ and density term $\density(\bx,\test)$, 
find $\u$, $\conv$, and $\pressure$
%$\initu \in \ltwospace$, boundary conditions 
%$\ubdycondition\u = \ubdydata, \convbdycondition\conv = \convbdydata, \pbdycondition\pressure = \pbdydata$, 
%forcing term $\heating(\bx, \timevar) \in \extendedltwospace{\starttime}{\finaltime}{\ltwospace}$, 
%viscosity $\viscosity\left(\u,\bx\right) \in \linftyspace$ and density term $\density(\bx,\test) \in \extendedltwospace{0}{\infty}{\domain}$, 
%find the triple $\left(\u(\bx,\timevar), \conv(\bx,\timevar), \pressure(\bx,\timevar) \right)$ 
such that
\newlength{\myrightlen}
\newcommand{\setmyrightlen}[1]{\settowidth{\myrightlen}{\( \displaystyle
#1\)}}
\newcommand{\backup}{\hskip\myrightlen\mkern-7mu}
\newcommand{\backuptwo}{\hskip\myrightlen\mkern-16mu}
\newcommand{\backupthree}{\hskip\myrightlen\mkern-35mu}
\setmyrightlen{-\density(\u,\bx)\gravity}
\begin{equation}
\begin{split}
\left.
\begin{array}{rl}
\timederiv{\u} - \diffusivity\laplacian\u + \conv\cdot\grad\u 
&= \heating(\bx,\timevar) \\
-\div\left(2\viscosity(\u,\bx)\symgrad{\conv}\right) + \grad\pressure 
&= -\density(\u,\bx)\gravity \\
\div\conv
&= 0
\end{array}
\right\} 
&\text{ in } \domain \times \timeinterval, \label{eqn:boussinesq}\\
\u(\bx,\starttime) \hspace{2mm}
= \initu(\bx)\backuptwo &\text{ in } \domain,\\
\u \hspace{2mm}
= \dirdata(\bx,\timevar) \backupthree
&\text{ on } \dirbdy \times \timeinterval,\\
\diffusivity\deriv{\u}{\normal} \hspace{2mm}
= \neudata(\bx,\timevar) \backupthree
&\text{ on } \neubdy \times \timeinterval,
\\
\left.
\begin{array}{rl}
\conv\cdot\normal
&= 0 \\
\symgrad{\conv}\normal\times\normal 
&= 0
\end{array}\backup
\right\}
&\text{ in } \domainbdy \times \timeinterval, 
\end{split}
\end{equation}
where $\symgrad{\conv} \coloneqq \half(\grad\conv+\grad\conv[\transpose])$ is the symmetric gradient operator. No initial conditions for the velocity are required as the velocity is assumed to be in a static equilibrium with the temperature.

The first equation is the energy equation for the temperature $\u$; the second and third form the Stokes system for the velocity and pressure $(\conv,\pressure)$. 
The system  is driven by the forcing term $\heating=\heating(\bx,\timevar)$ and gravity $\gravity$ and depends on thermal diffusion, viscosity, and density here denoted by $\diffusivity$, $\viscosity$, and $\density$, respectively.
The thermal diffusion $\diffusivity$ is considered to be constant and the viscosity $\viscosity(\u,\cdot) \in \linftyspace$ with a positive minimum 
$\viscosity(\u,\cdot) \geq \minviscosity > 0$. 
%$\symgrad{\conv}$  is the symmetric gradient operator, 
%$\symgrad{\conv} \coloneqq \half(\grad\conv+\grad\conv[\transpose])$,and 
For the  gravity vector $\gravity$ we use $\gravity=9.81e_r$, where 
$e_r$ is the radial unit vector (in the case of annular or shell geometries) or the unit 
downwards vector (in a box geometry). %The temperature- and position-dependent 
%viscosity is denoted by $\viscosity(\u,\bx)$ and density by $\density(\u,\bx)$.
%The precise functional analytic setting for the solution and problem data is 
%discussed below, after the necessary definitions are introduced.

The Stokes system does not necessarily admit a unique solution in the case of a thick-shell domain relevant to the modelling of mantle convection.
Indeed, in this case, 
defining the three rigid body motions $\mathbf{v}^{(i)}, i=1,2,3$ by 
$\mathbf{v}^{(i)}(\bx) \coloneqq \mathbf{e}^{(i)} \times \bx$
where $\mathbf{e}^{(i)}$ is the unit vector in the $i$-th coordinate direction
and $(\conv(\bx),\pressure(\bx))$ a solution at time 
$\timevar\in\timeinterval$, gives us that 
$\conv+\sum_{i=1}^3 c_i \mathbf{v}^{(i)}$ is also a solution, for $c_i \in\reals$.
In addition, the pressure solution is only unique up to an additive constant.

To circumvent this, we introduce three natural spaces for this problem:
\begin{align*}
\velbigspace 
&\coloneqq
\left\{ \mathbf{w}\in\left[\Hone\right]^3 \suchthat \mathbf{w} \cdot\normal=0 \text{ on } \domainbdy \right\},
\\
\velspace 
&\coloneqq 
\left\{ \mathbf{w}\in\velbigspace \suchthat (\mathbf{w},\mathbf{v}^{(i)}) = 0 \text{ for } i=1,2,3\right\},
\\
\pressurespace &\coloneqq \left\{ q\in\ltwospace \suchthat (q,1) = 0 \right\},
\end{align*}
 we define the bilinear forms
\begin{align}
\left.
\begin{array}{rl}
\bilinearstokes{\conv}{\btest} 
&\coloneqq 
\innerprod{2\viscosity\left(\u,\bx\right)\symgrad{\conv}}{\symgrad{\btest}},\\
\bilinearpressure{\btest}{\pressure} 
&\coloneqq 
-\innerprod{\div\btest}{\pressure},
\end{array} \right.\label{eqn:stokesbilinear}
\end{align}
and we consider the weak formulation of the Stokes system: 
find $\conv\in\velspace$, $\pressure\in\pressurespace$, such that
\begin{align}
\begin{array}{rl}
\bilinearstokes{\conv}{\btest} + \bilinearpressure{\btest}{\pressure}
&= -\left(\density(\u,\bx)\gravity,\btest\right) \\
\bilinearpressure{\conv}{q} &= 0,
\end{array}
\label{eqn:stokesweak}
\end{align}
for all $(\btest,q) \in \velspace\times\pressurespace$. We have the following result from \cite[Lemma 1]{Tabata2002}.
\begin{lemma}\label{exist_bous}%[Well-posedness of stationary Stokes system]
Let $\domain$ be a spherical domain 
$\domain~=~\left\{\bx\in\domain \suchthat R_1 < \abs{\bx} < R_2\right\}$, 
and suppose that 
\begin{align*}
\density(\u,\bx)\gravity \in \left[\ltwospace\right]^3, \quad
\viscosity \in \linftyspace, \quad
\viscosity(\u,\bx) \geq \minviscosity > 0.
\end{align*}
Then, \eqref{eqn:stokesweak} has a unique solution in $\velspace \times \pressurespace$.\qed
\end{lemma}

We 
introduce the weak form of the full system \eqref{eqn:boussinesq}:
for each $\timevar\in\timeinterval$, find 
$(\u,\conv,\pressure) \in \Hone \times \velspace \times \pressurespace$ 
such that
\begin{equation}
\begin{array}{rl}
\left(\dtimederiv{\u}, \test\right) + \bilineara{\u}{\test} &= \linearl{\test}\\
\bilinearstokes{\conv}{\btest} + \bilinearpressure{\btest}{\pressure}
&= -\left(\density(\u,\bx)\gravity,\btest\right) \\
\bilinearpressure{\conv}{q} &= 0 \\
\u\evalat{\dirbdy} &= \dirdata \\
\u(\bx,0) &= \u[0](\bx), 
\end{array}
\label{eqn:bousweak}
\end{equation}
for all $(\test, \btest, q) \in \Hone \times \velspace \times \pressurespace$, 
where we note the implicit dependence of the bilinear form 
$\bilineara{\cdot}{\cdot}$ upon the convection variable $\conv(\bx,\timevar)$.

Once again, \cite[Theorem 3]{Tabata2002} shows the well-posedness of this system
on a spherical domain, under certain conditions. 
%We use the notation $\closure{\domain}$ to denote the closure of $\domain$.

\begin{lemma} With the notation of Lemma \ref{exist_bous}, let $\viscosity : {\rm cl}{\domain} \times \reals \rightarrow (0,+\infty)$ and
\begin{align*}
\heating &\in \extendedlinftyspace{\starttime}{\finaltime}{\linftyspace}, \qquad\initu \in \linftyspace,\\
\dirdata &\in \extendedHonespace{\starttime}{\finaltime}{\Hpspace{\domainbdy}{\half}}\intersect\extendedlinftyspace{\starttime}{\finaltime}{\linftyspace[\domainbdy]}.
\end{align*}
Then, there exists a solution $(\u,\conv,\pressure)$ of \eqref{eqn:bousweak},
\begin{align*}
\conv&\in\extendedlinftyspace{\starttime}{\finaltime}{{\left[\Honespace[\domain]\right]}^3}, \qquad\pressure\in\extendedlinftyspace{\starttime}{\finaltime}{\ltwospace},\\
\u&\in\extendedltwospace{\starttime}{\finaltime}{\Honespace[\domain]}\intersect\extendedlinftyspace{\starttime}{\finaltime}{\linftyspace},
\end{align*}
and, if 
$
\conv\in\extendedlinftyspace{\starttime}{\finaltime}{{\left[\sobolevspace{1}{\infty}\right]}^3},
$
then, the solution is unique.\qed
\end{lemma}
%Since we have stronger conditions in place, we can apply this to our situation,  along with a standard treatment of the Neumann boundary conditions, and claim that \eqref{eqn:bousweak} is well-posed.

\subsection{Discretisation of the Boussinesq system}\label{sec:disc_bous}

The discretisation of the energy equation by the discontinuous Galerkin method has already been discussed above. For the Stokes system, we employ standard Taylor-Hood finite elements. To that end, we  introduce the following spaces for the discrete velocity and pressure: for $n=0,1,\ldots,\ntimesteps$ and for $k \geq 2$ let
\[
\velspacehk^n
\coloneqq 
\left[\pwpolyspace[n]{k}\right]^\spacedim,\quad\pressurespacehkminusone^n 
\coloneqq 
\left\{q\in\pwpolyspace[n]{k-1} \suchthat q\in\Czerospace[\domain] \right\}.
\]
%The pairing of these two spaces $\velspacehk^n \times \pressurespacehkminusone^n$  is the classical Taylor-Hood element, which is known to be inf-sup stable; see e.g., \cite{BrezziFortin1991}.

Defining the discrete versions of the bilinear forms $\bilinearstokes{\cdot}{\cdot}$ and $\bilinearpressure{\cdot}{\cdot}$,
\begin{align*}
\bilinearstokesdisc{\conv}{\btest} 
&\coloneqq 
\sum_{\cell\in\tria[n]} \cellinnerprod{2\viscosity\left(\uh[n],\bx\right)\symgrad{\conv}}{\symgrad{\btest}},\\
\bilinearpressuredisc{\btest}{\pressure} 
&\coloneqq 
-\sum_{\cell\in\tria[n]}  \cellinnerprod{\div\btest}{\pressure},
\end{align*}
we state the discretisation of the Stokes problem as: 
find $(\convh[n],\pressureh[n])\in\velspacehk^n\times\pressurespacehkminusone^n$, such that 
\begin{align}
\begin{array}{rl}
\bilinearstokesdisc{\convh[n]}{\btesth} + \bilinearpressuredisc{\btesth}{\pressureh[n]}
&= -\left(\density(\uh[n],\bx)\gravity,\btesth\right) \\
\bilinearpressuredisc{\convh[n]}{q_h} &= 0,
\end{array}
\label{eqn:stokesdiscrete}
\end{align}
for all $(\btesth,q_h) \in \velspacehk\times\pressurespacehkminusone$.

The well-posedness of this formulation is guaranteed as the chosen Taylor-Hood  finite element pair satisfies the  \emph{discrete inf-sup condition} \cite{BercovierPironneau1979,Verfurth1984error,GiraultRaviart1986,BrezziFortin1991}.

We now discuss the solution of the coupled energy-Stokes system. For computational tractability in large scale simulations, we employ a simple scheme that alternates between the numerical solution of 
\eqref{eqn:convdifffullydiscrete} and \eqref{eqn:stokesdiscrete} in the following 
manner.
Given an initial condition on the temperature, $\uh[0] = \inituh$, we use this to 
solve \eqref{eqn:stokesdiscrete} for $(\convh[0],\pressureh[0])$, with $\uh[0]$ used 
to evaluate $\viscosity(\uh[n],\bx)$ and $\density(\uh[n],\bx)$. Having established the 
initial convection field in this way, this is then used when timestepping forward:
at each timestep $\timek[n]$, we solve the convection-diffusion problem 
\eqref{eqn:convdifffullydiscrete} for $\uh[n]$
with the {previous convection field} $\conv[n-1]$ used to evaluate the term 
$\conv\cdot\grad\u$ in the bilinear form $\bilinearahblank$. We are then in turn 
able to employ $\uh[n]$ in solving \eqref{eqn:stokesdiscrete} for $\convh[n]$ and 
$\pressureh[n]$. 

\section{Adaptive resolution of Boussinesq system}\label{sec:numerics_mantle}
%\subsection{Comparison with standard methods within community-code \aspect{}}

We test the method proposed in Section~\ref{sec:Boussinesq}  for the solution of the Boussinesq equations. In all cases the fluid part is discretised using Taylor-Hood  elements as described in Section~\ref{sec:disc_bous} employ adaptivity, driven by the error estimator developed for the convection-diffusion energy equation.
In practice, we use only the term $\deSsub{1}{n}^2$ to mark elements for refinement and coarsening, viz.,
\begin{align}
\indicator{n}{\cell}^2 &\coloneqq  \cellweight^2 \ltwos{\Ak[n] + \diffusivity\laplacian\uh[n] - \conv[n] \cdot \nabla\uh[n] -\addreac[n]\uh[n]}{\cell}^2
% + \ltwos{\heating[n] - \heatingh[n]}{\cell}^2 + \ltwos{\left(\conv[n] - \convh[n]\right) \cdot \nabla \uh[n]}{\cell}^2 \right) 
\nonumber\\
 &\phantom{\coloneqq}+  \sumoverinternaledgesaroundcell \edgepatchweight \ltwos{\jump{\diffusivity \nabla \uh[n]}}{\edge}^2 
\nonumber\\
&\phantom{\coloneqq}+
\sumoveredgesaroundcell \left(
\frac{\penal\diffusivity}{\edgediam}\left(\weightmax{\edgepatch} + \edgepatchvarweight\penal\diffusivity %+ \edgepatchvarweight\diffusivity 
+ \frac{\weightmax{\edge}\alpha^2\diffusivity\maxabs[\edge]{\grad\convgrad}^2}{\correctedltwomin{\edgepatch}}\right) +
\edgepatchweight\linftys{ \conv }{\edge}^2 \right.
\nonumber\\
&\left.\phantom{\ +\sumoveredges \left(\right.}+ \edgediam\linftysw{\correctedltwo}{\edgepatch}
 +  \frac{\weightmax{\edgepatch}\edgediam}{\diffusivity} \linftys{\conv-\alpha\diffusivity\grad\convgrad}{\edgepatch}^2\right) \ltwos{\jump{\uh[n]}}{\edge}^2.
\label{eq:SS}
\end{align}

We employ refinement/coarsening either by fraction of total error or by fraction of cells strategy for adapting the mesh. A pre-defined refinement percentage value (in our case, 10\%) and coarsening percentage value (respectively, 5\%) is set. Then, in the case of total error strategy cells are marked for refinement, from highest indicator to lowest, until the sum of the indicator values reaches the refinement percentage value. Similarly, the lowest-indicator cells are marked for coarsening, until the sum of indicator values matches the coarsening percentage value. Instead, in the case of fraction of cells strategy, the pre-defined percentage of cells are marked for refinement and coarsening.  The fraction of total error strategy offers the ability to ensure a certain amount of  error is refined per adaptivity step, but is  difficult to use in the case where the total number of cells is required to  be limited. On the other hand, the fraction of cells strategy  has the benefit of offering  greater control over the number of cells in the simulation, but offers less in the way of user-defined control of error.

The discretisation method and estimator discussed in this chapter has been implemented within \aspect{}~\cite{KronbichlerHeisterBangerth2012,Heiste2017,Bangerth2018}. Built upon the \dealii{} C++ library, \aspect{} is a community-developed and maintained mantle convection distributed memory simulation code, with a focus on extensibility and research usability.
We exploit this setting to test our approach against the state-of-the-art methods used in \aspect{}. 
%We exploit this setting to test our approach against the state-of-the-art methods used in \aspect{}, considering one of the three-dimensional test cases from the \aspect{} manual~\cite{Bangerth2018}. 

\subsection{van Keken benchmark}
We consider the widely used isoviscous Rayleigh-Taylor thermochemical convection benchmark from~\cite{keken1997comparison}, cf., also the \aspect{} manual~\cite{Bangerth2018}. In this two-dimensional example, the thermal expansion is set to zero and thus the temperature is a passively advected field.  An advantage of the discontinuous Galerkin method is that it can seamlessly be applied in the pure transport case, thus, no changes in the method are required. We shall test the ability of the proposed estimator to track the sharp layers developing in this regime.

We consider the system~\eqref{eqn:boussinesq} with domain $
\domain=(0,0.9142)\times(0,1)
$ and for $\timeinterval=[0,2000]$. We set  $\diffusivity=0$ and $ \heating(\bx,\timevar)=0$ in the first equation in~\eqref{eqn:boussinesq} and set $\viscosity=100$ and $\density(\u,\bx)=10^6 \u$.
The system is initialised with a base of warm material below a colder material, with a small perturbation imposed on the interface to reliably initiate a convective flow. To this end, we set 
\[
\initu(\bx)= 
\left\{
\begin{array}{ll}
1 &\text{if } y<0.2(1+0.1\cos(\frac{\pi x}{0.9142}));\\
0 & \text{otherwise.}
\end{array}
\right.
\] 
We consider fixed Dirichlet boundary conditions for the temperature, compatible with the initial field shown in Figure~\ref{fig:vanKekeninit}. 
\begin{figure}[h!]
\centering
\includegraphics[height=0.3\textheight]{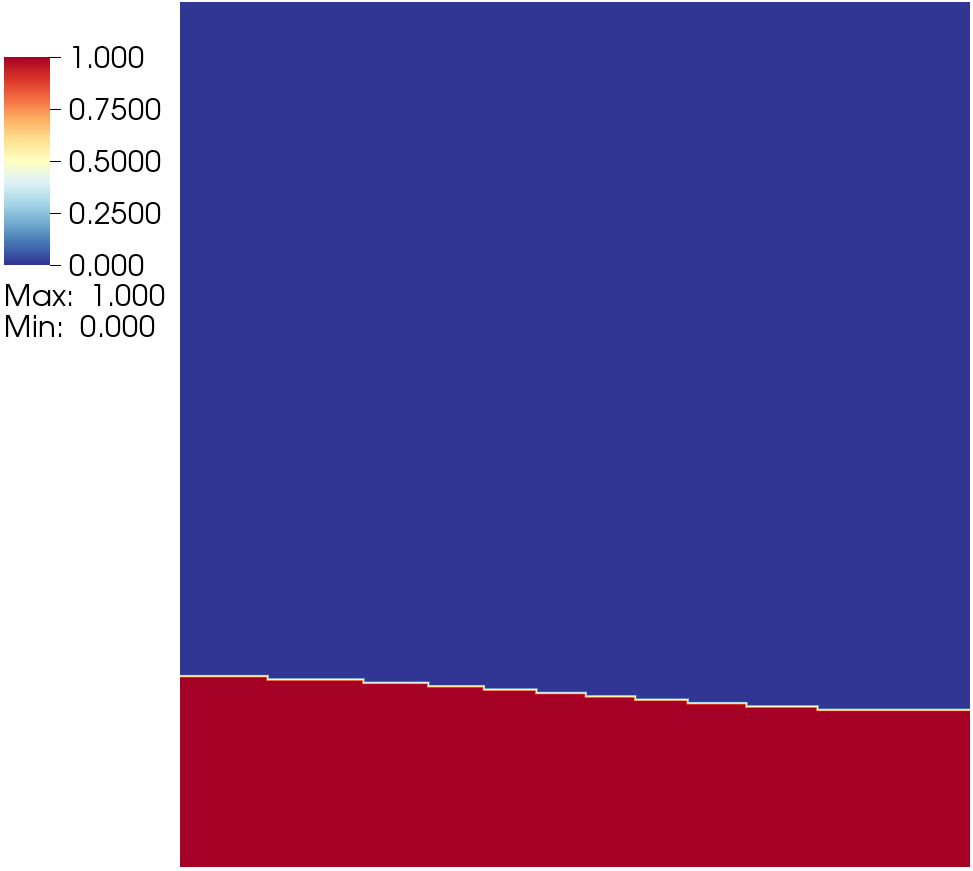}
\caption[Initial composition distribution in the van Keken benchmark]{The initial distribution of the temperature in the van~Keken isoviscous composition benchmark.}
\label{fig:vanKekeninit}
\end{figure} 
As a result,  the boundary conditions jump from 0 to 1 where the prescribed initial temperature field jumps on the left and right boundaries. Note that the resulting temperature transport initial and boundary value  problem can be interpreted as a \emph{compositional} equation for the warm material, initially sitting at the bottom of the domain. As such, the temperature is sometime referred to as compositional field.

%We are interested in this example not in the indicator choice, but in the ability of the dG method to approximate a sharp moving boundary, and a zero-diffusion field. As such, w
The discretisation of the compositional field by the dG method is first compared with that obtained with a standard artificial diffusion continuous finite elements on a fixed, uniform grid. 
Figure \ref{fig:vanKeken} demonstrates that the dG method can more effectively conserve the sharp interfaces of the composition field, resulting in less `smearing' of the field as time increases.
\begin{figure}[h!]
\centering
%\subfloat[fig 1]{\includegraphics[scale=0.3]{Images/dg_vk_70040.png}}\\
%\subfloat[fig 1]{\includegraphics[scale=0.305]{Images/cg_vk_70040.png}}
%\begin{subfigure}[b]{1\textwidth}
	\centering
    \includegraphics[height=0.31\textheight]{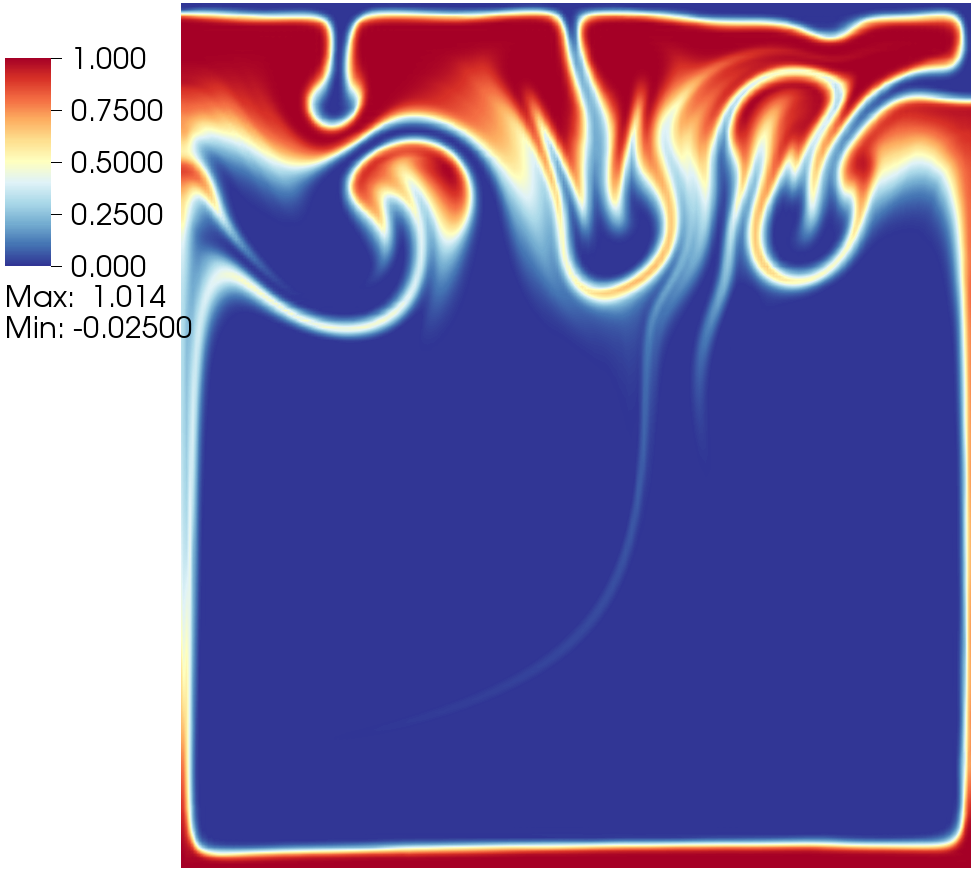}
%    \caption{FE discretisation}
 %   \label{fig:f1}
 % \end{subfigure}
 % 
%  \vspace{1em}
 % \begin{subfigure}[b]{1\textwidth}
%	\centering
    \includegraphics[height=0.31\textheight]{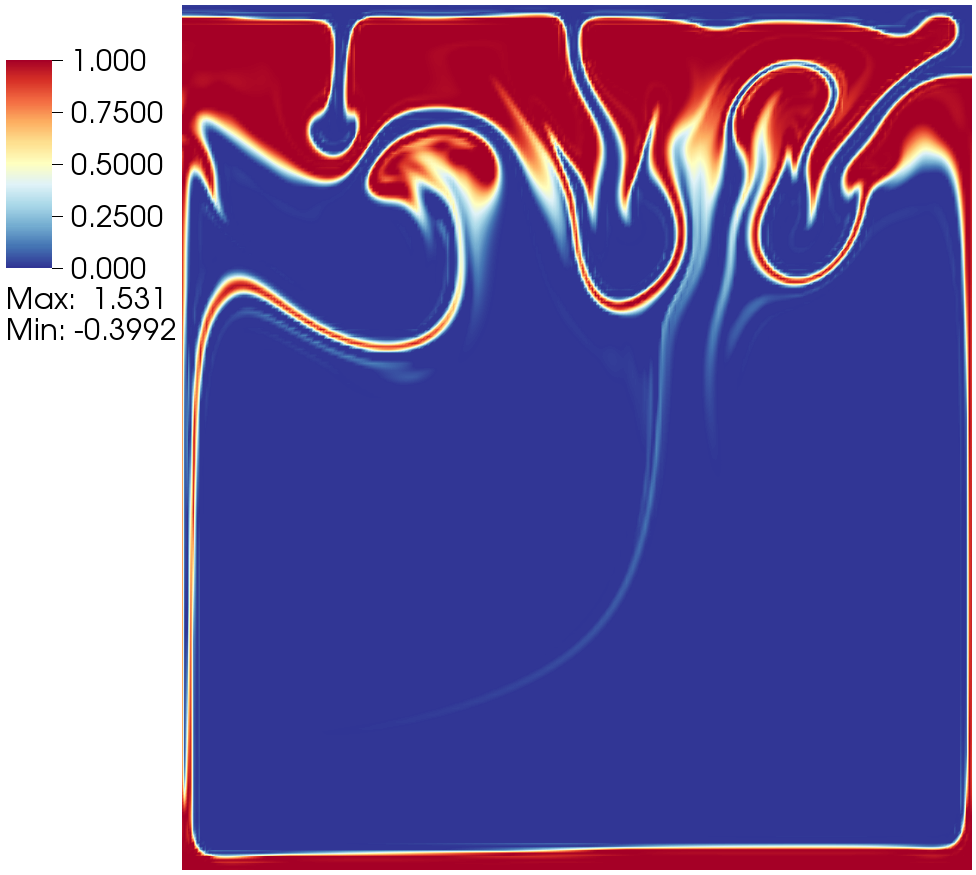}
  %  \caption{DG discretisation}
  %  \label{fig:f2}
 % \end{subfigure}
  \caption[van Keken benchmark: comparison between FE and dG solution]{van~Keken isoviscous composition benchmark: comparison between FE (left) and dG (right) solution. Fixed rectangular mesh refined 7 times. Solution at final time $\timevar = 2000$.}
  \label{fig:vanKeken}
\end{figure} 
On the other hand,  the dG method produces overshoots and undershoots around the discontinuities, a clear evidence that the mesh size  is not fully resolving the sharp solution's layers and of the necessity of mesh refinement.
We note that the dG method can, in principle, also naturally incorporate flux limiters within its numerical flux functions, to limit overshoots and undershoots. Such non-linear stabilisation techniques are implemented in
\aspect{} \cite{he2016}, limited to the case of divergence-free flow,
 building on the methods introduced in \cite{zhang2010positivity,zhang2013maximum}. Here, we opt not to use such limiters, in an effort to separate the effect of the dG method from that of the limiter.

Figure \ref{fig:ch5_adapt} shows the solution and mesh produced by the adaptive algorithm driven by~\eqref{eq:SS} as error indicator, employing the fraction of cells marking strategy. 
\begin{figure}[h!]
\centering
\begin{subfigure}[b]{.48\linewidth}
\includegraphics[height=0.28\textheight]{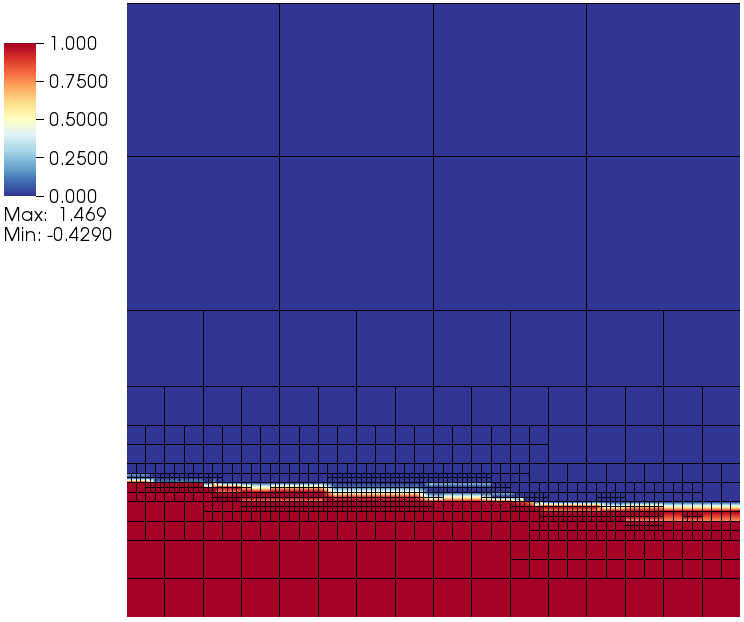}
\caption{$\timevar=0$}
\end{subfigure}
\centering
\begin{subfigure}[b]{.48\linewidth}
\includegraphics[height=0.28\textheight]{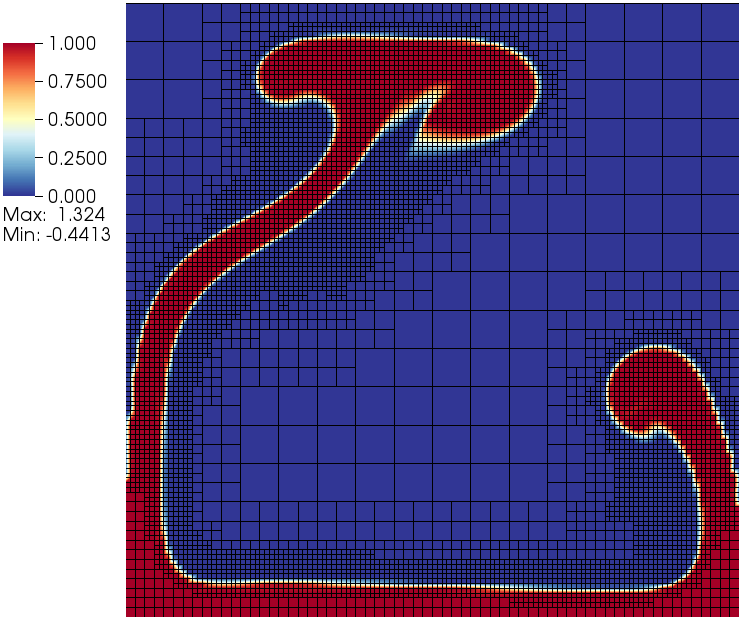}
\caption{$\timevar=1.2$}
\end{subfigure}
%\begin{figure}\ContinuedFloat
%\centering
\begin{subfigure}[b]{.48\linewidth}
\includegraphics[height=0.28\textheight]{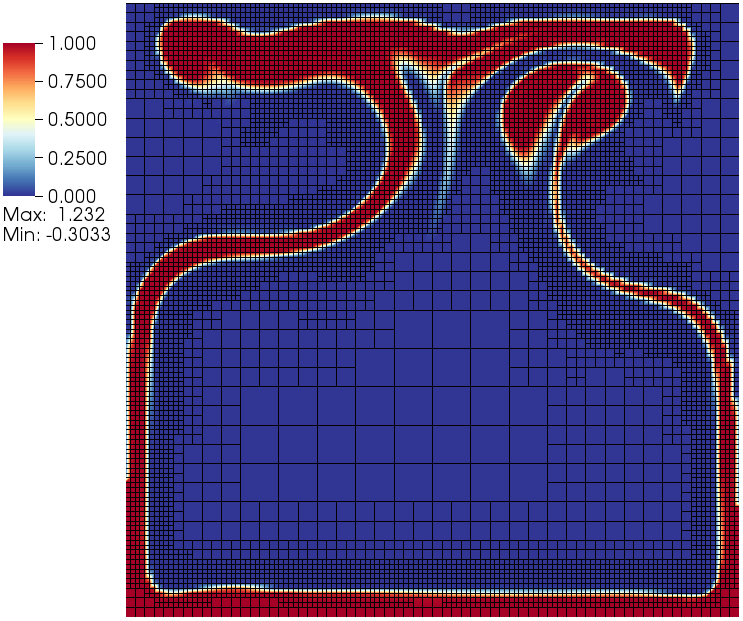}
\caption{$\timevar=2.4$}
\end{subfigure}
%\centering
\begin{subfigure}[b]{.48\linewidth}
\includegraphics[height=0.28\textheight]{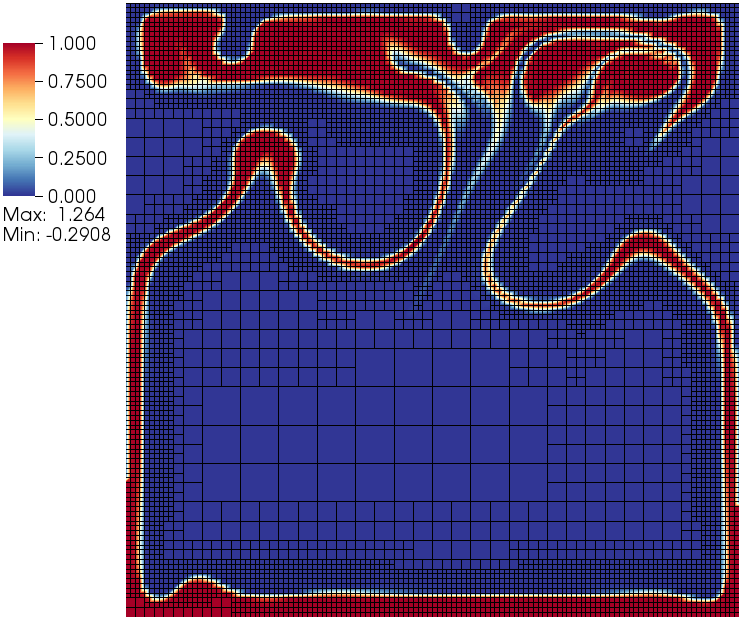}
\caption{$\timevar=3.6$}
\end{subfigure}
\caption[vanKeken_adaptive]{Adaptive simulation of the van Keken benchmark: temperature spatial distribution and adaptive meshes.}\label{fig:ch5_adapt}
\end{figure}
The adaptive algorithm accurately represents the sharp solution layers with reduced complexity, as can be clearly seen from Figure~\ref{fig:ch5_adapt_zoom} focusing on the upper-right portion of the domain.
\begin{figure}[h!]
\hspace{.1cm}
%\centering
\begin{subfigure}[b]{.5\linewidth}
\includegraphics[height=0.3\textheight]{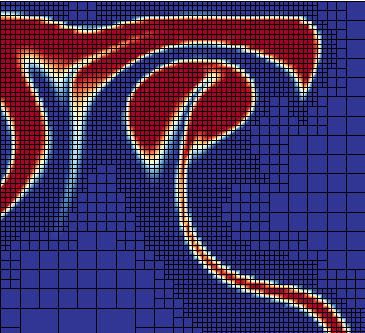}
\caption{$\timevar=2.4$}
\end{subfigure}
%\centering
\begin{subfigure}[b]{.5\linewidth}
\includegraphics[height=0.3\textheight]{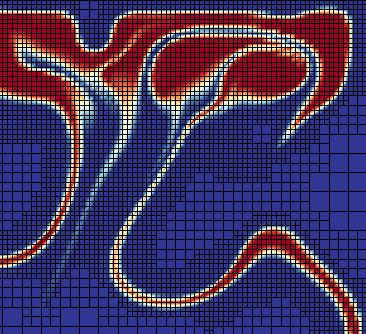}
\caption{$\timevar=3.6$}
\end{subfigure}
\caption[vanKeken_adaptive]{Zoom of the upper-right portion of the second row pictures in Figure~\ref{fig:ch5_adapt}.}\label{fig:ch5_adapt_zoom}
\end{figure}
However, albeit reduced, undershoots and overshoots are still present.
These may be reduced by refining more aggressively the initial mesh and/or applying flux limiters as mentioned above. 
\subsection{Three-dimensional test case}

We consider one of the three-dimensional test cases from the \aspect{} manual~\cite{Bangerth2018}. 
On the unit cube space domain $\Omega=[0,1]^3$ and with final time $T=0.5$, we solve problem~\eqref{eqn:boussinesq} with $\diffusivity=\viscosity=1$, $\density=1-T$, and $\heating=0$. Initial conditions for the temperature are set as a linear profile with a small perturbation, namely $\initu(\bx)=1-x_3-10^{-2}\cos(\pi x_1) \sin(\pi x_3) x_2^3$. Time-independent Dirichlet boundary conditions compatible with the initial condition are set on the bottom and top side of the cube while homogeneous Neumann conditions are fixed on all vertical sides.

We compare the following three adaptive methodologies: 
\begin{itemize}
\item the standard conforming finite element method stabilised by the  entropy viscosity method~\cite{GuermondPasquettiPopov2011} with Kelly error indicator;
\item the dG method with Kelly error indicator;
\item the dG method with the error indicator~\eqref{eq:SS}.  
\end{itemize}
In each case, the same  fraction of total error marking strategy is used. The so-called  Kelly error indicator~\cite{kelly1983posteriori}  is an \emph{ad hoc} widely-used error indicator among $h$-refinement codes: it employs the jump on the normal flux across element faces only, corresponding to \eqref{normal_flux_jump} without the weight.  

To simplify the error indicator~\eqref{eq:SS} within \aspect{}, we consider the modifications detailed in Section~\ref{sec:imple}. We compute~\eqref{eq:SS} to drive the mesh adaptivity.  We note that the union mesh would only be required for the computation of the projection $\ltwoproj[k]^n\uh[n-1]$ appearing in the factor
\begin{equation*}
\Ak[n] = \ltwoproj[n]\left(\heating[n]+\addreac[n]\uh[n]\right) - \left(\uh[n]-\ltwoproj[n]\uh[n-1]\right)/\timestep{n}.
\end{equation*}
To avoid forming the union mesh altogether, we replace the projection $ \ltwoproj[n]$ by the nodal interpolant $\interp{n}$ onto $\tdgspace[n]$.

In Figure \ref{fig:ex4-fig1} we display a snapshot of the temperature solution obtained with our approach. Those obtained with other approaches are indistinguishable visually and, thus, omitted for brevity. 
\begin{figure}
\centering
\includegraphics[height=0.4\textheight]{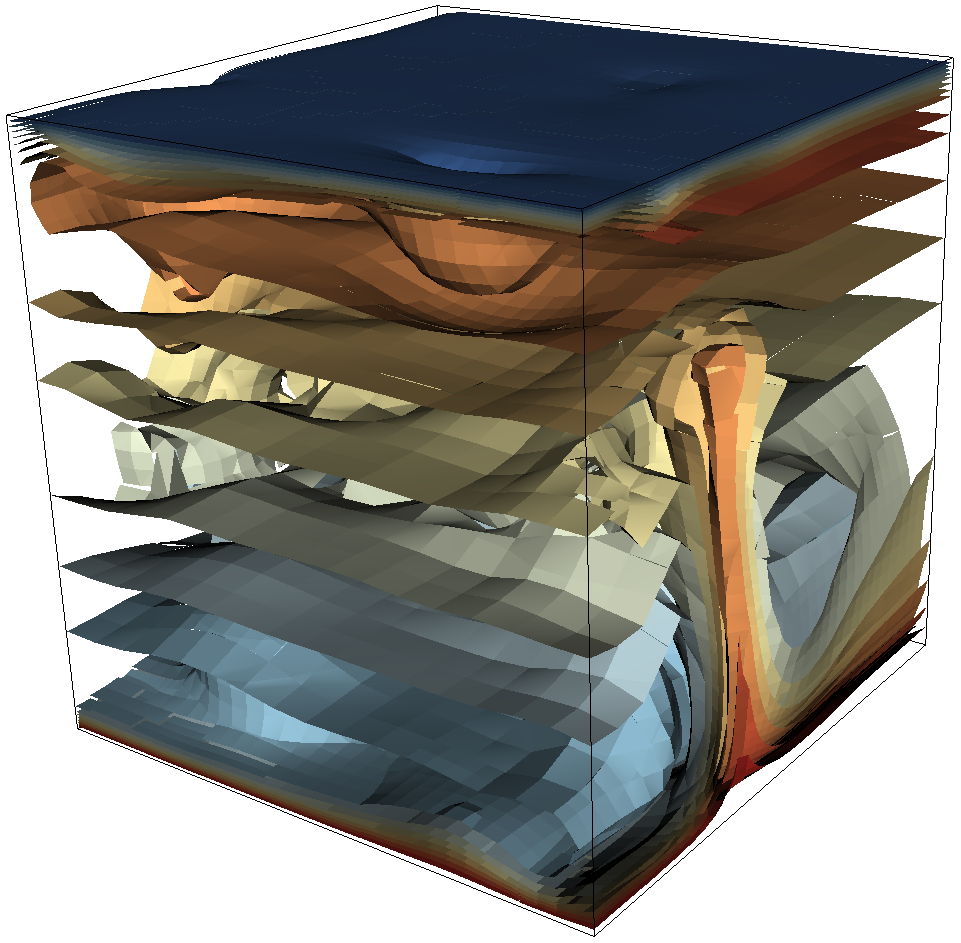}
\caption{Isocontours of a temperature solution obtained with the IPDG method with the newly developed error indicator.}\label{fig:ex4-fig1}
\end{figure}

\begin{figure}
	\centering
%	\begin{subfigure}[b]{.48\linewidth}
		\centering
		\includegraphics[height=0.45\textheight]{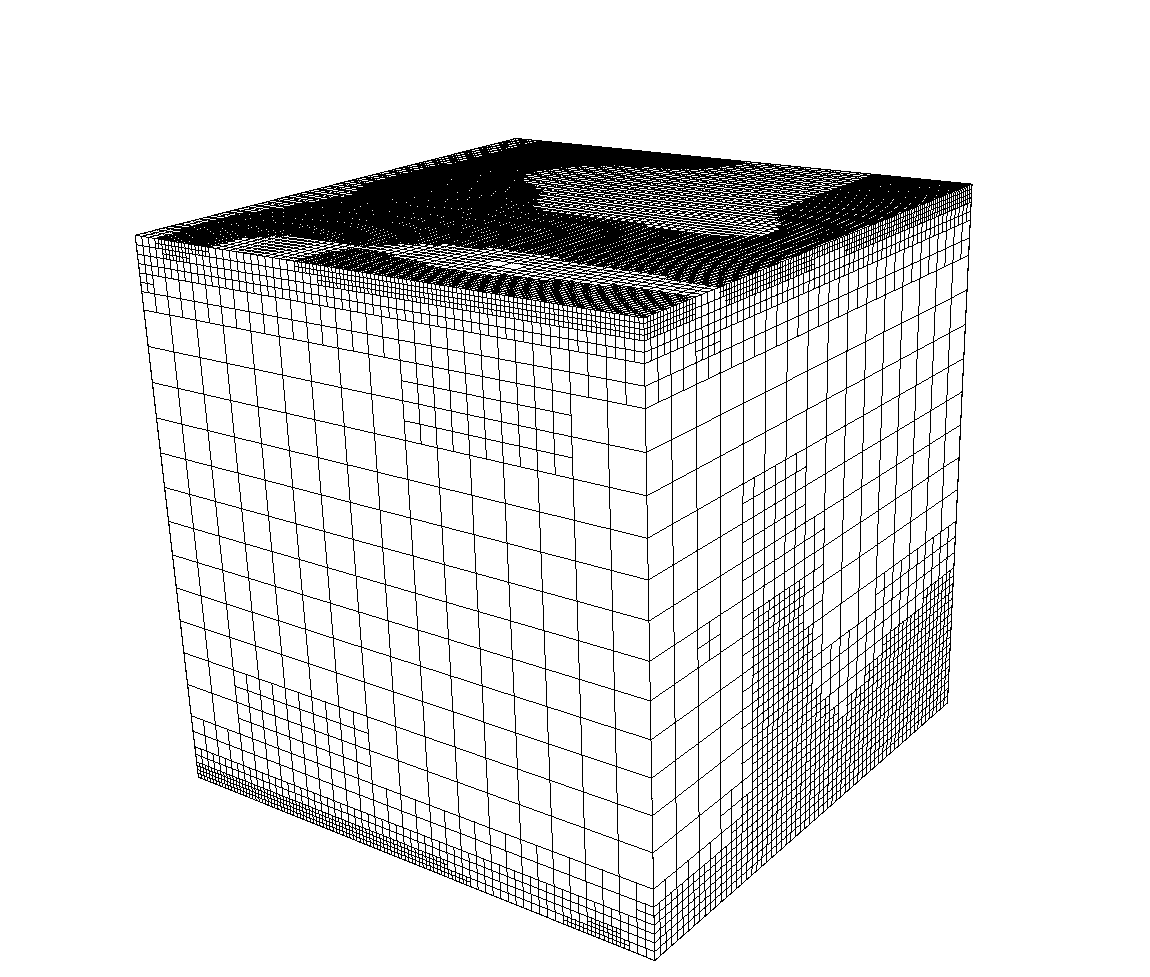}
		\caption{Outer mesh generated by the  FEM with the Kelly indicator.}
%	\end{subfigure}
\label{fig:ex4-fig2-FE_K}
\end{figure}
\begin{figure}
	\centering
%	\begin{subfigure}[b]{.48\linewidth}
		\centering
		\includegraphics[height=0.45\textheight]{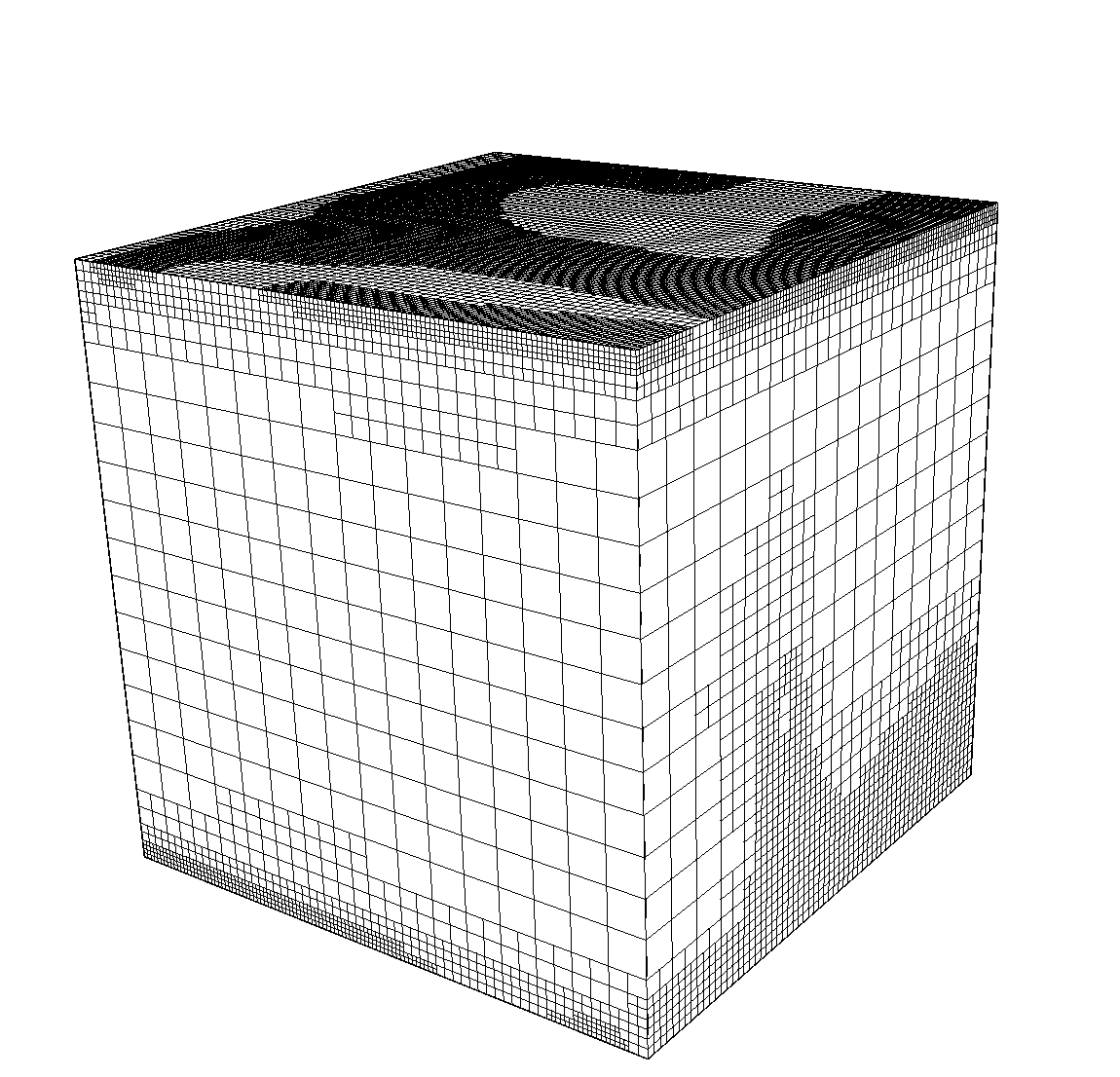}
		\caption{Outer mesh generated by the dG method with the Kelly indicator.}
%	\end{subfigure}
\label{fig:ex4-fig2-DG_K}
\end{figure}
\begin{figure}
	%\centering
%	\begin{subfigure}[b]{\linewidth}
		\centering
		\includegraphics[height=0.45\textheight]{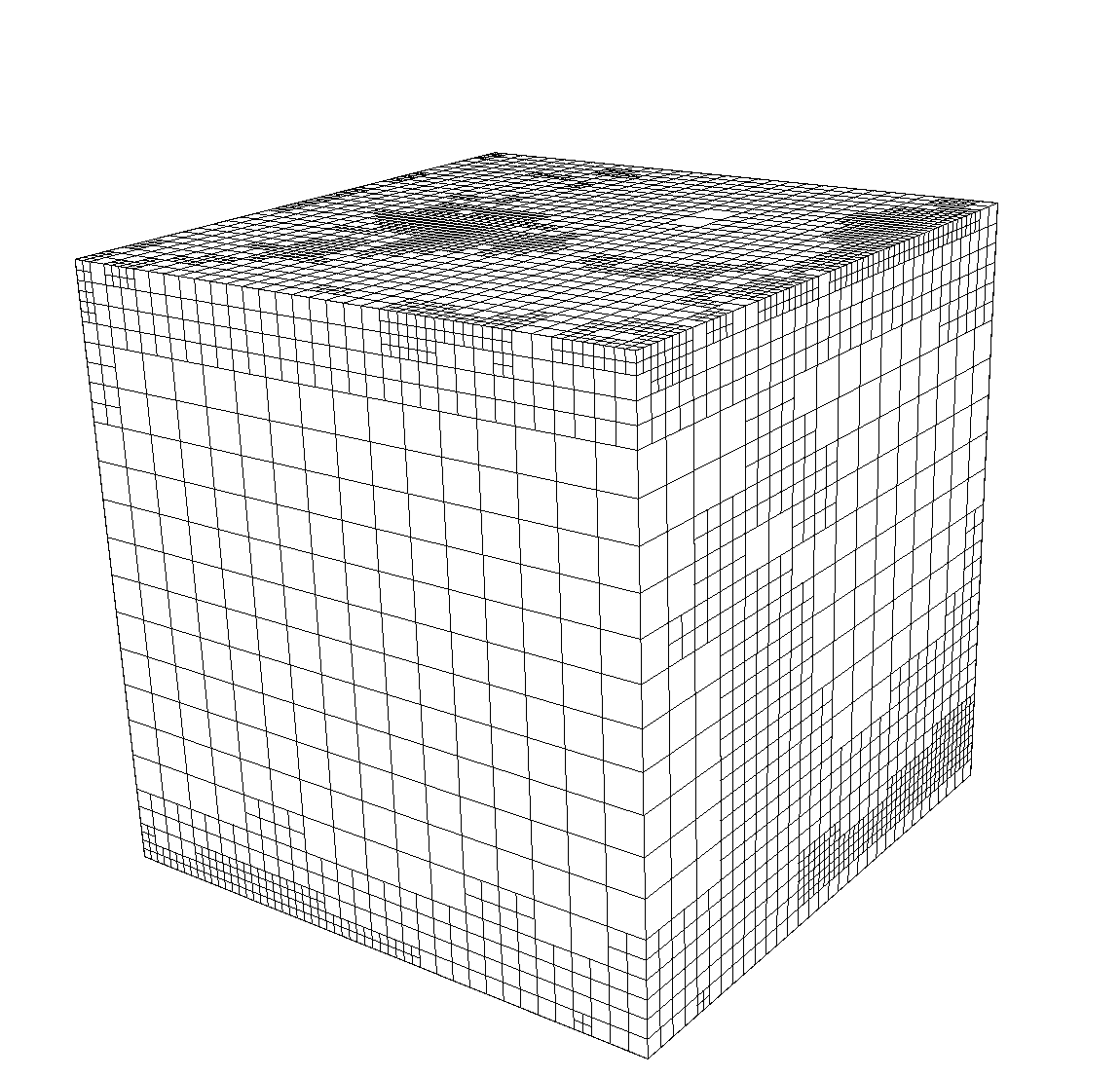}
		%\caption{dG, Derived indicator}
%	\end{subfigure}
	\caption{Outer mesh generated by the dG method with the derived indicator.}\label{fig:ex4-fig2-DG-D}
\end{figure}

\begin{figure}
	\centering
	%\centering
	\includegraphics[height=0.32\textheight]{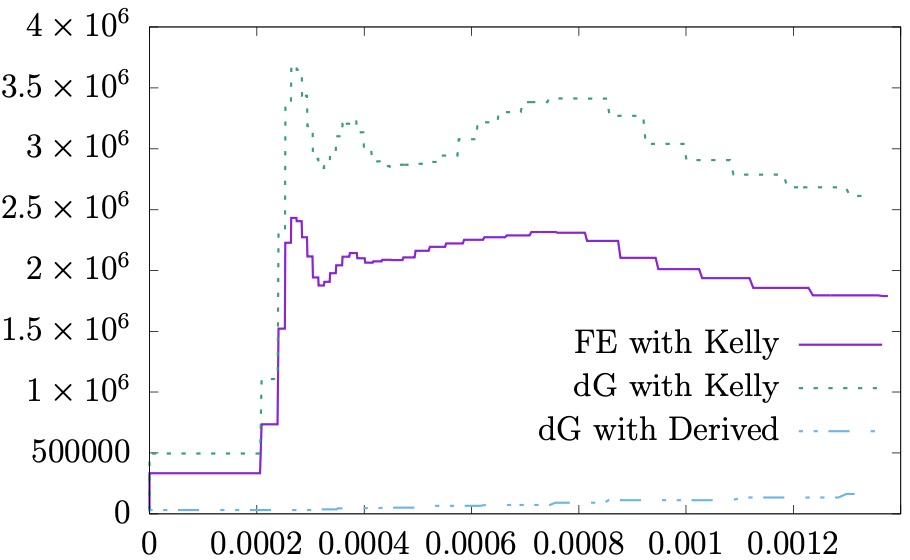}
	\caption{Degrees of Freedom (DoF) count (vertical axis) per timestep (horizontal axis), for the three combinations of discretisation and indicator.}\label{fig:ex4-fig3}
\end{figure}

Figures \ref{fig:ex4-fig2-FE_K}, \ref{fig:ex4-fig2-DG_K}, and \ref{fig:ex4-fig2-DG-D} compare  the outer surface of the meshes 
generated adaptively by the three methods. The Kelly indicator generates 
similar meshes in both the FE and dG case, while the derived indicator
admits more localised refinement, resulting in a less-refined mesh
overall. This is evident in the significant disparity between the mesh cardinalities shown in Figure~\ref{fig:ex4-fig3}.

\section{Conclusions}\label{sec:conclusions}

This work has been concerned with the  derivation of an \aposteriori{} error bound for the discontinuous Galerkin method  applied to convection-diffusion equations, in a modified norm, without the usual restrictions placed upon the divergence of the velocity field. The analysis is motivated by the need to handle convection-dominated problems
with positive divergence, such as when the convection field is obtained from a 
non divergence-free approximation. This bound
is subject to an exponential term in the event of non-negative divergence, as well as a non-standard Gr\"onwall argument. The error bound leads to an adaptivity indicator 
designed for the problem in question, enabling the adaptivity strategy to
be guided in a more rigorously supported fashion. Further work remains to understand the
full consequences of varying choices of parameter $\alpha$ in this bound, 
and to identify the exact circumstances under which this result improves 
on existing known bounds.

The scenario of convection-dominated problems
with positive divergence, is exemplified in the context of simulation of the Boussinesq system modelling Earth's mantle convection. There, the energy/temperature equation admits strong convection which is produced by a coupled Stokes equation. The Stokes system is solved using Taylor-Hood elements and may result to non-divergence-free or even positive velocities. The temperature equation is discretised via an interior penalty discontinuous Galerkin method. The new \aposteriori{} error estimators proven in the first part of the present work are used to drive dynamic adaptive mesh modification. The new adaptivity strategy based on the \aposteriori{} error 
estimator appears to give computational savings with no detriment to the observed convection patterns. We, thus, expect it to result in better approximation of full mantle 
simulations, compared to current approaches. 

\bibliographystyle{apalike}
\bibliography{bibliography}

\begin{thebibliography}{10}

\bibitem{ainsworth2000}
Mark Ainsworth and J.~Tinsley Oden.
\newblock {\em A posteriori error estimation in finite element analysis}.
\newblock Pure and Applied Mathematics (New York). Wiley-Interscience [John
  Wiley \& Sons], New York, 2000.

\bibitem{araya2005adaptive}
Rodolfo Araya, Edwin Behrens, and Rodolfo Rodr{\'{\i}}guez.
\newblock An adaptive stabilized finite element scheme for the
  advection-reaction-diffusion equation.
\newblock {\em Appl. Numer. Math.}, 54(3-4):491--503, 2005.

\bibitem{araya2005hierarchical}
Rodolfo Araya, Abner~H. Poza, and Ernst~P. Stephan.
\newblock A hierarchical a posteriori error estimate for an
  advection-diffusion-reaction problem.
\newblock {\em Math. Models Methods Appl. Sci.}, 15(7):1119--1139, 2005.

\bibitem{AyusoMarini2009}
Blanca Ayuso and L.~Donatella Marini.
\newblock Discontinuous {G}alerkin methods for advection-diffusion-reaction
  problems.
\newblock {\em SIAM J. Numer. Anal.}, 47(2):1391--1420, 2009.

\bibitem{Bangerth2018}
Wolfgang Bangerth, Juliane Dannberg, Menno Fraters, Rene Gassmoeller, Anne
  Glerum, and Timo; et~al. Heister.
\newblock Aspect: Advanced solver for planetary evolution, convection, and
  tectonics, user manual. figshare. journal contribution.
\newblock 2018.

\bibitem{Bansch2012}
E~Bansch, F~Karakatsani, and C~Makridakis.
\newblock {A posteriori error control for fully discrete Crank-Nicolson
  schemes}.
\newblock {\em SIAM Journal on Numerical Analysis}, 50(6):2845--2872, 2012.

\bibitem{BercoviciSchubertGlatzmeier1989}
Dave Bercovici, Gerald Schubert, and Gary~A Glatzmaier.
\newblock Three-dimensional spherical models of convection in the {E}arth's
  mantle.
\newblock {\em Science}, 244(4907):950--955, 1989.

\bibitem{BercovierPironneau1979}
M.~Bercovier and O.~Pironneau.
\newblock Error estimates for finite element method solution of the {S}tokes
  problem in the primitive variables.
\newblock {\em Numer. Math.}, 33(2):211--224, 1979.

\bibitem{berrone2004multilevel}
Stefano Berrone and Claudio Canuto.
\newblock Multilevel a posteriori error analysis for
  reaction-convection-diffusion problems.
\newblock {\em Appl. Numer. Math.}, 50(3-4):371--394, 2004.

\bibitem{BrezziFortin1991}
Franco Brezzi and Michel Fortin.
\newblock {\em Mixed and hybrid finite element methods}, volume~15 of {\em
  Springer Series in Computational Mathematics}.
\newblock Springer-Verlag, New York, 1991.

\bibitem{CaDoGe21}
Andrea Cangiani, Zhaonan Dong, and Emmanuil~H. Georgoulis.
\newblock {$hp$}-version discontinuous {G}alerkin methods on essentially
  arbitrarily-shaped elements.
\newblock {\em Math. Comp.}, 91(333):1--35, 2021.

\bibitem{CaDoGeHo16}
Andrea Cangiani, Zhaonan Dong, Emmanuil~H. Georgoulis, and Paul Houston.
\newblock hp-{V}ersion discontinuous {G}alerkin methods for
  advection-diffusion-reaction problems on polytopic meshes.
\newblock {\em ESAIM: M2AN}, 50(3):699--725, 2016.

\bibitem{CaDoGeHo17}
Andrea Cangiani, Zhaonan Dong, Emmanuil~H. Georgoulis, and Paul Houston.
\newblock {\em hp-{V}ersion discontinuous {G}alerkin methods on polygonal and
  polyhedral meshes}.
\newblock SpringerBriefs in Mathematics. Springer Cham, first edition, 2017.

\bibitem{cangiani2014hp}
Andrea Cangiani, Emmanuil~H. Georgoulis, and Paul Houston.
\newblock {$hp$}-version discontinuous {G}alerkin methods on polygonal and
  polyhedral meshes.
\newblock {\em Math. Models Methods Appl. Sci.}, 24(10):2009--2041, 2014.

\bibitem{Cangiani2013a}
Andrea Cangiani, Emmanuil~H. Georgoulis, and Stephen Metcalfe.
\newblock Adaptive discontinuous {G}alerkin methods for nonstationary
  convection-diffusion problems.
\newblock {\em IMA J. Numer. Anal.}, 34(4):1578--1597, 2014.

\bibitem{ern2005posteriori}
Alexandre Ern and Jennifer Proft.
\newblock A posteriori discontinuous {G}alerkin error estimates for transient
  convection-diffusion equations.
\newblock {\em Appl. Math. Lett.}, 18(7):833--841, 2005.

\bibitem{ern2010guaranteed}
Alexandre Ern, Annette~F. Stephansen, and Martin Vohral{\'{\i}}k.
\newblock Guaranteed and robust discontinuous {G}alerkin a posteriori error
  estimates for convection-diffusion-reaction problems.
\newblock {\em J. Comput. Appl. Math.}, 234(1):114--130, 2010.

\bibitem{evans1998partial}
Lawrence~C. Evans.
\newblock {\em Partial differential equations}, volume~19 of {\em Graduate
  Studies in Mathematics}.
\newblock American Mathematical Society, Providence, RI, 1998.

\bibitem{georgoulis2007discontinuous}
Emmanuil~H. Georgoulis, Edward Hall, and Paul Houston.
\newblock Discontinuous {G}alerkin methods for advection-diffusion-reaction
  problems on anisotropically refined meshes.
\newblock {\em SIAM J. Sci. Comput.}, 30(1):246--271, 2007/08.

\bibitem{georgoulisSub}
Emmanuil~H Georgoulis, Edward Hall, and Charalambos Makridakis.
\newblock An a posteriori error bound for discontinuous galerkin approximations
  of convection–diffusion problems.
\newblock {\em IMA Journal of Numerical Analysis}, 39(1):34--60, 12 2017.

\bibitem{Georgoulis2011}
Emmanuil~H. Georgoulis, Omar Lakkis, and Juha~M. Virtanen.
\newblock A posteriori error control for discontinuous {G}alerkin methods for
  parabolic problems.
\newblock {\em SIAM J. Numer. Anal.}, 49(2):427--458, 2011.

\bibitem{GM23}
Emmanuil~H. Georgoulis and Charalambos~G. Makridakis.
\newblock Lower bounds, elliptic reconstruction and {\it a~posteriori} error
  control of parabolic problems.
\newblock {\em IMA J. Numer. Anal.}, 43(6):3212--3242, 2023.

\bibitem{gerya2003characteristics}
Taras~V Gerya and David~A Yuen.
\newblock Characteristics-based marker-in-cell method with conservative
  finite-differences schemes for modeling geological flows with strongly
  variable transport properties.
\newblock {\em Phys. Earth Planet. Inter.}, 140(4):293--318, 2003.

\bibitem{GiraultRaviart1986}
Vivette Girault and Pierre-Arnaud Raviart.
\newblock {\em Finite element methods for {N}avier-{S}tokes equations},
  volume~5 of {\em Springer Series in Computational Mathematics}.
\newblock Springer-Verlag, Berlin, 1986.
\newblock Theory and algorithms.

\bibitem{GuermondPasquettiPopov2011}
Jean-Luc Guermond, Richard Pasquetti, and Bojan Popov.
\newblock Entropy viscosity method for nonlinear conservation laws.
\newblock {\em J. Comput. Phys.}, 230(11):4248--4267, 2011.

\bibitem{HansenEbel1984}
U~Hansen and A~Ebel.
\newblock Experiments with a numerical model related to mantle convection:
  boundary layer behaviour of small-and large scale flows.
\newblock {\em Phys. Earth Planet. Inter.}, 36(3):374--390, 1984.

\bibitem{he2016}
Ying He, Elbridge~Gerry Puckett, and Magali~I Billen.
\newblock A discontinuous {G}alerkin method with a bound preserving limiter for
  the advection of non-diffusive fields in solid {E}arth geodynamics.
\newblock {\em Phys. Earth Planet. Inter.}, 263:23--37, 2017.

\bibitem{Heiste2017}
Timo Heister, Juliane Dannberg, Rene Gassmöller, and Wolfgang Bangerth.
\newblock High accuracy mantle convection simulation through modern numerical
  methods – ii: realistic models and problems.
\newblock {\em Geophysical Journal International}, 210(2):833--851, 05 2017.

\bibitem{HoustonPerugiaSchotzau2004}
Paul Houston, Ilaria Perugia, and Dominik Sch{\"o}tzau.
\newblock Mixed discontinuous {G}alerkin approximation of the {M}axwell
  operator.
\newblock {\em SIAM J. Numer. Anal.}, 42(1):434--459 (electronic), 2004.

\bibitem{HoustonSchotzauWihler2007}
Paul Houston, Dominik Sch{\"o}tzau, and Thomas~P. Wihler.
\newblock Energy norm a posteriori error estimation of {$hp$}-adaptive
  discontinuous {G}alerkin methods for elliptic problems.
\newblock {\em Math. Models Methods Appl. Sci.}, 17(1):33--62, 2007.

\bibitem{houston2001adaptive}
Paul Houston and Endre S{\"u}li.
\newblock Adaptive {L}agrange-{G}alerkin methods for unsteady
  convection-diffusion problems.
\newblock {\em Math. Comp.}, 70(233):77--106, 2001.

\bibitem{hughes1979multidimensional}
T.~J.~R. Hughes and A.~Brooks.
\newblock A multidimensional upwind scheme with no crosswind diffusion.
\newblock In {\em Finite element methods for convection dominated flows
  ({P}apers, {W}inter {A}nn. {M}eeting {A}mer. {S}oc. {M}ech. {E}ngrs., {N}ew
  {Y}ork, 1979)}, volume~34 of {\em AMD}, pages 19--35. Amer. Soc. Mech. Engrs.
  (ASME), New York, 1979.

\bibitem{hughes1982theoretical}
Thomas~JR Hughes, A~Brooks, et~al.
\newblock A theoretical framework for petrov-galerkin methods with
  discontinuous weighting functions: Application to the streamline-upwind
  procedure.
\newblock {\em Finite elements in fluids}, 4(2):47, 1982.

\bibitem{Karakashian2003}
Ohannes~A. Karakashian and Frederic Pascal.
\newblock A posteriori error estimates for a discontinuous {G}alerkin
  approximation of second-order elliptic problems.
\newblock {\em SIAM J. Numer. Anal.}, 41(6):2374--2399 (electronic), 2003.

\bibitem{Karakashian2004}
Ohannes~A Karakashian and Frederic Pascal.
\newblock {Adaptive discontinuous Galerkin approximations of second-order
  elliptic problems}.
\newblock {\em ECCOMAS 2004 - European Congress on Computational Methods in
  Applied Sciences and Engineering}, 2004.

\bibitem{keken1997comparison}
PE~van Keken, SD~King, H~Schmeling, UR~Christensen, D~Neumeister, and M-P Doin.
\newblock A comparison of methods for the modeling of thermochemical
  convection.
\newblock {\em J. Geophys. Res. Solid Earth}, 102(B10):22477--22495, 1997.

\bibitem{kelly1983posteriori}
D.~W. Kelly, J.~P. de S.~R. Gago, O.~C. Zienkiewicz, and I.~Babu{\v{s}}ka.
\newblock A posteriori error analysis and adaptive processes in the finite
  element method. {I}. {E}rror analysis.
\newblock {\em Internat. J. Numer. Methods Engrg.}, 19(11):1593--1619, 1983.

\bibitem{kelly1980note}
D.~W. Kelly, S.~Nakazawa, O.~C. Zienkiewicz, and J.~C. Heinrich.
\newblock A note on upwinding and anisotropic balancing dissipation in finite
  element approximations to convective diffusion problems.
\newblock {\em Internat. J. Numer. Methods Engrg.}, 15(11):1705--1711, 1980.

\bibitem{KronbichlerHeisterBangerth2012}
Martin Kronbichler, Timo Heister, and Wolfgang Bangerth.
\newblock High accuracy mantle convection simulation through modern numerical
  methods.
\newblock {\em Geophys. J. Int.}, 191(1):12--29, 2012.

\bibitem{Kunert2003}
Gerd Kunert.
\newblock A posteriori error estimation for convection dominated problems on
  anisotropic meshes.
\newblock {\em Math. Methods Appl. Sci.}, 26(7):589--617, 2003.

\bibitem{LM06}
Omar Lakkis and Charalambos Makridakis.
\newblock Elliptic reconstruction and a posteriori error estimates for fully
  discrete linear parabolic problems.
\newblock {\em Math. Comp.}, 75(256):1627--1658, 2006.

\bibitem{M07}
Charalambos Makridakis.
\newblock Space and time reconstructions in a posteriori analysis of evolution
  problems.
\newblock In {\em E{SAIM} {P}roceedings. {V}ol. 21 (2007) [{J}ourn\'ees
  d'{A}nalyse {F}onctionnelle et {N}um\'erique en l'honneur de {M}ichel
  {C}rouzeix]}, volume~21 of {\em ESAIM Proc.}, pages 31--44. EDP Sci., Les
  Ulis, 2007.

\bibitem{Makridakis2003}
Charalambos Makridakis and Ricardo~H. Nochetto.
\newblock Elliptic reconstruction and a posteriori error estimates for
  parabolic problems.
\newblock {\em SIAM J. Numer. Anal.}, 41(4):1585--1594, 2003.

\bibitem{May2013}
DA~May, WP~Schellart, and L~Moresi.
\newblock Overview of adaptive finite element analysis in computational
  geodynamics.
\newblock {\em Journal of Geodynamics}, 70:1--20, 2013.

\bibitem{MckenzieRobertsWeiss1974}
Dan~P McKenzie, Jean~M Roberts, and Nigel~O Weiss.
\newblock Convection in the earth's mantle: towards a numerical simulation.
\newblock {\em J. Fluid Mech.}, 62(03):465--538, 1974.

\bibitem{MinearToksoz1970}
John~W Minear and M~Nafi Toks{\"o}z.
\newblock Thermal regime of a downgoing slab and new global tectonics.
\newblock {\em J. Geophys. Res.}, 75(8):1397--1419, 1970.

\bibitem{Sangalli2008}
Giancarlo Sangalli.
\newblock Robust a-posteriori estimator for advection-diffusion-reaction
  problems.
\newblock {\em Math. Comp.}, 77(261):41--70 (electronic), 2008.

\bibitem{SchmelingJacoby1981}
H~Schmeling and WR~Jacoby.
\newblock On modeling the lithosphere in mantle convection with non-linear
  rheology.
\newblock {\em Journal of Geophysics-Zeitschrift Fur Geophysik}, 50(2):89--100,
  1981.

\bibitem{Schotzau2009}
Dominik Sch{\"o}tzau and Liang Zhu.
\newblock A robust a-posteriori error estimator for discontinuous {G}alerkin
  methods for convection-diffusion equations.
\newblock {\em Appl. Numer. Math.}, 59(9):2236--2255, 2009.

\bibitem{sun2006posteriori}
Shuyu Sun and Mary~F. Wheeler.
\newblock A posteriori error estimation and dynamic adaptivity for symmetric
  discontinuous {G}alerkin approximations of reactive transport problems.
\newblock {\em Comput. Methods Appl. Mech. Engrg.}, 195(7-8):632--652, 2006.

\bibitem{Tabata2002}
Masahisa Tabata and Atsushi Suzuki.
\newblock Mathematical modeling and numerical simulation of {E}arth's mantle
  convection.
\newblock In {\em Mathematical modeling and numerical simulation in continuum
  mechanics ({Y}amaguchi, 2000)}, volume~19 of {\em Lect. Notes Comput. Sci.
  Eng.}, pages 219--231. Springer, Berlin, 2002.

\bibitem{Tackley2008}
Paul~J Tackley.
\newblock Modelling compressible mantle convection with large viscosity
  contrasts in a three-dimensional spherical shell using the yin-yang grid.
\newblock {\em Phys. Earth Planet. Inter.}, 171(1):7--18, 2008.

\bibitem{TackleyStevensonGlatzmaierSchubert1993}
Paul~J Tackley, David~J Stevenson, Gary~A Glatzmaier, and Gerald Schubert.
\newblock Effects of an endothermic phase transition at 670 km depth in a
  spherical model of convection in the earth's mantle.
\newblock {\em Nature}, 361(6414):699--704, 1993.

\bibitem{temam1977navier}
Roger Temam.
\newblock {\em Navier-{S}tokes equations. {T}heory and numerical analysis}.
\newblock North-Holland Publishing Co., Amsterdam-New York-Oxford, 1977.
\newblock Studies in Mathematics and its Applications, Vol. 2.

\bibitem{Torrance1971}
KE~Torrance and DL~Turcotte.
\newblock Thermal convection with large viscosity variations.
\newblock {\em J. Fluid Mech.}, 47(01):113--125, 1971.

\bibitem{Verfurth1984error}
R.~Verf{\"u}rth.
\newblock Error estimates for a mixed finite element approximation of the
  {S}tokes equations.
\newblock {\em RAIRO Anal. Num\'er.}, 18(2):175--182, 1984.

\bibitem{verfurth1998posteriori}
R.~Verf{\"u}rth.
\newblock A posteriori error estimators for convection-diffusion equations.
\newblock {\em Numer. Math.}, 80(4):641--663, 1998.

\bibitem{verfurth2005robust}
R.~Verf{\"u}rth.
\newblock Robust a posteriori error estimates for nonstationary
  convection-diffusion equations.
\newblock {\em SIAM J. Numer. Anal.}, 43(4):1783--1802 (electronic), 2005.

\bibitem{Verfurth2005}
R.~Verf{\"u}rth.
\newblock Robust a posteriori error estimates for stationary
  convection-diffusion equations.
\newblock {\em SIAM J. Numer. Anal.}, 43(4):1766--1782 (electronic), 2005.

\bibitem{vonneumann1950method}
J.~Von~Neumann and R.~D. Richtmyer.
\newblock A method for the numerical calculation of hydrodynamic shocks.
\newblock {\em J. Appl. Phys.}, 21:232--237, 1950.

\bibitem{Young1974}
Richard~E Young.
\newblock Finite-amplitude thermal convection in a spherical shell.
\newblock {\em J. Fluid Mech.}, 63(04):695--721, 1974.

\bibitem{zhang2010positivity}
Xiangxiong Zhang and Chi-Wang Shu.
\newblock On positivity-preserving high order discontinuous {G}alerkin schemes
  for compressible {E}uler equations on rectangular meshes.
\newblock {\em J. Comput. Phys.}, 229(23):8918--8934, 2010.

\bibitem{zhang2013maximum}
Yifan Zhang, Xiangxiong Zhang, and Chi-Wang Shu.
\newblock Maximum-principle-satisfying second order discontinuous {G}alerkin
  schemes for convection-diffusion equations on triangular meshes.
\newblock {\em J. Comput. Phys.}, 234:295--316, 2013.

\bibitem{zhu2011robust}
Liang Zhu and Dominik Sch{\"o}tzau.
\newblock A robust {\it a posteriori} error estimate for {$hp$}-adaptive {DG}
  methods for convection-diffusion equations.
\newblock {\em IMA J. Numer. Anal.}, 31(3):971--1005, 2011.

\end{thebibliography}
\end{document}